\documentclass[12pt]{amsart}
\newcommand{\PP}{\sf P}
\usepackage{amsthm}
\usepackage{comment}
\usepackage[latin1]{inputenc}
\usepackage{subfigure}
\usepackage{color}
\usepackage{amsmath}
\usepackage{amsthm}
\usepackage{amstext}
\usepackage{amssymb}
\usepackage{amsfonts}
\usepackage{graphicx}
\usepackage{young}
\usepackage{multicol}
\usepackage{mathrsfs}
\usepackage[all]{xy}
\usepackage{tikz}
\usepackage{mathtools}
\usepackage{tabularx}
\usepackage{array}
\usepackage{commath}
\usepackage{ytableau}
\usepackage[colorinlistoftodos]{todonotes}

\usepackage{enumerate}

\DeclareMathOperator{\rep}{\mathrm{rep}}

\DeclareMathOperator{\Hom}{\mathrm{Hom}}

\DeclareMathOperator{\NEnd}{\mathrm{NEnd}}

\theoremstyle{plain}
\theoremstyle{definition}
\newtheorem{theorem}{Theorem}[section]

\newtheorem{remark}[theorem]{Remark}
\newtheorem{lemma}[theorem]{Lemma}

\newtheorem{question}[theorem]{Question}
\newtheorem{example}[theorem]{Example}
\newtheorem{proposition}[theorem]{Proposition}
\newtheorem{corollary}[theorem]{Corollary}

\DeclareMathAlphabet{\mathpzc}{OT1}{pzc}{m}{it}

\newcommand{\diff}{\operatorname{diff}}
\usepackage{amssymb,amsmath,tabularx,graphicx}
\usepackage{tikz}
\usetikzlibrary[calc,intersections,through,backgrounds,arrows,decorations.pathmorphing]

\def\pr{\prime}

\newcommand{\op}{\textup{op}}

\newcommand{\NN}{\mathbb N\mskip-5mu \reflectbox{$\mathbb N$}}

\newcommand{\End}{\operatorname{End}}

\newcommand{\bd}{\textbf{d}}
\newcommand{\bl}{{\boldsymbol\lambda}}
\newcommand{\blambda}{\bl}
\newcommand{\bmu}{{\boldsymbol\mu}}
\newcommand{\bnu}{\boldsymbol\nu}
\newcommand{\JF}{\operatorname{JF}}
\newcommand{\fk}{\Bbbk}
\newcommand{\sP}{\textsf P}
\newcommand{\pos}{\textsf{P}_{{Q,m}}}
\newcommand{\posprime}{\textsf{P}_{{Q',m}}}
\newcommand{\tpos}{\tilde{\textsf{P}}_{{Q,m}}^i(M)}

\def\RSK{\rho_{Q,m}(M)}
\newcommand{\cat}{\mathcal{C}_{Q,m}}
\newcommand{\pro}{\operatorname{pro}}

\newcommand{\adj}{{\textup{adj}(k)}}
\newcommand{\repi}{\rep_{\textup{inj}}}
\newcommand{\reps}{\rep_{\textup{surj}}}
\newcommand{\trepi}{\widetilde{\operatorname{rep}}_\textup{inj}}
\newcommand{\treps}
{\widetilde{\operatorname{rep}}_\textup{surj}}
\newcommand{\refl}{\sigma}

\newcommand{\JFg}{\operatorname{GenJF}}
\newcommand{\GR}{\operatorname{GenRep}}
\newcommand{\DR}{\widetilde R}

\renewcommand{\P}{{\sf P}}
\newcommand{\J}{\sf J}
\newcommand{\OO}{\sf O}
\newcommand{\x}{{\sf x}}
\newcommand{\y}{{\sf y}}
\newcommand{\z}{{\sf z}}
\newcommand{\even}{\textup{even}}
\newcommand{\odd}{\textup{odd}}
\newcommand{\Ant}{\sf Ant}
\newcommand{\RPP}{\it RPP}

\title{Minuscule reverse plane partitions via quiver representations}
\date{}
\author{Alexander Garver}
\email{alexander.garver@gmail.com}

\author{Rebecca Patrias}
\email{rebecca.patrias@stthomas.edu}
\address{Department of Mathematics\\ University of St. Thomas}

\author{Hugh Thomas}
\email{thomas.hugh\_r@uqam.ca}
\address{Laboratoire de Combinatoire et d'Informatique Math\'ematique,
Universit\'e du Qu\'ebec \`a Mont\-r\'eal}

\begin{document}

\begin{abstract} {
A nilpotent endomorphism of a quiver representation induces a linear
transformation on the vector space at each vertex.  Generically among all
nilpotent endomorphisms, there is a well-defined Jordan form for these
linear transformations, which is an interesting new invariant of a
quiver representation.  If $Q$ is a Dynkin quiver and $m$ is a minuscule vertex,
we show that representations consisting of direct sums of indecomposable
representations all including $m$ in their support, the category of which we denote by
$\mathcal C_{Q,m}$, are determined up to
isomorphism by this invariant.  We use this invariant to define a bijection
from isomorphism classes of representations in $\mathcal C_{Q,m}$ to
reverse plane partitions whose shape is the minuscule poset corresponding to
$Q$ and $m$.  By relating the piecewise-linear promotion action on reverse plane partitions to
Auslander--Reiten translation in the derived category, we give a uniform
proof that the order of promotion equals the Coxeter number.  In type
$A_n$, we show that special cases of our bijection include the
Robinson--Schensted--Knuth
and Hillman--Grassl correspondences.}

\end{abstract}

\maketitle

\tableofcontents

\section{Introduction}

\subsection{Recovering a representation from information about its generic nilpotent endomorphisms}

Let $Q$ be a quiver with $n$ vertices numbered 1 to $n$.  Let $\bd=(d_1,\dots,d_n)$ be an
$n$-tuple of non-negative integers.  
Let $X$ be a 
representation of $Q$ with dimension vector $\bd$, over an algebraically closed ground field $\fk$.  Let $\phi$ be a nilpotent endomorphism of $X$.  At each
vertex $i$ of $Q$, the endomorphism $\phi$ induces an endomorphism $\phi_i$ of $X_i$.  We can consider the
Jordan form of each of these vector space endomorphisms,
which gives us a sequence of partitions $\lambda^i \vdash d_i$.  We show that for a generic choice of $\phi$, the 
$n$-tuple $\bl=(\lambda^1,\dots,\lambda^n)$ is well-defined.
We refer to this as the \textit{Jordan form data} of $X$, and we write it 
$\JFg(X)$.  

Note that $\JFg(X)$ is generally not enough information to 
recover $X$.  Consider, for example, $Q$ the quiver of type $A_2$, with $\bd=(1,1)$. There are two non-isomorphic representations of $Q$ with dimension vector $\textbf{d}$, and each representation has Jordan form data equal to $((1),(1))$.

\begin{question}\label{Q_V_from_genJF(V)} For which subcategories $\mathcal C$ of
$\rep Q$ is it the case that if we know that $X\in \mathcal C$, then we can recover $X$ from $\JFg(X)$?\end{question}

We say that such a subcategory is \textit{Jordan recoverable}.  (When we refer to subcategories, we always mean full subcategories  closed under direct sums and direct summands.)  Clearly, any subcategory with the property that the dimension vectors of its indecomposable representations are linearly independent is Jordan recoverable, since for such a subcategory $\mathcal C$, if we know $X\in\mathcal C$, then $X$ can be recovered from its dimension vector $\textbf{dim}(X)$. However, there are more interesting examples.

\begin{example}\label{ex:intro1}
The following is a non-trivial example of Jordan recoverability that we will use as a running example throughout this section. Let $\mathcal{C}_{Q,2}$ denote the subcategory of representations $X$ of $Q=1\rightarrow 2 \leftarrow 3$ such that each indecomposable summand of $X$ has support over vertex 2 of $Q$. By identifying indecomposable representations of $Q$ with their dimension vectors, each $X\in \mathcal{C}_{Q,2}$ is isomorphic to $010^a\oplus 011^b\oplus 110^c\oplus 111^d$ for some $a,b,c,d\in\mathbb{N}$. (Here, and throughout the paper, we write $\mathbb N$ for the non-negative integers.) Either by direct calculation, or by using results from Section~\ref{sec:RSK}, we see that $$\JFg(X)=((c+d),(\max(b,c)+a+d, \min(b,c)), (b+d)).$$ Given this Jordan form data, we can recover $X$ up to isomorphism. Concretely, this amounts to saying that if we know $c+d$, $\max(b,c)+a+d$, $\min(b,c)$, and $b+d$, then we can recover $a$, $b$, $c$, and $d$, which is easily verified.
\end{example}

One strategy for reconstructing $X$ from $\JFg(X)$ is the following.  Suppose we are given an $n$-tuple of partitions $\bl=(\lambda^i)$. Given this information, define the $n$-tuple $\bd=(d_i)$ by $d_i=|\lambda^i|$. 
Let $W_i$ be a vector space of dimension $d_i$, and fix $\phi_i$ a nilpotent linear operator on $W_i$ with Jordan block sizes given by $\lambda^i$. We write $\rep_\bl(Q)$ for the representations whose vector spaces are $W_i$ and such that $\phi=(\phi_i)$ 
defines an endomorphism of the representation.  
We show that $\rep_\bl(Q)$ is an irreducible variety.  It turns out that there is a dense open set $U\subset \rep_\bl(Q)$ such that for any representation in $U$, the dimension vectors of the indecomposable summands are well-defined.   As we explain, this is a 
generalization of Kac's well-known canonical decomposition of dimension vectors.  Under good circumstances (for example, if $Q$ is Dynkin), 
this implies that all the representations in $U$ are actually 
isomorphic.  

We say that a subcategory $\mathcal C$ of $\rep Q$ is 
\emph{canonically Jordan recoverable} if, for any $X\in\mathcal C$, there is a dense open set $U\subset \rep_{\JFg(X)}(Q)$ such that the representations at all points in $U$ are isomorphic to $X$.  

Our first main result is a non-trivial example of canonical Jordan recoverability.  For $i$ a vertex of $Q$, let $\mathcal{C}_{Q,i}$ be the subcategory of $\rep Q$ consisting of direct sums of indecomposable representations all of which have $i$ in their support.  

\begin{theorem}\label{Thm_can_jordan_recov} If $Q$ is a Dynkin quiver and $m$ is a minuscule vertex of $Q$, then
$\cat$ is canonically Jordan recoverable. \end{theorem}

For a conceptual definition of what it means for a vertex to be minuscule, see Section \ref{min-poset-def}.  In type $A_n$, all vertices are minuscule; in type $D_n$, the minuscule vertices are the vertices of degree 1; in types $E_6$ and $E_7$, a subset of the vertices of degree 1 are minuscule, while in type $E_8$ no vertices are minuscule.

\begin{example}\label{ex:intro2}

Example~\ref{ex:intro1} is an instance of Theorem~\ref{Thm_can_jordan_recov}. Using the quiver from that example, suppose we start with the representation $X=010\oplus 011\oplus 110$.  In terms of that example, we have $a=b=c=1$ and $d=0$.  According to that example, or by direct computation, we determine that the Jordan form data for this representation is $((1),(2,1),(1))$.  Now suppose we want to recover $X$ from its Jordan form data.  

We carry out the above procedure: we start with 
1-dimensional $W_1$, 3-dimensional $W_2$, and 1-dimensional $W_3$, and on $W_i$ we have a linear transformation $\phi_i$, with $\phi_1=0=\phi_3$, and with $\phi_2$ having two Jordan blocks of sizes 1 and 2.  Let $f_{1,2}$ be a generic linear map from $W_1$ to $W_2$ that is compatible with $\phi_1$ and $\phi_2$, i.e., such that $f_{1,2}\phi_1=\phi_2f_{1,2}$.  This holds if and only if the image of $f_{1,2}$ lies in the kernel of $\phi_2$.  The same analysis applies to $f_{3,2}$, a generic linear map from $W_3$ to $W_2$.  Let $w_1\neq 0 \in W_1$ and $w_3\neq 0 \in W_3$. For generic choices of $f_{1,2}$ and 
$f_{3,2}$, we have that $f_{1,2}(w_1)$ and $f_{3,2}(w_3)$ are linearly independent in the kernel of $\phi_2$, and we may thus extend the pair to a basis of $W_2$, $\{f_{1,2}(w_1),f_{3,2}(w_3), w_2\}$. One then checks that a representation of $Q$ with vector spaces $W_1,W_2,W_3$ and such linear maps $f_{1,2}$ and $f_{3,2}$ is isomorphic to  $010\oplus 011\oplus 110$, so we have canonically recovered the isomorphism class of $X$ from its Jordan form data.  

Alternatively, suppose that our starting representation had been
$X'=010^2 \oplus 111$.  Its Jordan form data is $\bl'=((1),(3),(1))$.  So, to recover $X'$, we start with vector spaces $W'_1,W'_2$, and $W'_3$. On $W'_i$ we have nilpotent maps $\phi'_i$ with $\phi'_1=\phi'_3=0$ and $\phi_2'$ having one Jordan block of size 3. This time, for $w'_1\neq 0\in W'_1$, $w'_3\neq0 \in W'_3$, and generic maps $f_{1,2}', f_{3,2}'$, we see that $f_{1,2}'(w'_1)$ and $f_{3,2}'(w'_3)$ are not linearly independent as they must both lie in the 1-dimensional kernel of $\phi_2'$. Such a representation of $Q$ is isomorphic to $010^2\oplus111$, so we have recovered $X'$.  \end{example}

\begin{example}
Let us consider instead the case of $\mathcal C_{Q,m}$ for $m$ a non-minuscule vertex.  Let $Q$ be the quiver of type 
$D_4$ shown in Figure~\ref{fig:D4Quiver}. 
The representations $1100 \oplus 1011$, $1010\oplus 1101$, and $1001\oplus 1110$ in $\mathcal C_{Q,1}$ all have Jordan data $((1,1),(1),(1),(1))$, so $\mathcal C_{Q,1}$ is not Jordan recoverable.
\end{example}
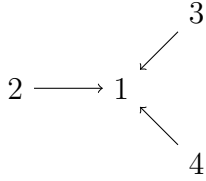
\begin{figure}
\begin{tikzpicture}
\node (1) at (0,0) {1};
\node (2) at (-1.41,0) {2};
\node (3) at (1,1) {3};
\node (4) at (1,-1) {4};
\draw[->] (2)--(1);
\draw[->] (3)--(1);
\draw[->] (4)--(1);
\end{tikzpicture}
\caption{A type $D_4$ quiver.}
\label{fig:D4Quiver}
\end{figure}

\subsection{Structure of Jordan form data in the minuscule case}
In light of Theorem \ref{Thm_can_jordan_recov}, which says that the Jordan form data of a representation in $\cat$ 
characterizes the representation up to isomorphism, it is particularly natural to 
ask what can be said about the Jordan form data.  It turns out to have a
very particular form.  In order to describe it, we need to introduce
some further notation.  Associated to $Q$ and the minuscule vertex $m$,
there is a \emph{minuscule poset} $\pos$, whose definition we defer to Section \ref{min-poset-def}.
The minuscule poset $\pos$ is equipped with a 
map $\pi$ to the vertices of $Q$.  This map has, in particular, the property that each fibre $\pi^{-1}(j)$ is totally ordered.

\begin{theorem}  \label{th-rpp}
\begin{enumerate} 
\item Let $X\in \cat$.  The number of parts in the partition $\JFg(X)^j$ is less than or equal to the size of the fibre $\pi^{-1}(j)$.
\item For $X\in \cat$, define a map $\rho_{Q,m}(X):\pos \rightarrow \mathbb{N}$ as follows: The values of $\rho_{Q,m}(X)$ restricted to $\pi^{-1}(j)$ are the entries 
of $\JFg(X)^j$, padded with extra zeros if necessary, and ordered so that, restricted to $\pi^{-1}(j)$, the function is order-reversing.
Then $\rho_{Q,m}(X)$ is order-reversing as a map from $\pos$ to $\mathbb{N}$.
\item The map from isomorphism classes in $\cat$ to 
order-reversing maps from $\pos$ to $\mathbb{N}$, sending
$X$ to $\rho_{Q,m}(X)$, is a bijection.
\end{enumerate}
\end{theorem}

The bijection described in part (3) of Theorem~\ref{th-rpp} corresponding to our running example is shown in Figure~\ref{fig:introRPP}. The arrows of the Hasse quiver of this poset  are pointing left-to-right to indicate that larger elements of the poset appear further to the right. We will consistently use this unusual convention for compatibility with the conventional way to draw Auslander--Reiten quivers.

\begin{figure}
\[\raisebox{.65in}{$010^a\oplus 011^b\oplus 110^c\oplus 111^d \hspace{.2in}\longleftrightarrow$\hspace{.2in}} 
\begin{tikzpicture}[scale=1.5]
\node (1) at (1,0) {{$b+d$}};
\node (3) at (2,1) {{$\min(b,c)$}};
\node (2) at (0,1) {{$\max(b,c)+a+d$}};
\node (4) at (1,2) {{$c+d$}};
\draw[->] (1)--(3);
\draw[->] (2)--(1);
\draw[->] (2)--(4);
\draw[->] (4)--(3);
\end{tikzpicture}\]
\caption{The correspondence between isomorphism classes of representations of $Q=1\rightarrow 2 \leftarrow 3$ belonging to $\mathcal{C}_{Q,2}$ and order-reversing maps from $\textsf{P}_{Q,2}$  to $\mathbb{N}$.}
\label{fig:introRPP}
\end{figure}
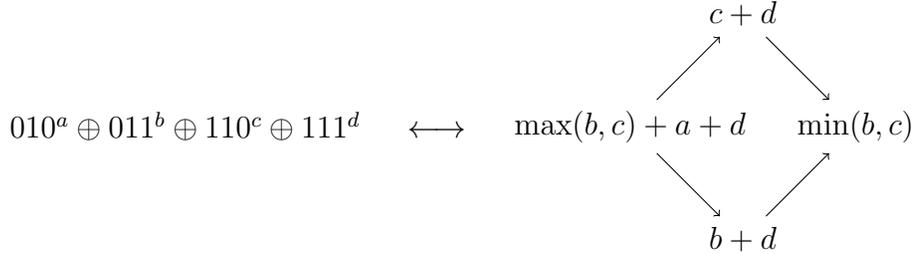

We call the order-reversing maps $\rho_{Q,m}(X)$ appearing in Theorem~\ref{th-rpp} \textit{reverse plane partitions} of the corresponding poset  $\pos$. We denote the collection of all reverse plane partitions of $\pos$ by $\RPP(\pos)$. From our proof of Theorem~\ref{th-rpp}, we also obtain a combinatorial algorithm for calculating $\rho_{Q,m}(X)$ from the multiplicities of the indecomposable summands of $X$; see Theorem~\ref{thm_4_7_alg} for the precise statement.
The proof of Theorem \ref{th-rpp} depends on the 
combinatorics of minuscule
posets. See \cite{green2013combinatorics} for a thorough introduction.  

\subsection{Reverse plane partitions for objects in the root category}

If we like, we can think of $\cat\subset \rep(Q)$ as being 
contained in the bounded derived category $D^b(Q)$.  It turns 
out that it is possible to define a more general reverse plane
partition which records, not only the structure of $X\in\cat$, but
also some information about the choice of an abelian subcategory of $D^b(Q)$ derived equivalent to $\rep Q$.

Define the orbit category $\mathcal R_Q= D^b(Q)/[2]$.  {This category was originally studied by Happel \cite{happel1985tilting}. It is called the \textit{root category} since its indecomposable objects are in bijection with the roots of the root system associated with $Q$. The root category is a triangulated category whenever $Q$ is acyclic \cite{peng1997root}.}  Every object in $\mathcal R_Q$
can be written as $X\simeq X^\even \oplus X^\odd [1]$, with $X^\even,X^\odd \in \rep(Q)$.

Let $\Xi$ be another quiver with the same underlying graph as $Q$.
The root categories of $Q$ and $\Xi$ are equivalent.  Fix an equivalence (subject to some natural technical conditions which we defer).  
Suppose that we have an object $X\in \cat\subset \mathcal R_Q$.  Using the equivalence of $\mathcal R_Q$ with $\mathcal R_\Xi$, we write 
$X\simeq X^\even_\Xi \oplus X^\odd_\Xi[1]$, with $X^\even_\Xi, X^\odd_\Xi \in \rep(\Xi)$.  

A generic nilpotent endomorphism of $X$ induces a generic nilpotent endomorphism of $X_\Xi^\even$ and $X_\Xi^\odd$, so it is natural to consider their Jordan form data.
Let $\JFg(X^\even_\Xi)=\bl$, and $\JFg(X^\odd_\Xi)=\bmu$.  We want to fit the entries of $\bl$ and $\bmu$ into
$\pos$ to form a reverse plane partition, putting the entries of $\bl$ into an order filter in $\pos$, and putting the entries of $\bmu$ into the complementary order ideal.  In order for this to have a hope
of defining a reverse plane partition, the entries of $\bmu$ would 
have to be all larger than the entries of $\bl$.  To ensure that
this is the case, we consider reverse plane partitions with
entries in the set $\NN=\{0,1,2,\dots,\infty-2,\infty-1,\infty\}$ (with the obvious total order).\footnote{ We recommend the pronunciation ``en-ne'' for $\NN$.}
Each part $j$ in $\bmu$ is entered into the reverse plane partition
as $\infty-j$.  

\begin{proposition} For any $X\in\cat$ and any $\Xi$, it is possible to carry out the above procedure, defining
a reverse plane partition on $\pos$ with entries in $\NN$, 
which we denote $\rho_{Q,m}^\Xi(X)$. \end{proposition}

(See Proposition \ref{derived-fits} for a more precise statement.) We also have the following converse.

\begin{proposition}\label{intro-enough-X}
Given any reverse plane partition $\rho$ on $\pos$ with entries in $\NN$, there exists $\Xi$ derived equivalent to $Q$ and an $X\in\cat$ such that $\rho=\rho_{Q,m}^\Xi(X)$.
\end{proposition} 

\subsection{Periodicity of toggling}
Fix $N\in\mathbb{N}$. Let $\P$ be a poset and $\rho: \textsf{P} \to [0,N]$ be a reverse plane partition.  
For $\textsf{x}$ an element of $\textsf{P}$, we define the \textit{toggle} of $\rho$ at $\textsf{x} \in \textsf{P}$ by
$$\begin{array}{rcl}
t_\textsf{x}\rho(\textsf{y}) & = & \left\{\begin{array}{lcl} \displaystyle \max_{\textsf{y} \lessdot \textsf{y}_1}\rho(\textsf{y}_1) +\min_{\textsf{y}_2\lessdot \textsf{y}}\rho(\textsf{y}_2)-\rho(\textsf{y}) & : & \text{if $\textsf{y} = \textsf{x}$} \\ \rho(\textsf{y}) & : & \text{if $\textsf{y} \neq \textsf{x}$,}\end{array}\right.
\end{array}$$
where $\textsf{y}$ is any element of $\textsf{P}$.  If $\textsf{y}$ is maximal,
we interpret
$\max_{\textsf{y} \lessdot \textsf{y}_1}\rho(\textsf{y}_1)$ as 0, and if $\textsf{y}$ is minimal, we 
interpret $\min_{\textsf{y}_2\lessdot \textsf{y}}\rho(\textsf{y}_2)$ as $N$.
Since $\rho$ is a reverse plane partition, so is $t_\textsf{x} \rho.$ Additionally, observe that $t_\textsf{x}\circ t_\textsf{x}(\rho) = \rho$. This is the piecewise-linear toggle operation considered by Einstein and Propp (up to rescaling, and restricted to lattice points) \cite{einstein2013combinatorial}.
 
Note that $t_{\sf x}$ and $t_{\sf y}$ commute unless $\sf x$ and $\sf y$ are related by a cover.  
For $\pos$, and $i$ a vertex of $Q$, the elements of $\pi^{-1}(i)$ are never related by a cover, so we can define $t_i$ as the composition of all $t_{\sf x}$ where ${\sf x} \in \pi^{-1}(i)$, without worrying about the order in which the composition is taken.

{Number the vertices of $Q$ in such a way that if there is an arrow from $j$ to $i$ then $j < i$.}  Define $\pro_Q=t_n\circ \dots \circ t_1$.
Define $h$ to be the Coxeter number of $Q$: by definition, this is the order of the product of the simple reflections in the Coxeter group, or, equivalently, the largest degree of the root system. We will prove the following theorem.

\begin{theorem}\label{intro-period-N} For $Q$ a Dynkin quiver and $m$ a minuscule vertex, $\pro_Q^h$ is the identity transformation on reverse plane partitions on
$\pos$ with entries in $[0,N]$.
\end{theorem}

Theorem~\ref{intro-period-N} has already been established in type $A$. We refer the reader to Remark~\ref{rem_prev_periodicity_results} for a more detailed explanation of its overlap with existing results. 											
Note that the definition of toggling easily extends to reverse plane partitions with entries in $\NN$.
For $i\leq k$ in $\mathbb{N}$, the operation sending 
$j$ to $i+k-j$ for $i\leq j\leq k$ can be understood as mapping the interval from $i$ to $k$ back to itself while reversing the order. The same idea applies to $i\leq j\leq k \in \NN$.
See Figure~\ref{fig:toggle} for an example.

\subsection{Toggling and the root category}

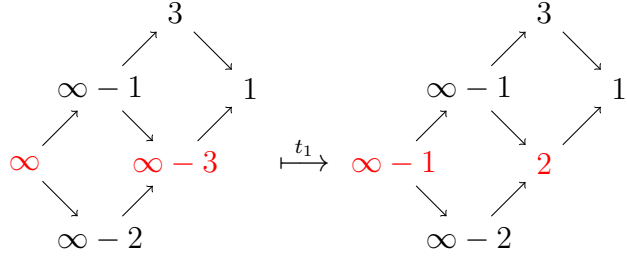
\begin{figure}

\begin{tikzpicture}
\node (a) at (0,0) {\fbox{$\infty$}};
\node (b) at (1,1) {$\infty-1$};
\node (c) at (1,-1) {$\infty-2$};
\node (d) at (2,2) {3};
\node (e) at (2,0) {\fbox{$\infty-3$}};
\node (f) at (3,1) {1};
\draw[->] (a)--(b);
\draw[->] (b)--(d); 
\draw[->] (d)--(f);
\draw[->] (a)--(c);
\draw[->] (b)--(e);
\draw[->] (e)--(f);
\draw[->] (c)--(e);
\end{tikzpicture}
\raisebox{.47in}{$\stackrel{t_1}{\longmapsto}$}
\begin{tikzpicture}
\node (a) at (0,0) {\fbox{$\infty-1$}};
\node (b) at (1,1) {$\infty-1$};
\node (c) at (1,-1) {$\infty-2$};
\node (d) at (2,2) {3};
\node (e) at (2,0) {\fbox{$2$}};
\node (f) at (3,1) {1};
\draw[->] (a)--(b);
\draw[->] (b)--(d);
\draw[->] (d)--(f);
\draw[->] (a)--(c);
\draw[->] (b)--(e);
\draw[->] (e)--(f);
\draw[->] (c)--(e);
\end{tikzpicture}
\caption{We apply $t_1$ to a reverse plane of shape $\textsf{P}_{Q,1}$, where $Q=2\rightarrow 1 \leftarrow 3 \leftarrow 4$. The values corresponding to the elements of $\pi^{-1}(1)$ are shown in boxes.}
\label{fig:toggle}
\end{figure}

{ Suppose $\Xi$ is a reorientation of $Q$, and as above, 
choose an identification of
$\rep \Xi$ as a subcategory of $\mathcal R_Q$.  Let $i$ be a source of $\Xi$.  Let
$\Xi'$ be obtained from $\Xi$ by reversing all arrows incident to vertex $\Xi$.  (We write
$\Xi'=\sigma_i(\Xi)$.))  It is possible to choose an identification of $\rep \Xi'$ as a subcategory of $\mathcal R_Q$ such that the indecomposable objects of $\rep \Xi$ and of $\rep \Xi'$ coincide except for the simple projective representation at vertex $i$ of $\rep \Xi'$ and the simple injective representation at vertex $i$ of $\rep \Xi$.

Given an object $X\in \cat$, we want to describe the relationship between $\rho_{Q,m}^\Xi(X)$ and $\rho_{Q,m}^{\Xi'}(X)$.  The relationship turns out to be very simple:

\begin{theorem}\label{tog-ref} Let $X\in \cat \subset \mathcal R_Q$. {Let $\Xi$ be a quiver with the same underlying graph as $Q$, let $i$ be a source of $\Xi$, and let $\Xi'=\sigma_i(\Xi)$, and let $\rep \Xi$ and $\rep \Xi'$ be identified with subcategories of $\mathcal R_Q$ as above. Then} $\rho_{Q,m}^{\Xi'}(X) = t_i\rho_{Q,m}^\Xi(X)$.\end{theorem}

Now suppose that we label the vertices of $\Xi$ in such a way that if there is an arrow from $j$ to $i$ then $j < i$.  It follows that the first vertex is a source.  After reversing all arrows incident with 1, the second vertex will be a source, and so forth. The effect is that we can consider applying the above theorem successively at vertex 1, vertex 2, and so on, up to vertex $n$.  The quiver that results from reversing the arrows at each vertex is isomorphic to the original quiver, but the final term in the sequence of subcategories $\rep \Xi$, $\rep \Xi'$, \ldots is not $\rep \Xi$ again, though it is equivalent to $\rep \Xi$: in fact, it is $\tau \rep \Xi$. Throughout this paper, $\tau$ denotes the \textit{Auslander--Reitan translation}. Abusing notation, we write $\rho_{Q,m}^{\tau\Xi}(X)$ for the reverse plane partition associated to the splitting of $\mathcal R_Q$ into $\tau \rep \Xi$ and $\tau \rep \Xi[1]$.  Applying the previous theorem successively at vertex $1, 2, \ldots, n$ yields the corollary below.

\begin{corollary}\label{pro_-1} {Let $X\in \cat \subset \mathcal R_Q$.  Let $\Xi$ be a quiver with the same underlying graph as $Q$. Then $\rho_{Q,m}^{\tau\Xi}(X) = \pro_Q\rho_{Q,m}^{\Xi}(X)$.}
\end{corollary}

This is useful because the order of $\tau$, as it acts on isomorphism classes of objects in $\mathcal R_Q$, is known to be $h$. Applying promotion $h$ times therefore corresponds to applying $\tau$ to $\rep \Xi$ that many times, which has the effect of doing nothing. 
The theorem below follows almost immediately: 

\begin{theorem}\label{thm_proh_is_Id}
Considering $\pro_Q$ as a permutation of reverse plane partitions on $\pos$ with entries in $\NN$, its order is $h$.  
\end{theorem}

Finally, from this, we deduce the more conventional periodicity result, Theorem \ref{intro-period-N} stated above.
}

\subsection{Combinatorial applications}
As an enumerative corollary of Theorem~\ref{th-rpp}, reverse plane partitions for minuscule posets have the following beautiful generating function, originally established by Proctor in \cite{proctor1984bruhat}. 
\begin{corollary}\label{thm_gen_fctn} For $Q$ a Dynkin quiver and $m$ a minuscule vertex, we have
\[\sum_{\rho\in \RPP(\pos)} \prod_{i=1}^n q_i^{|\rho_{i}|} = \sum_{X\in \cat} \prod_{i=1}^n q_i^{\textbf{dim}(X)_i}=\prod_{\textsf{u}\in \pos} \frac{1}{1- \prod_{i=1}^nq_i^{\textbf{dim}(M_\textsf{u})_i}},\]
where we write $|\rho_{i}|$ for the sum of the values $\rho(\textsf{x})$ over all $\textsf{x} \in \pi^{-1}(i)$. The second sum is over isomorphism classes of representations in $\cat$. In the third sum, $M_\textsf{u}\in\cat$ is the indecomposable representation of $Q$ corresponding to $\textsf{u}\in \pos$. 
 \end{corollary}

\begin{proof}The first equality is from Theorem \ref{th-rpp}, while the second comes from the fact that any representation in $\cat$ can be decomposed in a unique way as a sum of some number of copies of the representations $M_\textsf{u}$ for $\textsf{u}\in \pos$.
\end{proof}

In fact, our techniques allow us to prove a similar statement where reverse plane partitions on the minuscule poset $\pos$ are replaced by reverse plane partitions on any order filter of $\pos$.  Except in type $A_n$, this result seems to be new. The type $A_n$ result can be deduced from \cite[Corollary 5.2]{gansner1981hillman} or \cite[Corollary 15]{hopkins2014rsk}. See Corollary~\ref{cor_ord_filt_gn_fctn} for our precise statement.

Another combinatorial application is that we can consider the map $\rho_{Q,m}:\mathcal{C}_{Q,m}\to\RPP(\pos)$ from part (2) of Theorem~\ref{th-rpp} to be a generalization of the classical Robinson--Schensted--Knuth (RSK) correspondence. Indeed, we will see in Proposition~\ref{prop:GKinv} that in type $A$, this map has the same Greene--Kleitman invariants \cite{greene1976structure} as RSK. See, for example, \cite{pak2001hook}, \cite{garver2017greene}, and \cite{hopkins2014rsk} for specific instances of the type $A$ map described explicitly in terms of RSK.

\subsection{Connections to previous work}{\ }

\subsubsection{Generalizing Robinson--Schensted--Knuth} Another generalization of the RSK correspondence in the literature is due to Berenstein--Zelevinsky \cite{BZ}. The map in \cite[Theorem 3.7]{BZ} is a bijection converting Lusztig data for a canonical basis element into string cone data for the corresponding element. 
 It would be interesting to understand the relationship between 
this map and 
our bijection from arbitrary fillings of $\pos$ to reverse plane
partitions on $\pos$.
See in particular \cite[Remark 2.13]{BZ}, which
asserts that in  the type $A_n$ minuscule case, their bijection amounts to the Robinson--Schensted--Knuth correspondence. 

\subsubsection{Scrambled RSK}
In February 2020, Duncan Dauvergne posted a preprint \cite{DD} in which he rediscovered the different versions of type $A_n$ RSK which we define corresponding to different orientations of the $A_n$ quiver, under the name of ``scrambled RSK''.

\subsubsection{Invariant subspaces of nilpotent linear operators}
Ringel and Schmidmeier \cite{ringel2008invariant} considered a problem which is similar in spirit to the setting in which we work.  
They focused on the problem of classifying all triples $(U,V,T)$ where $V$ is a finite-dimensional $\Bbbk$-vector space, $T:V\rightarrow V$ is a linear operator with $T^n=0$, and $U$ is a $T$-invariant subspace of $V$.  This is equivalent to representations of the $A_2$ quiver for which the linear map corresponding to the arrow is an injection, together with an endomorphism $T$ of the representation satisfying $T^n=0$.  
For $n < 6$, they show that this category has only finitely many indecomposable representations. For $n = 6$, there are infinitely many, and they present a complete classification of these indecomposables. For $n > 6$, they show that this category is of wild representation type and thus no such classification is feasible.  Because we only focus on the Jordan form of the nilpotent endomorphism rather than remembering the specific choice of endomorphism, the wildness which they observe does not pose a problem for our approach.

\section{Quiver representations}

  In this section, we recall the definition of a quiver and of quiver representations. We recommend \cite{assem2006elements} and \cite{derksenweyman} for further background on this topic.
  We show that the nilpotent endomorphisms of a quiver representation $X$ form an irreducible algebraic
  variety, which allows us to define the notion of a generic property of a nilpotent endomorphism
  (namely, a property that holds on a dense open set of this variety).

  We also prove a strengthening of Kac's canonical decomposition theorem.  Let $Q$ be a quiver without loops,
  and suppose we have a vector space at each vertex.  Kac's theorem says that if we choose the representation
  generically, the dimension vectors of the indecomposable summands are well-defined. We choose, in addition,
  a nilpotent linear transformation acting on the vector space at each vertex, and we demand that the linear
  maps be compatible with the nilpotent linear transformations (in the sense that the nilpotent linear
  transformations define an endomorphism of the resulting representation).
{It turns out that if a representation is chosen generically among those compatible with the given maps, the dimension vectors of its indecomposable summands do not depend on the choice.} 

\subsection{Quivers}
A \textit{quiver} $Q$ is a directed graph. In other words, $Q$ is a 4-tuple $(Q_0,Q_1,s,t)$, where $Q_0$ is a set of \textit{vertices}, $Q_1$ is a set of \textit{arrows}, and $s, t:Q_1 \to Q_0$ are two functions defined so that for every $a \in Q_1$, we have $s(a) \xrightarrow{a} t(a)$. 

A \textit{representation} $V = ((V_i)_{i \in Q_0}, (f_a)_{a \in Q_1})$ of a quiver $Q$  is an assignment of a finite-dimensional $\fk$-vector space $V_i$ to each vertex $i$ and a $\fk$-linear map $f_a: V_{s(a)} \rightarrow V_{t(a)}$ to each arrow $a$ where $\fk$ is a field.  The \textit{dimension vector} of $V$ is the vector $\textbf{dim}(V):=(\dim V_i)_{i\in Q_0}$. {The \textit{dimension} of $V$ is the defined as $\dim V := \sum_{i\in Q_0}\dim V_i$.}

Let $V = ((V_i)_{i \in Q_0}, (f_a)_{a \in Q_1})$ and $W  = ((W_i)_{i \in Q_0}, (g_a)_{a \in Q_1})$ be two representations of a quiver $Q$. A \textit{morphism} $\theta : V \rightarrow W$ consists of a collection of linear maps $\theta_i : V_i \rightarrow W_i$ that are compatible with each of the linear maps in $V$ and $W$.  That is, for each arrow $a \in Q_1$, we have $\theta_{t(a)} \circ f_a = g_a \circ \theta_{s(a)}$. We say that a collection of linear maps $\{\theta_i\}_{i\in Q_0}$ is \textit{compatible} with the representation $V$ when they define a morphism. An \textit{isomorphism} of quiver representations is a morphism $\theta: V \to W$ where $\theta_i$ is a $\fk$-vector space isomorphism for all $i \in Q_0$. 

The representations of a quiver $Q$ along with morphisms between them form an abelian category, denoted by $\rep Q$. The category $\rep Q$ is equivalent to the category of finitely-generated left modules over the path algebra of $Q$. 

Fix $\textbf{d} \in \mathbb N^n$, and consider the representations of $Q$ with dimension vector $\textbf d$.  
Choosing a basis for each of the vector spaces, we can identify the representations of $Q$ with the points of the affine space
$$\rep (Q,\textbf{d}) = \prod_{a \in Q_1} \text{Hom}_\Bbbk(\Bbbk^{\dim V_{s(a)}}, \Bbbk^{\dim V_{t(a)}}) = \prod_{a \in Q_1} \text{Mat}_{{\dim V_{s(a)}\times \dim V_{t(a)}}}(\Bbbk).$$
We refer to $\rep(Q,\textbf d)$ as the \textit{representation space} of representations with dimension vector $\mathbf d$.

The algebraic group $GL(\bd)=\prod_{i\in Q_0} GL(d_i)$
acts on $\rep (Q,\bd)$ by change of basis at each vertex.  
The orbits of this group action are exactly the isomorphism classes of representations of $Q$ with dimension vector $\bd$.

\subsection{Nilpotent endomorphisms of quiver representations}

\newcommand{\rad}{\operatorname{rad}}
Throughout this subsection, we let $A$ denote a finite dimensional basic $\Bbbk$-algebra.
There exists a quiver $Q$ and a set of relations $I$ such that the category of finite-dimensional left $A$-modules is equivalent to the category of representations of $Q$ satisfying the relations in $I$; we freely pass back and forth between these two perspectives. {(See \cite[Chapter II]{assem2006elements} for further details.)}

\begin{lemma}\label{nend-irred}
Let $A$ be a finite-dimensional $\fk$-algebra, for $\fk$ an
algebraically closed field,
and let $X$
be a finite-dimensional left module over $A$.  
Let  
$$ \text{NEnd}(X): = \{N \in \text{End}(X): N^k = 0 \ \text{for some $k \ge 0$}\}.$$ 
Then $\NEnd(X)$ is an irreducible algebraic variety.\end{lemma}

\begin{proof} First, note that $\rad\End(X)$ is a nilpotent ideal.  Thus, $f\in \End(X)$ is nilpotent if and only if its image in  $\End(X)/\rad\End(X)$ is nilpotent.
  On the other hand, $\End(X)/\rad\End(X)$ is semisimple, so it is isomorphic to a product of matrix rings over $\fk$.  The nilpotent elements in a matrix ring form an irreducible variety, as we now explain.  Any nilpotent matrix is conjugate to a strictly upper triangular matrix, {and conversely} {any matrix conjugate to a strictly upper triangular matrix is nilpotent}.  Therefore, let $U(r)$ denote the strictly upper triangular matrices over $\fk$, and consider the map from $GL(r)\times U(r)$ to $GL(r)$ sending $(G,U)$ to $GUG^{-1}$.  Now $GL(r)\times U(r)$ is clearly irreducible, and the image of this map is exactly the nilpotent matrices.  It follows that the variety of nilpotent matrices is also irreducible.
\end{proof}

{Now, let $X$ be a finite-dimensional left module over $A$ and $((X_i)_{i\in Q_0}, (f_a)_{a\in Q_1})$  
the corresponding representation under the equivalence mentioned above.} For $N\in \NEnd(X)$, we write $N_i$ for the induced linear transformation on $X_i$.  We write $\JF(N_i)$ for the \textit{Jordan form} of $N_i$, understood as a partition whose parts are the sizes of the Jordan blocks.  So $\JF(N_i) \vdash d_i$,  where $d_i$ is the dimension of $X_i$. 
We write $\JF(N)$ for the \textit{Jordan form} 
of $N$, i.e., the $n$-tuple $(\JF(N_i))_{i\in Q_0}$.  If $\bmu=(\mu^1,\dots\mu^n)$
is an $n$-tuple of partitions with $\mu^i\vdash d_i$ for each $i \in Q_0$, we 
write $\bmu\vdash \bd$.  So we write $\JF(N)\vdash \bd$.

Let $\gamma$ and $\kappa$ be partitions of $m$. We say that $\gamma\leq\kappa$ in \textit{dominance order} if $\gamma_1+\cdots+\gamma_k\leq \kappa_1+\cdots+\kappa_k$ for each $k\geq 1$, where we add zero parts to $\gamma$ and $\kappa$  as necessary. We extend this definition to $n$-tuples of partitions as follows.
Given $\bmu=(\mu^1,\mu^2,\ldots,\mu^n)$ and $\bl=(\lambda^1,\lambda^2,\ldots,\lambda^n)$ with $\mu^i$ and $\lambda^i$ partitions of $m_i$, we say that $\bl\leq \bmu$ if $\lambda^i\leq\mu^i$ in dominance order for each $i$. 

Given a partition $\gamma=(\gamma_1,\ldots,\gamma_s)$, we define the \textit{length} of $\gamma$ to be the number of parts of $\gamma$. It is denoted $\ell(\gamma)= s$. 
Recall that the \textit{conjugate} partition is $\gamma^t=(\gamma_1^\pr,\ldots,\gamma_{s^\pr}^\pr)$, where $\gamma_k^\pr$ is the number of parts $\gamma_j$ with $\gamma_j\geq k$. It is well known that transposition reverses dominance order: $\gamma\leq\kappa$ if and only if $\kappa^t \leq \gamma^t$. Therefore, by defining the conjugate of an $n$-tuple of partitions by ${\bmu}^t := ((\mu^1)^t,\ldots, (\mu^n)^t)$, we see that ${\bl} \le {\bmu}$ if and only if ${\bmu}^t \le {\bl}^t$.

\begin{lemma}\label{closed} Let $\bmu\vdash \bd$.  
Let $\NEnd_{\leq \bmu}(X)$
be the subset of $\NEnd(X)$ consisting of those nilpotent endomorphisms $N$
such that $\JF(N) \leq \bmu$.
Then $\NEnd_{\leq \bmu}(X)$ is closed in $\NEnd(X)$.
\end{lemma}

\begin{proof}
Instead of considering the conditions that $\JF(N_i) \leq \mu^i$ in dominance order for all $i$,
we consider the equivalent condition of having $\JF(N_i)^t \geq (\mu^i)^t$
in dominance order for all $i$.  

The condition that $\JF(N_i)^t_1 \geq (\mu^i)^t_1$ is precisely the
condition that the rank of $N_i$ be less than or equal to $\dim V_i - (\mu^i)^t_1$.  Similarly, the condition that $(\JF(N_i)^t_2 \geq (\mu^i)^t_2$ is
precisely the condition that the rank of $N_i^2$ be less than or equal to
$\dim V_i - (\mu^i)^t_1 -(\mu^i)^t_2$, and similarly for the other conditions
which need to be checked.  These rank conditions are closed
conditions, which proves the result.\end{proof}

\begin{theorem} Let $A$ be a finite-dimensional $\fk$-algebra, for $\fk$ an
algebraically closed field,
and let $X$
be a finite-dimensional left module over $A$. There is a maximum value of $\JF$ on $\NEnd(X)$, and it is attained on a dense open set of $\NEnd(X)$.
\end{theorem}

\begin{proof} 
Since there are only a finite number of possible Jordan forms $\JF(N)$ for $N\in \NEnd(X)$, 
there must be (at least) one that is attained at a dense set of points,
i.e., at a set of points whose closure is all of $\NEnd(X)$.  Let one such be $\bmu$.
By  Lemma
\ref{closed}, $\NEnd_{\leq \bmu}(X)$,
is a closed set.  Since it includes the dense set where $\JF(N)=\bmu$, it
must be all of $\NEnd(X)$.  Thus the value of $\JF(N)$ is at most $\bmu$
for any $N\in \NEnd(X)$.
Further, by applying Lemma \ref{closed} to each tuple of partitions that is
covered
by $\bmu$, we find that the set of all $N'\in \NEnd(X)$ such that
$\JF(N')$ is strictly less than $\bmu$ is a closed
set.  Thus the set of nilpotent endomorphisms with Jordan form exactly
$\bmu$ is a dense open set.  \end{proof}

The maximum value of $\JF(N)$ for $N\in \NEnd(X)$ will be referred to as the \textit{generic value} of $\JF$ on $\NEnd(X)$, and written $\JFg(X)$.

\subsection{Canonical decompositions}\label{Sec_canon_decomp}

Let $Q$ be a quiver without loops and $\textbf{d}$ be a dimension vector. Kac shows \cite[p. 85]{kac-infinite-i} that there is a decomposition
$$\textbf{d}=\textbf{d}_1+\ldots+\textbf{d}_r$$
and a dense open set in
$\rep (Q,\textbf{d})$ such that all the representations in this dense open set can
be written as
$$M=M_1\oplus \dots \oplus M_r$$
where each $M_i$ is indecomposable and $\textbf{dim}(M_i)=\textbf{d}_i$ for all $i \in \{1,\ldots, r\}$.  Note that different choices of $M$ may lead to non-isomorphic representations $M_i$; all that is determined is their dimension vectors. (If $Q$ is of finite representation type, an indecomposable representation is determined up to isomorphism by its dimension vector, so in fact, all the representations in the dense open set are isomorphic, but this is not the general behaviour.) We will prove that there is a similar decomposition once one demands compatibility
with a nilpotent endomorphism of specified Jordan form.

\begin{theorem}\label{thm:GenericMpi}
Let $Q$ be a quiver and $\textbf{d}=(d_1,\dots,d_n)$ a dimension vector. Let $\blambda \vdash \bd$, and let $N$ be an $n$-tuple of 
linear transformations whose Jordan form is $\blambda$.  
Consider the representations of $Q$ that are compatible with the action of $N$. Then there is a dense open subset of the variety of such representations and a decomposition $\textbf{d}=\textbf{d}_1+\dots+\textbf{d}_r$, such that all the representations in this dense open set can be written as $$M=M_1\oplus\dots\oplus M_r,$$ where each $M_i$ is indecomposable and $\textbf{dim}(M_i)=\textbf{d}_i$ for all $i \in \{1,\ldots, r\}$.
\end{theorem}

\begin{proof}
Let $V_i$ be a vector space  of dimension $|\lambda^i|$ for each $i\in Q_0$.  Choose $N_i$ a linear transformation of $V_i$ with Jordan form $\lambda^i$.  For each part $\lambda^i_{j}$, choose $v_{ij}\in V_i$ in an $N_i$-invariant subspace corresponding to that Jordan block such that the nonzero elements of the form $N_i^kv_{ij}$ span the {$N_i$-invariant subspace}.  A representation of $Q$ compatible with $(N_i)$ is determined by specifying the image of $v_{ij}$ under the map corresponding to each arrow $a:i \rightarrow k$, and the image of $v_{ij}$ can be freely chosen in the subspace of $V_k$ annihilated by $N_k^{\lambda^i_j}$. This shows that the representations compatible with $N$ form an affine space $X$ inside $\rep (Q,\textbf{d})$.  

Now, we can use exactly the same argument as given by Kraft
and Riedtmann \cite{KR} in their proof of Kac's canonical decomposition
theorem.  Namely, for any decomposition $$\textbf{d}=
\textbf{d}_1+\dots+\textbf{d}_r$$
the locus within $\rep(Q,\textbf{d})$ such that the corresponding representation admits a direct sum decomposition with those dimension vectors forms a constructible set; it therefore follows that the locus admitting such a direct sum decomposition where the summands are indecomposable, is also constructible.  These constructible sets are obviously disjoint and cover $\rep(Q,\textbf{d})$.  The intersection of  
each of these sets with the affine space constructed above are also constructible; it follows that exactly one of them contains a dense open subset of the affine space.

This establishes that the representations compatible with
$N$ form an irreducible variety.  The representations compatible with some collection of nilpotent endomorphisms of the specified Jordan form are then found by closing under the base change action of $GL(\bd)$, which, as in the proof of Lemma \ref{nend-irred}, preserves irreducibility.
\end{proof}

We have the following immediate corollary, since if $Q$ is Dynkin, then knowing the dimension vectors of the indecomposable summands of a representation determines the representations up to isomorphism.

\begin{corollary}\label{cor:genrep} Let $Q$ be a Dynkin quiver, let 
$\bd\in \mathbb{N}^n$, and let $\bnu\vdash \bd$.  Let $N$ be a 
collection of nilpotent linear transformations with $\JF(N)=\bnu$.
Then there is a dense open set in the variety of representations 
compatible with $N$ within which the representations are all 
isomorphic.  
\end{corollary}

In the setting of Corollary~\ref{cor:genrep}, we define $\GR(\bnu)$ to be a representation that is isomorphic to the 
representations corresponding to points in the dense open set.

\section{Reflection functors}\label{Sec_refln_functors}
For the remainder of the paper, we assume that $Q$ is an acyclic quiver. This section proceeds as follows. In Section~\ref{sec:refl}, we recall the definition of reflection functors, which are functors from $\rep Q$ to $\rep Q'$, where $Q'$ is obtained by reversing all the arrows at a source or sink of $Q$. In Section~\ref{ref-and-geom}, we present a geometric interpretation of the reflection functors: we show that they give a canonical identification between certain quotients of open subsets of the representation spaces $\rep(Q,\textbf{d})$ and $\rep(Q',\textbf{d}')$. Here $\textbf{d}'$ is determined by $\textbf{d}$. In Section~\ref{sec:reflandcanonical} and Section~\ref{sec:refl_nilpotent}, we prove the main results of this section: Theorems \ref{th-ref-canon} and \ref{th-ref-nil}. Informally, these theorems say that (under some conditions) reflection functors map generic representations to generic representations.

\subsection{Definition of reflection functors}\label{sec:refl}
Following \cite{assem2006elements}, we now review the definitions of reflection functors in the sense of Bernstein--Gelfand--Ponomarev, which were introduced in \cite{bernstein1973coxeter}. 
We only recall how reflection functors act on objects.

Given a vertex $k \in Q_0$, let $\sigma_k(Q)$ be the quiver obtained from $Q$ by reversing the direction of all arrows of $Q$ that are incident to $k$. Now fix some $k \in Q_0$ that is a sink, and let $Q^\prime = \sigma_k(Q)$. 
We define the \textit{reflection functor} $${R}^+_k: \rep Q \longrightarrow \rep Q^\prime$$ as follows. Given $V = ((V_i)_{i \in Q_0}, (f_a)_{a \in Q_1}) \in \rep Q$, we set ${R}_k^+(V) := ((V^\prime_i)_{i \in Q^\prime_0}, (f^\prime_a)_{a \in Q^\prime_1}) \in \rep Q^\prime$ where
\begin{itemize}
\item $V_i^\prime = V_i$ for $i \neq k$ and $V^\prime_k$ is the kernel of the map $(f_a)_{a: s(a)\to k}: \left(\bigoplus_{a: s(a) \to k} V_{s(a)}  \right)\longrightarrow V_k ,$
\item $f^\prime_a = f_a$ for all arrows $a: i \to j \in Q_1$ with $j \neq k$, and for any arrows $a: i \to k \in Q_1$ the map $f^\prime_a: V_k^\prime \to V^\prime_i = V_i$ is the composition of the inclusion of $V_k^\prime$ into $\bigoplus_{a: s(a) \to k} V_{s(a)}$ with the projection onto the direct summand $V_i$.
\end{itemize}
It is convenient to introduce the notation $V_\adj$ for $\bigoplus_{a: s(a) \to k} V_{s(a)}$.

Now suppose we fix a source $k \in Q_0$, and let $Q^\prime = \sigma_k(Q).$ We define the \textit{reflection functor} $${R}^-_k: \rep Q \longrightarrow \rep Q^\prime$$ as follows. Given $V = ((V_i)_{i \in Q_0}, (f_a)_{a \in Q_1}) \in \rep Q$, we set ${R}_k^-(V) := ((V^\prime_i)_{i \in Q^\prime_0}, (f^\prime_a)_{a \in Q^\prime_1}) \in \rep Q^\prime$ where
\begin{itemize}
\item $V_i^\prime = V_i$ for $i \neq k$ and $V^\prime_k$ is the cokernel of the map $(f_a)_{a:k\to t(a)}: V_k \longrightarrow \left(\bigoplus_{a: k \to t(a)} V_{t(a)}  \right),$
\item $f^\prime_a = f_a$ for all arrows $a: i \to j \in Q_1$ with $i \neq k$, and for any arrows $a: k\to j \in Q_1$ the map $f^\prime_a: V^\prime_j = V_j \to V_k^\prime $ is the composition of the inclusion of $V_j$ into $\bigoplus_{a: k \to t(a)} V_{t(a)}$ with the cokernel projection onto $V^\prime_k$.
\end{itemize}
Similarly, in this setting, we write $V_\adj$ for $\bigoplus_{a: k \to t(a)} V_{t(a)}$.

For $\bd\in \mathbb Z^n$ and $k$ a source, define $\refl_k(\bd)$ to be the vector that coincides with $\bd$ except that 
$$\refl_k(\bd)_k = -d_k+ \sum_{a: k\to t(a)} d_{t(a)}.$$
If $V$ is a representation of $Q$ that has no $S_k$ summand, then $\textbf{dim}(R^-_k(V)) = \refl_k(\textbf{dim}(V))$.
We {similarly} define $\refl_k(\textbf{d})$ when $k$ is a sink.  

\subsection{Geometry of reflection functors}\label{ref-and-geom}
It will be useful to interpret reflection functors in
terms of the geometry of representation spaces. 
Suppose $Q$
is a quiver, $k$ is a source of $Q$, and $\bd$ is a 
dimension vector.

For $V$ a representation of $Q$, the collection of all the maps leaving $V_k$ in $V$ can be viewed as a single map from $V_k$ into $V_\adj$. 
The kernel of the map from $V_k$ to $V_\adj$ is exactly
the sum of all the copies of the simple $S_k$ in $V$.  In particular, $V$ has no $S_k$ summand if and only if the map from $V_k$ to $V_\adj$ is injective.  

Inside $\rep(Q,\bd)$, we can consider the subset $\repi(Q,\bd)$ where the map from $V_k$ to $V_\adj$ is injective.  
(Note that the definition of $\repi(Q,\bd)$ also depends
on the choice of $k$, but $k$ will be fixed throughout, so we suppress it.)
This is an open subset, since 
its complement is defined by the fact that the matrix defining the map from $V_k$ to $V_\adj$ is not of full rank, and that is a closed condition.  

Similarly, let $Q'= \sigma_k(Q)$, and $\bd'\in \mathbb N^n$. 
Consider $\rep(Q',\bd')$, and
let $W$ be the representation corresponding to a point 
in it. Define $W_\adj$ to be the direct sum over all arrows into $k$ of the vector space at the source of the arrow. Inside $\rep(Q',\bd')$, there is an open subset 
$\reps(Q',\bd')$ where the map from $W_\adj$ to
$W_k$ is surjective.  This is exactly the region where $W$
has no summand isomorphic to the simple representation of $Q'$ at $k$, which we denote $S_k'$.  

The most naive thing one might hope for would be to guess that if $\bd'=\refl_k(\bd)$, then reflection functors would establish an isomorphism of varieties between $\repi(Q,\bd)$ and $\reps(Q',\bd')$.  But this is obviously the wrong thing to hope for because these spaces admit natural actions by different groups:
$\repi(Q,\bd)$ has an action of $GL(\bd)$, while $\reps(Q,\bd')$ has an action of $GL(\bd')$.

Define $\trepi(Q,\bd)$ to be the quotient of $\repi(Q,\bd)$
by the action of $GL(d_k)$ at $V_k$. 
Concretely, $\trepi(Q,\bd)$ consists of the matrices assigned to  the arrows of $V$ which do not involve $k$, times the Grassmannian of $d_k$-dimensional subspaces of 
$V_\adj$, which defines the image of $V_k$.

Similarly, we can define $\treps(Q',\bd')$. It consists of
the matrices {assigned to} the arrows of $V$ which do not involve $k$, {times} the Grassmannian of $d'_k$-dimensional quotient spaces of $V_\adj$.

Now observe that $d_k+d'_k=\dim V_\adj$, and we see that there is a canonical identification between $\trepi(Q,\bd)$ and $\treps(Q',\bd')$, identifying the $d_k$-dimensional subspace $L$ with the $(\dim V_\adj - d_k)$-dimensional quotient $V_\adj/L$.  This identification is, of course, exactly what is effected by the reflection functors.

\subsection{Interlacing partitions}
In this section, we introduce the notion of what it means
for two partitions to be interlaced, or, more 
generally, $t$-interlaced. We will apply this in the
next section to describe the effect of reflection functors.

Two partitions $\lambda$ and $\mu$ are interlaced if 
$$\mu_1\geq \lambda_1\geq \mu_2 \geq \mu_3 \geq \lambda_2 
\geq \mu_4\geq \mu_5 \geq \lambda_3\dots.$$
We think of $\lambda$ and $\mu$ as being padded with an infinite string of zeros, and we insist that the above inequalities hold there too. (In other words, after the last non-zero part of $\lambda$, there can be at most two non-zero values of $\mu$, and before the last non-zero part of $\lambda$, all the parts of $\mu$ must be non-zero.) 

The example to bear in mind that explains the relevance of
this condition is the following. Let $\nu$ be a partition, which we think of as defining a Ferrers diagram consisting of $\nu_1$ boxes in the top row, $\nu_2$ boxes in the next, and so on, all left justified. We will be interested in reverse plane partitions of shape $\nu$, that is to say, fillings of this shape $\nu$ with entries from $\mathbb{N}$ which weakly
increase to the right along rows and down columns. 

We group the boxes of $\nu$ into northwest-southeast diagonals. We say that a box is removable from $\nu$ if the diagram obtained by removing the box would still be a Ferrers shape. Similarly, we say that a box is addable if the result of adding it would still be a Ferrers shape. 

\begin{example} \label{exa}
Suppose that 
$R$ is a reverse plane partition of shape $\nu$ which has neither a 
removable box nor an addable box in
diagonal $i$. Let $\lambda$ be the partition composed of the entries of $R$ in diagonal $i$, and let $\mu$ be the partition composed of the entries in the adjacent diagonals. Then $\lambda$ and $\mu$ are interlaced.

In the example below, let $\lambda$ be composed of the entries in the main diagonal so that $\lambda=(9,5,1)$. There is neither an addable box nor a removable box in this diagonal.
The partition formed from the adjacent diagonals is $\mu=(9,7,5,3,2)$. We see that $\lambda$ and $\mu$ are interlaced.
$$\begin{ytableau}
1&3&4&4\\ 
2&5&7\\
3&5&9\\
4&6&9
\end{ytableau}$$
\end{example}

We now define what it means for $\lambda$ and $\mu$ to be 
$t$-interlaced for $t\geq 0$. Let 
$\mu_{>t}$ be the partition obtained by removing the first 
$t$ parts from $\mu$. Then $\lambda$ and $\mu$ are 
$t$-interlaced if and only if $\lambda$ and $\mu_{>t}$ are 
interlaced. 

\begin{example} \label{exb} To return to the example of reverse plane partitions, suppose that the partition $\nu$ has an addable box in diagonal $i$.
Defining $\lambda$ and $\mu$ as before, we see that 
they are 1-interlaced.

In the example below, let $\lambda$ be composed of the entries on the main diagonal, so $\lambda=(4,1)$, while 
$\mu$ is composed of the entries in the adjacent diagonals, so $\mu=(6,5,2,1)$.
We see that $\mu_{>1}=(5,2,1)$, and
observe that $\lambda$ and $\mu_{>1}$ 
are interlaced, so $\lambda$ and 
$\mu$ are 1-interlaced.
$$
\begin{ytableau}
1&2&5\\
1&4&6\\
2&5\\
\end{ytableau}$$
\end{example}

We also need to define being $t$-interlaced for $t<0$.  
The partitions $\lambda$ and 
$\mu$ are $t$-interlaced if $\lambda_1=\mu_1$, $\lambda_2=\mu_2$, \dots, $\lambda_t=\mu_t$, and 
$\lambda_{>t}$ and $\mu_{>t}$ are interlaced.

\begin{example} \label{exc}
Again considering reverse plane partitions, suppose that
$\nu$ has a removable box in the $i$-th diagonal. Let $R$ be a reverse plane partition of shape $\nu$, and suppose that the entry in the removable box of in diagonal $i$ is equal to the larger of its neighbors. Let $\lambda$ be composed of the entries in the $i$-th diagonal, while $\mu$ is composed of the entries in the adjacent diagonals. Then $\lambda$ and $\mu$ are $(-1)$-interlaced

Consider the following example, where we choose $\lambda$ to be composed of the entries from the main diagonal:
$$
\begin{ytableau}
0&2&5\\
1&3&6\\
2&5&6\\
\end{ytableau}$$
We see that $\lambda=(6,3)$ and 
$\mu=(6,5,2,1)$, so they are indeed $(-1)$-interlaced.
\end{example}

If $\lambda$ and $\mu$ are $t$-interlaced for $t\geq 0$, define $$\diff(\mu,\lambda)=(\mu_1,\mu_2,\dots,\mu_{t},
\mu_{t+1}+\mu_{t+2}-\lambda_1, \mu_{t+3}+\mu_{t+4}-\lambda_2,\dots).$$

\begin{example} In the setting of Example \ref{exa}, 
if we toggle the entries in the $i$-th diagonal, the resulting entries are given by $\diff(\mu,\lambda)$.

For the specific reverse plane partition given in that example, we obtain:
$$\begin{ytableau}
1&3&4&4\\ 
2&3&7\\
3&5&7\\
4&6&9
\end{ytableau}$$
The result of toggling is $(7,3,1)$, which is indeed
$\diff(\mu,\lambda)=
(9+7-9,5+3-5,2+0-1)=(7,3,1)$

Let $R$ is a tableau with an addable box in the $i$-th diagonal, 
as in Example \ref{exb}, and let $R'$ be obtained by toggling the entries of $R$ on the $i$-th diagonal and then adding into the addable box the entry given by the larger of its neighbours. The result is a reverse plane partition as in \ref{exc}. Let $\lambda$ be composed of the entries on the $i$-th diagonal of $R$, let $\mu$ be composed of the entries on the adjacent diagonals, and let $\lambda'$ be composed of the entries on the $i$-th diagonal of $R'$. Then 
$\lambda'=\diff(\mu,\lambda)=$. \end{example}

If $\lambda$ and $\mu$ are $t$-interlaced for $t\leq 0$, define
$$\diff(\mu,\lambda)=(\mu_{-t+1}+\mu_{-t+2}-\lambda_{-t+1},
\mu_{-t+3}+\mu_{-t+4}-\lambda_{-t+2},\dots).$$

We state some elementary properties of interlacing.

\begin{lemma} \begin{enumerate}
\item Partitions $\lambda$ and $\mu$ can be $t$-interlaced for at most three different values of $t$.
\item If $\lambda$ and $\mu$ are $t$-interlaced for more than one value of $t$, the partition $\diff(\mu,\lambda)$ does not 
depend on the value of $t$ used in the definition.  
\item If $\lambda$ and $\mu$ are $t$-interlaced, then 
$\diff(\mu,\lambda)$ and $\mu$ are $(-t)$-interlaced.
\item $\diff(\mu, \diff(\mu,\lambda))=\lambda$.
\end{enumerate}
\end{lemma}

\subsection{Reflection functors and canonical representations}\label{sec:reflandcanonical}
In this section, we suppose that $V$ is a generic representation compatible with an $n$-tuple of nilpotent linear transformations 
$(N_i)$ with Jordan forms given by $\bnu$.  
Under some assumptions on $\bnu$, {and assuming $Q$ is Dynkin,} we show that the result of applying a reflection functor to $V$ is isomorphic to the generic representation compatible with an $n$-tuple of nilpotent linear transformations of Jordan form $\bnu'$, where $\bnu'$ is determined by $\bnu$.

\if false
We say that two partitions $\lambda$ and $\mu$ are \textit{$t$-interlaced}, for $t\geq 0$, if 
$$\mu_1\geq \dots\geq \mu_{t+1} \geq \lambda_1\geq \mu_{t+2} 
\geq \mu_{t+3}\geq \lambda_2\geq \mu_{t+4}\geq \mu_{t+5}
\geq \lambda_3 \dots.$$
We think of $\lambda$ and $\mu$ as being padded with an infinite string of zeros, and we insist that the above inequalities hold there too. (In other words, after the last non-zero part of $\lambda$, there can be at most two non-zero values of $\mu$, and before the last non-zero part of $\lambda$, all the parts of $\mu$ must be non-zero.)

We say that $\lambda$ and $\mu$ are \textit{$t$-interlaced} for $t\leq 0$
if $\lambda_1=\mu_1$, $\lambda_2=\mu_2$,\dots, $\lambda_{-t}=\mu_{-t}$, and 
$$\mu_{-t+1} \geq \lambda_{-t+1} \geq \mu_{-t+2} 
\geq \mu_{-t+3}\geq \lambda_{-t+2}\geq \mu_{-t+4}\geq \mu_{-t+5}
\geq \lambda_{-t+3} \dots.$$

\fi 

Let $k$ be a source of $Q$.  
Fix $\bd\in \mathbb N^n$, and $\bnu \vdash \bd$. We let $\nu^\adj$ be the partition whose parts are $\bigsqcup_{a:k\rightarrow t(a)} \nu^{t(a)}$ (i.e., we take the multiset of all the parts of all the neighbours of $k$, with multiplicity given by the number of arrows to each neighbour).
Suppose that $\nu^k$ and $\nu^\adj$ 
are $t$-interlaced for some $t$.  Define 
$\refl_k(\bnu)$ to be the $n$-tuple of partitions that coincides with $\bnu$ except that $\nu^k$ has been replaced by $\diff(\nu^\adj,\nu^k)$.  

\begin{lemma} \label{big-lemma} 
Let $k$ be a source of $Q$ and let $\bnu\vdash \bd$.  Let $\lambda=\nu^k$, and 
$\mu=\nu^{\adj}$, and suppose that $\lambda$ and $\mu$ are
$t$-interlaced for some $t\in\mathbb Z$.  
Let $N$ be a nilpotent linear operator of Jordan form $\bnu$. Let $V$ be a representation of $Q$ chosen generically among representations compatible with $N$.  
Then the Jordan blocks of $N$ as it acts on the representation $R_k^-(V)$ are given by $\refl_k(\bnu)$, and $V$ does not have any indecomposable summands of the form $S_k$.  
\end{lemma}

\begin{remark}
Recall that applying $R_k^-$ to the representation $V = ((V_i)_{i \in Q_0}, (f_a)_{a \in Q_1})$ replaces vector spaces $V_k$ with $V^\prime_k$, which is the cokernel of the map $(f_a)_{a:k\to t(a)}: V_k \longrightarrow \left(\bigoplus_{a: k \to t(a)} V_{t(a)}  \right)$ and the other vectors spaces $V$ are unchanged. There is a natural action of $N$ on this cokernel. Proving the lemma amounts to showing that the Jordan blocks of the action of $N$ on this cokernel are given by $\diff(\nu^\adj,\nu^k)$.
\end{remark}

The dual version of the lemma, where $k$ is a sink, also holds. This follows by applying the lemma to the representation of the opposite quiver using the dual vector spaces.

\begin{proof}
Fix $t$ such that $\lambda$ and $\mu$ are $t$-interlaced.  Suppose first that $t\geq 0$.

Let $V_1,\dots,V_n$ be a collection of vector spaces with 
dimensions given by $\bd=(d_1,\dots, d_n)$.  
Fix $N$ an $n$-tuple of nilpotent linear transformations, with 
$\JF(N)=\bnu$. 

Let $U$ be a vector space on which a nilpotent linear transformation $T$ acts, and let $\lambda$ be the Jordan form of $T$.  We say that 
$\{u_1,\dots,u_{\ell(\lambda)}\}$  
is a \emph{good set of generators} for $U$ if 
dim$\fk[T]u_i=\lambda_i$ {for all $i \in \{1, \ldots, \ell(\lambda)\}$} and $U=\fk[T]u_1\oplus\cdots \oplus \fk[T]u_{\ell(\lambda)}$.

Fix a good set of generators for $V_k$ with respect to the action of 
$N_k$. There is a natural action of $N$ on $V_\adj$.  
In $V_{\adj}$, fix $N$-invariant subspaces $W_i$ such that for each $i$, dim$W_i=\mu_i$, each $W_i$ is generated by a single vector, and $V_{\adj}=\bigoplus_j W_j$. 

A map $f\in \Hom _{\fk[N]}(V_k,V_{\adj})$ can be specified by giving the image of each $v_i$.  We can take $f(v_i)$ and split it up according to the subspaces $W_j$. $$ f(v_i)=\sum_{j} \tilde w_{ij}$$ with $\tilde w_{ij} \in W_j$. However, we would like to record the fact that typically $\tilde w_{ij}$ cannot be arbitrary in $W_j$ because $f$ must respect the $k[N]$-module structure. So let us write instead:
$$f(v_i)=\sum_{j\leq 2i+t-1} N^{\mu_{j} - \lambda_i}w_{ij} + \sum_{j>2i+t-1} w_{ij},$$
where the elements $w_{ij}$ can be chosen arbitrarily in $W_j$, and they automatically define such a map $f$. 

There is a dense open subset among the representations $V$ compatible with $N$ such that if we take $f$ to be the corresponding map from $V_k$ to $V_\adj$, then $w_{ij}$ generates $W_j$ as a module over $\fk[N]$.  

Next, we want to rewrite $f$. This will lead to a different direct sum decomposition of $V_{\adj}$. As a result of the new direct sum decomposition, the problem  reduces to the situation where $k[N]/N^b$ is mapped generically into $k[N]/N^a \oplus k[N]/N^c$ with $a\geq b \geq c$.

  Define
      \begin{eqnarray*}
    x_i&=& \sum_{j\leq 2i+t-1} N^{\mu_j-\mu_{2i+t-1}} w_{ij}\\
    y_i&=& \sum_{j> 2i+t-1} w_{ij}. \end{eqnarray*}
Then we have the following:
\begin{equation}\label{key} f(v_i)=N^{\mu_{2i+t-1}-\lambda_i} x_i + y_i
\end{equation}

For $1\leq i\leq t$, define $w_i$ to be a generator for $W_i$. The following lemma shows that the vectors which we have found above form a good set of generators for $V_{\adj}$.

\begin{lemma}\label{lem:directsum}
There is a direct sum decomposition of $V_{\adj}$ into cyclic subspaces generated by the elements $x_i$ and $y_i$ as above, and by $w_i$ for $1\leq i \leq t$, where 
  
\begin{enumerate}
\item[(i)] $w_i$ generates a subspace of dimension $\mu_i$ (for $1\leq i \leq t$),

\item[(ii)] $x_i$ generates a subspace of dimension $\mu_{2i+t-1}$,

\item[(iii)] $y_i$ generates a subspace of dimension $\mu_{2i+t}$, and

\item[(iv)] these subspaces are complementary.
\end{enumerate}
\end{lemma}

\begin{proof} It is clear that $x_i$, $y_i$, and $w_i$ generate subspaces of the claimed dimensions.  It is therefore enough to show that the given vectors and suitable powers of $N$ applied to the given vectors suffice to span $V_\adj$.  This is a condition that a certain determinant (which we can think of as a function of the coefficients defining $f$) is non-zero.  For a particular choice of $f$, this is certainly true, namely the case that $x_i$ generates $W_{\mu_{2i+t-1}}$
and $y_i$ generates 
$W_{\mu_{2i+t}}$.  Since the condition that a determinant vanishes is a closed condition and it does not hold for all choices of $f$, it does not hold for a generic choice of $f$.  \end{proof}

Thus we obtained the alternative direct sum decomposition of $V_\adj$ given by $$\bigoplus_i \fk[N]w_i \oplus \bigoplus_{i} \fk[N]x_i \oplus \bigoplus_i \fk[N]y_i,$$ and (\ref{key}) provides a very simple description of $f$.  Specifically, each $v_i$ is mapped into exactly two summands of this decomposition, and no two $v_i$ are mapped into the same summand. Thus the problem reduces to the situation where $\fk[N]/N^b$ is mapped generically into $\fk[N]/N^a \oplus \fk[N]/N^c$ with $a\geq b \geq c$ as claimed. The following lemma is easily verified.

\begin{lemma}\label{Lem_coker_dim}
The cokernel of a generic $\fk[N]$-homomorphism from $\fk[N]/N^b$ into $\fk[N]/N^a \oplus \fk[N]/N^c$ with $a\geq b \geq c$ is isomorphic to $\fk[N]/N^{a+c-b}$.  
\end{lemma}

For $t\geq 0$, this establishes the claim that for $V$ chosen in a dense open subset of the representations compatible with $N$,
the natural action of $N$ on $R_k^-(V)$ has Jordan form 
$\refl_k(\bnu)$.  The argument for $t\leq 0$ is similar.

This also shows that generically $V$ does not have
any indecomposable summands isomorphic to $S_k$, since if it did, the map from $V_k$ to $V_\adj$ would not be injective.\end{proof}

For convenience, we state the following theorem in the Dynkin setting, which allows us to use Corollary \ref{cor:genrep} and thus to refer to $\GR(\bnu)$. 

\begin{theorem} \label{th-ref-canon} 
Let $k$ be a source of {a Dynkin quiver} 
$Q$ and let $\bnu\vdash \bd$.  Let $\lambda=\nu^k$, and 
$\mu=\nu^{\adj}$, and suppose that $\lambda$ and $\mu$ are
$t$-interlaced for some $t\geq 0$. Let $V\simeq\GR(\bnu)$.     
Then $R_k^-(V)$ is isomorphic to $\GR(\refl_k(\bnu))$.   
\end{theorem}

\begin{proof}
Let $N$ be an $n$-tuple of nilpotent linear transformations such that
$\JF(N)=\bnu$.  
As in Section \ref{ref-and-geom}, we consider $\trepi(Q,\bd)$.  The locus of representations compatible with $N$ inside 
$\rep(Q,\bd)$ is irreducible, and it follows that the same is true of the locus of representations compatible with $N$ up to change of basis at $k$ inside $\trepi(Q,\bd)$.  Call this locus $X$.  More concretely, the points of $X$ correspond to representations compatible with $N_i$ for $i\ne k$ and some nilpotent transformation with Jordan form $\nu^k$ at $k$.  
Under the identification from Section \ref{ref-and-geom}, $X$ also corresponds to representations compatible with $N$ up to change of basis at $k$ inside $\treps(Q',\bd')$.  On a dense open
set inside $X$,  by Lemma \ref{big-lemma}, the
action of $N$ has Jordan form $\refl_k(\bnu)$.  Thus, on this open set inside $X$, we find representations compatible with a nilpotent endomorphism of Jordan form $\refl_k(\bnu)$. 

We would like to conclude that the generic representation compatible with a nilpotent endomorphism of Jordan form $\refl_k(\bnu)$ is isomorphic to $R^-_k(V)$.  This does not yet follow, because we could imagine that there is some larger region $U$ inside $\treps(Q',\bd')$ consisting of representations compatible with a nilpotent endomorphism of Jordan form $\refl_k(\bnu)$, with $X$ closed inside $U$.  Choose a point from $U\setminus X$.  This determines a nilpotent endomorphism $N'$ and a representation $V'$ compatible with it.  Considering all the representations compatible with $N'$, we know that not all of them correspond to points in $X$ (since $V'$ does not); thus, an open set of the representations compatible with $N'$ lie outside $X$.  Call this open set $Z$.  
We can now apply the dual version of Lemma \ref{big-lemma} to conclude that for $W$ a generic representation in $Z$, the nilpotent endomorphism induced on $R^+_k(W)$ by $N'$ has Jordan form $\bnu$.  But by assumption $R^+_k(W)$ does not correspond to a point of $\trepi(Q,\bd)$ lying in $X$, which is a contradiction.  It follows that the region of $\treps(Q',\bd')$ compatible with some nilpotent endomorphism with Jordan form $\refl_k(\bnu)$ is contained in $X$.  
Since the generic isomorphism class inside $X$ as a subset of $\treps(Q',\bd')$ is $R^-_k(V)$, we are done.   
\end{proof}

\subsection{Reflection functors and nilpotent endomorphisms}\label{sec:refl_nilpotent} In this section, we assume that {$Q$ is Dynkin and that} we have both $V \simeq \GR(\bnu)$ and $\bnu= \JFg(V)$, and we deduce that $\JFg(R^-_k(V))=\sigma_k(\bnu)$, under some assumptions on $\bnu$.
This complements the result of Theorem~\ref{th-ref-canon}, which tells us that
under weaker hypotheses, $R^-_k(V)\simeq \GR(\sigma_k(\bnu))$.  Combining the two results, under some assumptions on $\bnu$, we go from the assumption that the isomorphism class of $V$ and the Jordan form data $\bnu$ each determine the other, and we deduce that the same is true for $R^-_k(V)$ and $\sigma_k(\bnu)$.

We first need the following lemma.

\begin{lemma} \label{jf-interlace}
Let $0\rightarrow P \rightarrow Q \rightarrow R \rightarrow 0$
be a short exact sequence of finite-length $\fk[N]$-modules.
\begin{enumerate}
\item Let the Jordan forms of $N$ acting on $P$, $Q$, and $R$ be 
$\lambda$, $\mu$, and $\nu$, respectively. Then $\mu\geq \lambda+\nu$,
where we write $\lambda+\nu$ for the partition whose multiset of parts is composed of the parts of $\lambda$ and $\nu$.  
\item Suppose further that $\lambda$ and $\mu$ are $t$-interlaced for some $t\in\mathbb Z$.  
Then $\nu\not> \diff(\mu,\lambda)$.  
\end{enumerate}
\end{lemma}

\begin{proof} (1) For $k$ any positive integer, there is an 
exact sequence 
$$0\rightarrow \ker_{P} N^k \rightarrow \ker_{Q} N^k 
\rightarrow \ker_R N^k,$$
where we write $\ker_{P} N^k$ for the kernel of $N^k$ as it acts on $P$, and similarly for the other two expressions.  

Note that the dimension of $\ker_P N^k$ is the sum of the first $k$ parts of $\lambda^t$, and similarly for $Q$ and $R$.  
From the short exact sequence, it follows that the sum of the first $k$ parts of $\lambda^t$
plus the first $k$ parts of $\nu^t$ is at least as great as the sum of the first $k$ parts of $\mu^t$.  This implies that
$\mu^t \leq (\lambda+\nu)^t$.  Thus $\mu \geq \lambda+\nu$.

(2) Suppose $t\leq 0$. Suppose that $\nu_1+\cdots+\nu_k>\diff(\mu,\lambda)_1+\cdots+\diff(\mu,\lambda)_k$ for some $k\geq 1$. It follows that
\[\nu_1+\cdots+\nu_k+\lambda_{-t+1}+\cdots+\lambda_{-t+k}>\mu_{-t+1}+\cdots\mu_{-t+2k}.\] By adding $\lambda_1+\cdots+\lambda_{-t}=\mu_1+\cdots+\mu_{-t}$ to both sides, we obtain
\[\nu_1+\cdots+\nu_k+\lambda_{1}+\cdots+\lambda_{-t+k}>\mu_{1}+\cdots+\mu_{-t+2k}.\] By the definition of $\lambda+\nu$, we see that 
\[(\lambda+\nu)_1+\cdots+(\lambda+\nu)_{-t+2k}\geq\nu_1+\cdots+\nu_k+\lambda_{1}+\cdots+\lambda_{-t+k},\] which, together with the previous equation, contradicts $\mu\geq\lambda+\nu$. The argument for $t\geq 0$ is analogous.
\end{proof}

The following theorem adds to the hypotheses of Theorem 
\ref{th-ref-canon}: we pick a Jordan form $\bnu$, and not only do we assume that $V\simeq\GR(\bnu)$, but also $\bnu=\JFg(V)$.  In this case, we can also
describe $\JFg(R^-_k(V))$.  

\begin{theorem} \label{th-ref-nil} 
Let $k$ be a source of {a Dynkin quiver} $Q$ and let $\bnu\vdash \bd$.  Let $\lambda=\nu^k$, and 
$\mu=\nu^{\adj}$, and suppose that $\lambda$ and $\mu$ are
$t$-interlaced for some $t\geq 0$.  
Suppose that $V\simeq\GR(\bnu)$ and also $\bnu=\JFg(V)$.
Then $\JFg(R_k^-(V)) = \refl_k(\bnu)$.  
\end{theorem}

\begin{proof} 
Choose an $n$-tuple of nilpotent transformations $N_i$ with $\JF(N_i)=\bnu^i$, and then choose a generic representation $W$ among those compatible with $N$.
By assumption, $W\simeq V$.  Lemma \ref{big-lemma} applies to $W$, and we conclude that the action of $N$ on $R_k^-(W)$ has Jordan form $\refl_k(\bnu)$.

Thus, $R_k^-(V)$ (which is isomorphic to $R_k^-(W)$) has some nilpotent endomorphism acting on it with Jordan form $\refl_k(\bnu)$.  Since the Jordan form of a generic nilpotent endomorphism is maximal among nilpotent endomorphisms acting on $R_k^-(V)$, we know that 
$\JFg(R_k^-(V))\geq \refl_k(\bnu)$.  

Now consider the short exact sequence of vector spaces 
$$0 \rightarrow V_k \rightarrow V_{\adj} \rightarrow R_k^-(V)_k
\rightarrow 0$$
Let $N'$ be a generic nilpotent endomorphism of $V$, which induces nilpotent transformations of these three vector spaces.  
Lemma \ref{jf-interlace}(2) therefore applies to say that
the Jordan form of $N'$ acting on $R_k^-(V)_k$ is not greater than $\diff(\nu^\adj, \nu^k)$. 

Therefore the inequality which we already proved must be an equality: $\JFg(R_k^-(V))=\refl_k(\bnu)$.  
\end{proof}

\section{Minuscule posets}
\subsection{Definitions}
\label{min-poset-def}

\newcommand{\mP}{P}
\renewcommand{\mod}{\operatorname{mod}}
Let $Q$ be a Dynkin quiver.  
A \textit{heap} over $Q$ is a poset $\sP$ equipped with a {surjective} map $\pi$ from $\sP$ to the vertices of the underlying
graph of $Q$, with the properties that \begin{itemize}
\item[(H1)] the inverse image of $i$ is
totally ordered, 
\item [(H2)] for $i$ and $j$ adjacent, the union of the
inverse images of $i$ and $j$ is totally ordered,
\item [(H3)] the poset structure on $\sP$ is the transitive closure of the relations corresponding to the two previous points.
\end{itemize}

We further say that a heap is \textit{two-neighbourly} if in the interval between any two
consecutive elements of $\pi^{-1}(i)$, there are exactly two occurrences of elements of the form $\pi^{-1}(j)$ for some $j$ adjacent to $i$. Note that the two elements may correspond to different neighbours $j$ and $j'$ or the same neighbour repeated twice. We say a heap has the weaker property of being \textit{neighbourly} if such an interval contains at least two occurrences of elements of the form $\pi^{-1}(j)$ for some $j$ adjacent
to $i$.   

We call a neighbourly heap $(\sP,\pi)$ over $Q$ a \textit{maximal neighbourly heap} if one cannot add a new element to $\sP$ while maintaining the relative order of elements of $\sP$ and obtain a larger neighbourly heap over $Q$. In \cite{wildberger2003minuscule}, Wildberger proves that the set of maximal neighbourly heaps that are also two-neighbourly is exactly the set of minuscule posets. The reader unfamiliar with minuscule posets may take this as the definition of a minuscule poset. We recommend \cite{proctor1984bruhat,green2013combinatorics} for the standard definition of minuscule posets in terms of the representation theory of simple Lie algebras. We give explicit descriptions of the isomorphism types of minuscule posets below. 

A vertex $m$ of $Q$ is called 
\emph{minuscule} if every indecomposable representation of $Q$ supported over $m$ has dimension 1 at $m$. (Properly speaking, this is the definition of a \textit{cominuscule} vertex of $Q$, but since we are working with simply-laced types only, a vertex of $Q$ is minuscule if  and only if it is cominuscule.)

We give explicit descriptions of the minuscule posets, which appear in the context of minuscule representations of complex semisimple Lie algebras and were classified up to isomorphism by Proctor in \cite{proctor1984bruhat}. Recall that for a poset $\P$, an \textit{order ideal} is a subset $\OO \subset \P$ where if $\x \in \OO$ and $\y \le_{\P} \x$, one has that $\y \in \OO$. Similarly, an \textit{order filter} of $\P$ is a subset of $\OO \subset \P$ where if $\x \in \OO$ and $\y \ge_{\P} \x$, one has that $\y \in \OO$. We let $\J(\P)$ denote the poset of order ideals of $\P$, ordered by inclusion. By \cite[Theorem 8.3.10]{green2013combinatorics}, there is a minuscule poset for each choice of a minuscule vertex of a simply-laced Dynkin diagram. Their isomorphism types appear in Table~\ref{table:1}. 
There, we write $[n]$ for the poset that is a chain whose elements are $1, \ldots, n$ in increasing order.

\begin{table}[!htbp]
\centering
\begin{tabular}{c c c } 
 \hline
 Type & $m$ & minuscule poset \\ [0.5ex] 
 \hline\hline
 $A_n$ & $k$ & $[k]\times[n+1-k]$ \\ 
 $D_n$ & $1$ & $\J^{n-3}([2]\times[2])$ \\
 $D_n$ & $n-1$, $n$ & $\J([2]\times[n-2])$ \\
 $E_6$ & $1$, $5$ & $\J^2([2]\times[3])$ \\
 $E_7$ & $6$ & $\J^3([2]\times[3])$  \\ [1ex] 
 \hline
\end{tabular}
\caption{The isomorphism types of the minuscule posets. Here we are referring to the vertex labeling of the Dynkin diagrams appearing in Figure~\ref{diagrams}.}
\label{table:1}
\end{table}

\begin{figure}
\begin{displaymath}
    \scalebox{.75}{\xymatrix{A_n & 1 \ar@{-}[r] & 2 \ar@{-}[r] & & \cdots & & n\ar@{-}[l] & & \\
    & & & & & & n \ar@{-}[dl] \\
    D_n & 1 \ar@{-}[r] & 2 \ar@{-}[r] & & \cdots & n-2 \\
    & & & 6 \ar@{-}[d] & & & n-1\ar@{-}[ul] \\
    E_6 & 1\ar@{-}[r] & 2 \ar@{-}[r] & 3 \ar@{-}[r] & 4 \ar@{-}[r] & 5 \\
   & & & 7 \ar@{-}[d] \\
    E_7 & 1\ar@{-}[r] & 2 \ar@{-}[r] & 3 \ar@{-}[r] & 4 \ar@{-}[r] & 5 \ar@{-}[r] & 6 \\}}
\end{displaymath}
\caption{The Dynkin diagrams.}
\label{diagrams}
\end{figure}
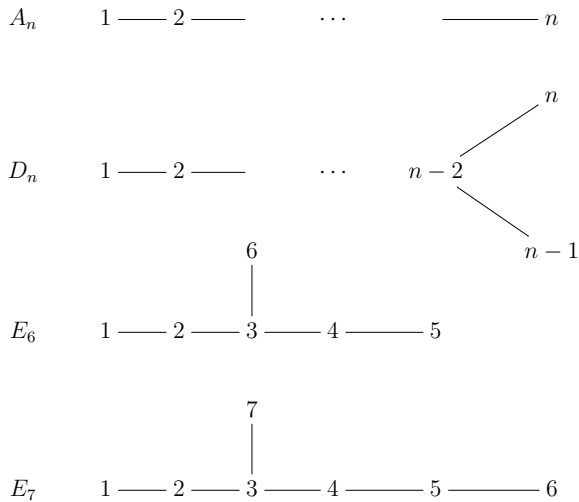

\begin{example}
In Figure~\ref{min_posets_ex}, we show some examples of minuscule posets. In these examples, we have labelled each element of the poset with its corresponding value of $\pi$.
\end{example}

\begin{figure}
$$\begin{array}{ccccccccccc} \includegraphics[scale=.75]{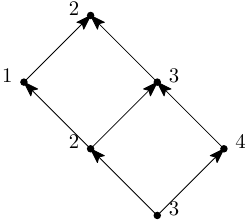} & & & \includegraphics[scale=.75]{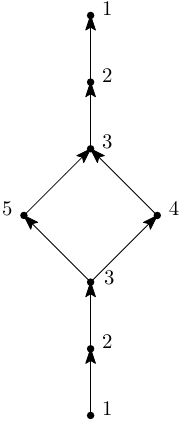} & & & \includegraphics[scale=.75]{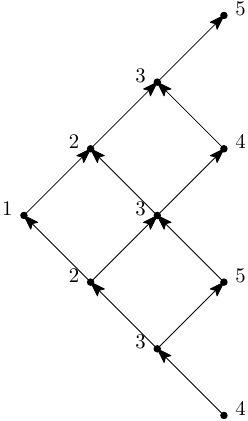}\\ (a) & & & (b) & & & (c)   \end{array}$$
\caption{The minuscule poset of  type $A_4$ with $m = 3$, of type $D_5$ with $m=1$, and type $D_5$ with $m = 4$, respectively, are shown in ($a$), ($b$), and ($c$).}\label{min_posets_ex} 
\end{figure}

Our next result shows how the minuscule posets are related to the representation theory of quivers. Given a minuscule vertex $m \in Q_0$, let $\pos$ denote the poset that is the transitive closure of the arrows in the full subquiver of the AR quiver of $Q$ whose vertex set is the set of isomorphism classes of representations supported at $m$. Now define $\pi:\pos \to Q_0$ to be the map where $\pi(\textsf{x})$ is defined as the vertex corresponding to the indecomposable projective in the same $\tau$-orbit as a representation corresponding to $\textsf{x}$.

\begin{lemma}\label{four-two}
The poset $\pos$ is isomorphic to the minuscule poset determined by the underlying graph of $Q$ and minuscule vertex $m$, and the pair $(\pos, \pi)$ is a two-neighbourly heap.
\end{lemma}

\begin{proof}
First, suppose that $Q^\prime$ is the quiver whose arrows are all oriented toward $m$. The elements of $\posprime$ are exactly those corresponding to the indecomposable representations whose socle is $S_m$. One checks that $\posprime$ is isomorphic to the desired poset. 

With regards to the two-neighbourly heap assertion, one checks that the map $\pi^\prime: \posprime \to Q_0$ produces the same labeling of the elements of $\posprime$ appearing in the $X$-heap of the same type in the sense of \cite{wildberger2003minuscule}. Here $X$ is the underlying graph of $Q^\prime$. Therefore, the pair $(\posprime, \pi^\prime)$ is equivalent to the data of the corresponding $X$-heap, which was shown to be a two-neighbourly heap in \cite{wildberger2003minuscule}. 

Next, we show that the desired result holds for quiver $Q$ whose arrows are oriented in any direction. Note that there exists $i_1, \ldots, i_k \in Q_0\backslash\{m\}$ such that $Q = \sigma_{i_k}\cdots \sigma_{i_1}(Q^\prime)$. The composition of the corresponding sequence of reflection functors defines an equivalence of categories ${R}^{\pm}_{i_k}\cdots {R}^{\pm}_{i_1}: \mathcal{C}_{Q^\prime,m} \to \mathcal{C}_{Q,m}$, since for any $j \in \{1, \ldots, k\}$ none of the representations in $\mathcal{C}_{\sigma_{i_{j-1}}\cdots \sigma_{i_1}(Q^\prime),m}$ have any summands isomorphic to one of the simple representations $S_{i_{j}}, \ldots, S_{i_1}$. We know this by induction: at each step, all the indecomposable of $\mathcal{C}_{\sigma_{i_{j-1}}\cdots \sigma_{i_1}(Q^\prime),m}$ have non-zero support at vertex $m$, and reflecting at a vertex other than $m$ does not change this fact. By the notation ${R}_i^{\pm}$, we mean that we reflect at vertex $i$ and $i$ may be a source or a sink. In particular, for any $M, N \in \mathcal{C}_{Q^\prime,m}$ we have the following isomorphism of spaces of irreducible morphisms $$\text{Irr}({R}^\pm_{i_{k}}\cdots {R}^\pm_{i_1}(M), {R}^\pm_{i_{k}}\cdots {R}^\pm_{i_1}(N))\simeq \text{Irr}(M,N).$$ The poset $\pos$ is therefore isomorphic to $\posprime$, and the isomorphism respects the labelling by $\tau$-orbits.
\end{proof}

\begin{remark} There is another classic construction of minuscule posets. One puts a poset structure on the set of all the positive roots, where $\alpha\geq\beta$ if $\alpha-\beta$ is a non-negative linear combination of simple roots. The minuscule poset associated to the minuscule root $\alpha_i$ can then be described as the interval between the $\alpha_i$ and the highest root; the set of roots appearing in the interval are exactly the positive roots in whose simple root expansion $\alpha_i$ appears. See \cite[Section 8.3]{green2013combinatorics} for more details.

Interestingly, while the construction that we give also defines a bijection between the elements of the poset and this set of positive roots, the bijections are typically not the same. Indeed, the specific bijection that we obtain depends on the choice of orientation of the starting quiver.\end{remark}

The property in the next lemma can be thought of as a converse of the property of being 2-neighbourly.
\begin{lemma}\label{lem:2neighconverse}
The pair $(\pos, \pi)$ has the following property: if there exist elements $\textsf{x},\textsf{y}_1,\textsf{y}_2\in\pos$ such that $\textsf{x}<\textsf{y}_1$, $\textsf{x}<\textsf{y}_2$, and $\pi(\textsf{y}_1)$ and $\pi(\textsf{y}_2)$ are neighbours of $\pi(\textsf{x})$, then there is another element $\textsf{x}'\in\pos$ with $\textsf{x}'\geq \textsf{y}_1$, $\textsf{x}'\geq\textsf{y}_2$, and $\pi(\textsf{x}')=\pi(\textsf{x})$.
\end{lemma}
\begin{proof}
If the described property does not hold, we may add a new maximal element $\textsf{x}'$ to $\pos$ with $\pi(\textsf{x}')=\pi(\textsf{x})$, and the resulting heap would be neighbourly. This contradicts Wildberger's result that minuscule posets are maximal neighbourly heaps.
\end{proof}

The dual version of this lemma also holds.

\begin{remark}\label{lemma_upper_int}
Let $(\P,\pi)$ be a two-neighbourly heap.  
It will be useful to note that for any interval $[\x, \y] \subset \P$, the pair $([\x,\y], \pi|_{[\x,\y]})$ is also a two-neighbourly heap.
\end{remark}

\subsection{Generic Jordan forms and reverse plane partitions}\label{sec:RSK}

\newcommand{\HX}{\mathcal{C}^{\Xi}_{Q,m}}
\newcommand{\HXp}{\mathcal{C}^{\Xi'}_{Q,m}}
\newcommand{\PX}{\textsf{P}_{Q,{m}}^\Xi}
\newcommand{\PXp}{\textsf{P}_{Q,{m}}^{\Xi'}}
\newcommand{\pix}{\pi_\Xi}
\newcommand{\JFX}{\operatorname{JF}_\Xi}
\newcommand{\JFXg}{\operatorname{GenJF}_\Xi}
\newcommand{\JFXprime}{\operatorname{JF}_{\Xi^\prime}}
\newcommand{\JFXprimeg}{\operatorname{GenJF}_{\Xi^\prime}}

Let $Q$ be a fixed Dynkin quiver with a minuscule vertex $m$. In this section, we conceptually describe a bijection between isomorphism classes of representations $M\in\cat$ and reverse plane partitions of $\pos$. In Section~\ref{sec:explicitRSK}, we describe the bijection combinatorially. A reader seeking to understand the bijection combinatorially may safely skip to Section~\ref{sec:explicitRSK}.

Let $D^b(Q)$ be the bounded derived category of $Q$. Now, let $\Xi$ be a quiver derived equivalent to $Q$, together with
an identification of $D^b(Q)$ with $D^b(\Xi)$.  We can therefore
talk about $\rep \Xi \cap \cat$, which is an additive subcategory of
both $\rep Q$ and of $\rep \Xi$.  We write $\HX$ for $\rep \Xi \cap \cat$.  
We write $\PX$ 
for the poset whose vertices are the indecomposable objects of $\HX$, with the order given by the transitive closure of the arrows.  We write
$\pix$ for the restriction to $\PX$ of the map $\pi$.  

Recall that the indecomposable objects of $D^b(Q)$ are of the form $M[i]$ for $M$ an indecomposable representation of $Q$ and $i\in \mathbb Z$. We say that $\rep\Xi$ is \emph{to the right} of $\rep Q$ if all the indecomposable objects in $\rep\Xi$ are non-negative shifts of indecomposable representations of $Q$.  

For $M\in \HX$, write $\JFXg(M)$ for the Jordan form data of $M$, thought of as a $\Xi$-representation.  We write $\JFXg(M)^i$ for the Jordan
form data at vertex $i$ of $\Xi$, where we number the vertices of 
$\Xi$ so that the indecomposable projective at vertex $i$ of $\Xi$ and the indecomposable projective at vertex $i$ of $Q$ are in the
same $\tau$-orbit. 

\begin{figure}
\begin{center}
\begin{tikzpicture}
\node at (1,0) {$\cdots$};
\node (001[-1]) at (2,1) {$001[-1]$};
\node (111) at (5,1) {$\mathbf{111}$};
\node (100[1]) at (8,1) {$100[1]$};
\node (010[1]) at (11,1) {$010[1]$};
\node (110) at (3.5,0) {$\mathbf{110}$};
\node (011) at (6.5,0) {$\mathbf{011}$};
\node (110[1]) at (9.5,0) {$110[1]$};
\node (100) at (2,-1) {$100$};
\node (010) at (5,-1) {$\mathbf{010}$};
\node (001) at (8,-1) {$001$};
\node (111[1]) at (11,-1) {$111[1]$};
\node at (12,0) {$\cdots$};
\node at (6.5,-1.5) {$(a)$};
\draw[->] (001[-1])--(110);
\draw[->] (110)--(111);
\draw[->] (111)--(011);
\draw[->] (011)--(100[1]);
\draw[->] (100[1])--(110[1]);
\draw[->] (110[1])--(010[1]);
\draw[->] (100)--(110);
\draw[->] (110)--(010);
\draw[->] (010)--(011);
\draw[->] (011)--(001);
\draw[->] (001)--(110[1]);
\draw[->] (110[1])--(111[1]);
\end{tikzpicture}\\
\vspace{.1in}
\begin{tikzpicture}
\node at (1,0) {$\cdots$};
\node (001[-1]) at (2,1) {$001[-1]$};
\node (111) at (5,1) {$\mathbf{011}$};
\node (100[1]) at (8,1) {$100$};
\node (010[1]) at (11,1) {$110[1]$};
\node (110) at (3.5,0) {$\mathbf{010}$};
\node (011) at (6.5,0) {$\mathbf{111}$};
\node (110[1]) at (9.5,0) {$010[1]$};
\node (100) at (2,-1) {$100[-1]$};
\node (010) at (5,-1) {$\mathbf{110}$};
\node (001) at (8,-1) {$001$};
\node (111[1]) at (11,-1) {$011[1]$};
\node at (12,0) {$\cdots$};
\node at (6.5,-1.5) {$(b)$};
\draw[->] (001[-1])--(110);
\draw[->] (110)--(111);
\draw[->] (111)--(011);
\draw[->] (011)--(100[1]);
\draw[->] (100[1])--(110[1]);
\draw[->] (110[1])--(010[1]);
\draw[->] (100)--(110);
\draw[->] (110)--(010);
\draw[->] (010)--(011);
\draw[->] (011)--(001);
\draw[->] (001)--(110[1]);
\draw[->] (110[1])--(111[1]);
\end{tikzpicture}\\
\vspace{.1in}
\begin{tikzpicture}
\node at (1,0) {$\cdots$};
\node (001[-1]) at (2,1) {$011[-1]$};
\node (111) at (5,1) {$\mathbf{001}$};
\node (100[1]) at (8,1) {$110$};
\node (010[1]) at (11,1) {$100[1]$};
\node (110) at (3.5,0) {$010[-1]$};
\node (011) at (6.5,0) {$\mathbf{111}$};
\node (110[1]) at (9.5,0) {$010$};
\node (100) at (2,-1) {$110[-1]$};
\node (010) at (5,-1) {$\mathbf{100}$};
\node (001) at (8,-1) {$011$};
\node (111[1]) at (11,-1) {$001[1]$};
\node at (12,0) {$\cdots$};
\node at (6.5,-1.5) {$(c)$};
\draw[->] (001[-1])--(110);
\draw[->] (110)--(111);
\draw[->] (111)--(011);
\draw[->] (011)--(100[1]);
\draw[->] (100[1])--(110[1]);
\draw[->] (110[1])--(010[1]);
\draw[->] (100)--(110);
\draw[->] (110)--(010);
\draw[->] (010)--(011);
\draw[->] (011)--(001);
\draw[->] (001)--(110[1]);
\draw[->] (110[1])--(111[1]);
\end{tikzpicture}
\end{center}
\caption{In $(a)$, $(b)$, and $(c)$, we show the AR quiver of $D^b(Q^1)$, $D^b(Q^2)$, and $D^b(Q^3)$, respectively, where $Q^1 = 1 \leftarrow 2 \leftarrow 3$, $Q^2 = 1 \rightarrow 2 \leftarrow 3$, and $Q^3 = 1 \leftarrow 2 \rightarrow 3.$ The representations in $\mathcal{C}_{Q^1,2}$, $\mathcal{C}^{Q^2}_{Q^1,2}$, and $\mathcal{C}^{Q^3}_{Q^1,2}$ are in bold.}
\label{Q_xi_figure}
\end{figure}
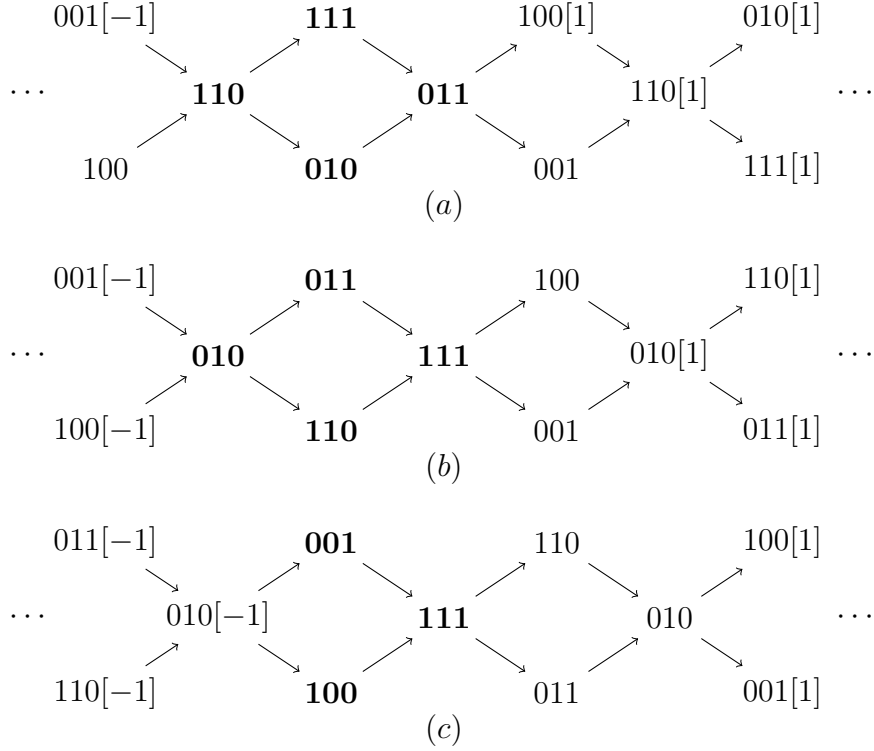

\begin{example}
Let $Q^1 = 1 \leftarrow 2 \leftarrow 3$ and let $Q^3 = 1 \leftarrow 2 \rightarrow 3$ where $D^b(Q^3)$ is identified with $D^b(Q^1)$ 
by identifying objects in the corresponding positions in Figures \ref{Q_xi_figure} (a) and (c).  We see that $\rep Q^3$ is therefore to the right of $\rep Q^1$. 
(More formally, the identification of the two derived categories is done by a composition of reflection functors acting on the derived category: ${\DR}_1^-{\DR}_2^-:D^b(Q^3) \rightarrow D^b(Q^1)$.  
See Section \ref{derived-reflection} for the definition of the reflection functors being used here.)

Choosing vertex $2$ of $Q^1$ to be our minuscule vertex, we see that the indecomposable objects of $\mathcal C_{Q^1,2}$ are $110, 111, 010, 011 \in \rep Q^1.$ In addition, the indecomposable objects of $\mathcal C_{Q^1,2}^{Q^3}$ are $111, 010, 011 \in \rep Q^1$. By the identification of $D^b(Q^3)$ with  $D^b(Q^1)$, these indecomposable objects are identified with $001$, $100$, and $111$, respectively.
\end{example}

We need the following combinatorial lemma.  

\begin{lemma} \label{comb-alt} Let $i$ be a source in 
$\Xi$, and let $\Xi'=\sigma_i(\Xi)$.  Identify $\rep \Xi'$ with 
the full subcategory of $D^b(Q)$ which has the same indecomposable objects as 
$\rep \Xi$ except for the indecomposable projective $S_i'$ of $\rep \Xi'$ and the
indecomposable injective $S_i$ of $\rep \Xi$.  Let $\rho$ be a reverse
plane partition on $\PX$.  Let $\lambda$ be the partition whose parts are the entries of $\rho$ from
$\pi^{-1}(i)$, and let $\mu$ be the partition whose parts are the
entries of $\rho$ from $\pi^{-1}(j)$ for all $j$ adjacent to $i$.  

Then we have the following two statements:
\begin{enumerate} \item If $S'_i\in\cat$, then $\lambda$ and $\mu$ are 1-interlaced.  \item If $S_i'\not\in\cat$, then $\lambda$ and $\mu$ are 0-interlaced.\end{enumerate}
\end{lemma}

\begin{proof}
The condition that between two consecutive parts of $\lambda$ there are two parts of $\mu$ follows from the fact that $\pos$ is two-neighbourly: 
the parts of $\mu$ between two parts of $\lambda$ are taken to be those
corresponding to the two elements of $\pos$ in fibres adjacent to $i$ and in the interval between the
corresponding elements in the fibre of $i$. By Lemma \ref{lem:2neighconverse}, there is at most one part of 
$\mu$ corresponding to an element of $\pos$ above the maximum
element in $\pi^{-1}(i)$.
It remains to verify the interlacing condition involving the elements
of $\mu$ that are required to be greater than the largest part of $\lambda$.

{\it Proof of (1): $S'_i\in\cat$}.  Let $\x$ be the corresponding element of $\pos$. {Note that $\x$ is a minimal element of $\PXp$.}

If $\x$ is not the maximal element of $\pi^{-1}(i)$, then the two-neighbourly condition again implies that
there are two entries in $\mu$ greater than the largest part of $\lambda$.  This establishes that $\lambda$ and $\mu$ are 1-interlaced in this case.  

If $\x$ is the maximal element of $\pi^{-1}(i)$, then $\lambda$ is the
empty partition of 0.  There can only be at most one element of $\pi^{-1}(j)$ {in $\PXp$} with $j$ adjacent to $i$ by Lemma \ref{lem:2neighconverse}.  Since there is at most one element of 
$\pi^{-1}(j)$ with $j$ adjacent to $i$, we know that $\mu$ has at most
one non-zero element, and therefore $\lambda$ and $\mu$ 1-interlace in
this case as well.

{\it Proof of (2): $S_i'\not\in\cat$}.  We now split into three cases, 
depending on whether $\pi^{-1}(i)\cap \PX$ is:
\begin{enumerate} \item[(i)] empty,
\item[(ii)] equal to all of $\pi^{-1}(i)$, or
\item[(iii)] neither empty nor equal to all of $\pi^{-1}(i)$.
\end{enumerate}
Consider first case (i).  In this case $\lambda$ is the empty partition.
Let $\x'$ be the maximal element of $\pos\cap\pi^{-1}(i)$.  If $\x'$ is
the maximum element of $\pos$, $\mu$ is the empty partition and the
desired result holds, so assume otherwise.  
By Lemma \ref{four-two}, the indecomposable object
$M_{\x'}$ {has an immediate successor $M_{\z}$ with
$\z\in \pos$.}  On the other hand, by Lemma \ref{lem:2neighconverse}, there
is at most one element in an adjacent fibre to $\x'$ and above it in $\pos$.
Therefore $\z$ is this unique element.  But $\z$ is not in 
$\PX$, since $S_i'$, which is a successor of $\z$, is not in $\PX$. Thus $\mu$ is again the empty partition. 

In cases (ii) and (iii), let $\y$ 
be the minimum element of $\PX\cap \pi^{-1}(i)$.  
By Lemma \ref{four-two}, $\y$ is immediately preceded in the AR quiver
by 
some $\z'$ in an adjacent fibre to $i$.  Note that $\z'$ is contained
in $\PX$ since $\z'$ is a successor of $S_i'$. 
This therefore provides the remaining needed element of 
$\mu$.  We must check that there is no other element of $\mu$.  

In case (ii), let $\x'$ be the maximum element of $\pi^{-1}(i)$ which
is not in $\PX$. Again by Lemma \ref{four-two}, 
$M_{\x'}$ is directly followed in the AR quiver by some $M_{\z}$ with
$\z\in \pos$. But $\z\not\in \PX$, and $\z'$ and $\z$ are the only
two elements of adjacent fibres to $i$ between $\x$ and $\x'$ by the
two-neighbourly property.  We have therefore accounted for all the 
elements of $\mu$.  

In case (iii), there cannot be another element in a fibre adjacent to
$i$ and below $\y$, by the dual of Lemma \ref{lem:2neighconverse}.  We are 
done in this case as well.
\end{proof}

We say that the Jordan form data for a representation of $\Xi$ \emph{fits in} 
the poset $\PX$ if: \begin{itemize}
\item The number of entries of $\JFXg(M)^i$ is at most $|\pix^{-1}(i)|$.
\item If we define a filling of $\PX$ by putting the entries of 
$\JFXg(M)^i$ into $\pix^{-1}(i)$ in decreasing order going up $\PX$,
padding with zeros if necessary, then the filling defines a reverse plane
partition.
\end{itemize}
If $\JFg_\Xi(M)$ fits in the poset $\PX$, then we denote the above reverse plane partition by {$\rho_{Q,m}^{\Xi}(M)$.} When $\rep Q=\rep \Xi$ (as subcategories of $D^b(Q)$), we denote the reverse plane partition by ${\rho_{Q,m}(M)}$.

\begin{theorem}\label{thm_rsk} Suppose that $\rep \Xi$ is to the right of $\rep Q$.  
\begin{enumerate}
\item  \label{pta} Let $M\in \HX$.  Then $\JFXg(M)$ fits in $\PX$.
\item \label {ptb} The map $M \mapsto {\rho_{Q,m}^{\Xi}(M)}$ from isomorphism classes of objects in $\HX$ to 
reverse plane partitions is a bijection.
\item \label {ptc} The inverse map is given by taking the generic representation compatible with nilpotent transformations having the given Jordan form.
\end{enumerate}
\end{theorem}

We illustrate the arguments of the following proof in Example~\ref{ex:rskthm}.

\begin{proof}  The proof is by induction on the position of $\rep \Xi$ in
$D^b(Q)$.  If $\rep \Xi$ is far enough to the right, then $\HX$ is zero
and the poset $\PX$ is empty, and the claims are vacuously true.  

Now, suppose that the statements are true for $\Xi$, and let $i$ be a source of $\Xi$.  Let $\Xi'$ be the result of reflecting $\Xi$ at 
$i$, and consider $\rep \Xi'$ as embedded naturally in $D^b(Q)$: that is to say,
the indecomposable objects in the representation categories of $\Xi'$ and of $\Xi$ coincide, except for $S_i'$, the simple projective at $i$ of
$\rep \Xi'$ and $S_i$, the simple injective at $i$ of $ \rep \Xi$.  Suppose further
that $\rep \Xi'$ is still to the right of $\rep Q$.  This means in particular 
that $S_i$ is not in $\cat$. 

{Let $\lambda=\JFXg(M)^i$, and let $\mu$ be the partition whose parts are the collection of all the parts of $\JFXg(M)^j$ for vertices $j$ adjacent to $i$.}

Let $M'\in \HXp$.
There are now two different possibilities, depending on whether
$S'_i$ is in $\cat$.

{\it Case I: $S'_i$ is not in $\cat$.}  Let $M=R_i^+(M')$.  By 
the induction hypothesis, $\JFXg(M)$ fits in $\PX=\PXp$.  Also by 
the induction hypothesis, $M$ is the generic representation compatible with
nilpotent linear transformations with Jordan form data given by $\JFXg(M)$.

Since $S_i'$ is not in $\cat$, {by Lemma \ref{comb-alt}, $\lambda$ and $\mu$ are 0-interlaced,} so the hypotheses of Theorem \ref{th-ref-nil} hold with $t=0$. 
Thus $M'=R_i^-(M)$ has Jordan form data given 
by toggling at $i$ the filling of $\PX$ determined by $\JFXg(M)$.  Toggling can potentially change a part equal to zero of $\JFXg(M)$ into a non-zero part, but this only happens if the smallest non-zero part of $\JFXg(M)$ has two neighbouring non-zero parts above it in $\PX$.  In this case, Lemma \ref{lem:2neighconverse} applies to show that there was an actual 0 entry in $\pi^{-1}(i)$ in $\rho_{Q,m}^{\Xi}(M)$, which gives us space to fill in the new non-zero entry.   Thus $\JFg_\Xi(M')$ fits in $\PX$, but $\PXp=\PX$.  This establishes (\ref{pta}). 

(\ref{ptb}) This follows from the induction hypothesis together with the
fact that toggling is itself a bijection.  

(\ref{ptc}) This follows from Theorem \ref{th-ref-canon}.

{\it Case II: $S_i'$ is in $\cat$.} 
Let $\hat M'$ be $M'$ with any summands of $S'_i$ removed.  Let $\hat M=R^+_i(\hat M')$.  
By the induction hypothesis, $\JFXg(\hat M)$ fits in $\PX$.  Also by 
the induction hypothesis, $\hat M$ is the generic representation compatible with
nilpotent linear transformations with Jordan forms given by $\JFXg(\hat M)$.  
Since $S'_i$ is in $\cat$, {by Lemma \ref{comb-alt}, $\lambda$ and $\mu$ are 1-interlaced,} and the hypotheses of 
Theorem \ref{th-ref-nil} hold with $t=1$.
This theorem then implies that  $\JFg(\hat M)^i$ has as its largest 
part the largest of the parts of the Jordan blocks of the neighbours, and that the 
remaining parts agree with the results of 
toggling at $i$.  

We now consider the difference between 
$\JFg(\hat M')$ and $\JFg(M')$.  The difference
is confined to the partition corresponding to
vertex $i$, since $S'_i$ is only supported over
vertex $i$.  Since $S'_i$ is simple projective,
it admits a non-zero morphism to any other
indecomposable with support over $i$.  Therefore, $\JFg(M')$ is obtained from
$\JFg(\hat M')$ by adding the number of summands of $S'_i$ in $M'$ to the largest part
of $\JFg(\hat M')^i$.  We note that, compared to $\P_{Q,m}^{\Xi}$, 
$\P_{Q,m}^{\Xi'}$ contains an extra box, which accommodates the largest
part of $\JFg(M')$.  The result of toggling the entries of $\rho_{Q,m}^{\Xi}(M)$ in $\pi^{-1}(i)$ still fits in the boxes that were present
in $\P_{Q,m}^{\Xi}$, by the same argument as in Case I, using Lemma \ref{lem:2neighconverse}.
The Jordan
data for $M'$ therefore forms a reverse plane partition
for $\PX$, which establishes (1).   

For (2), we see that we can clearly recover the multiplicity of $S_i'$ from $\rho_{Q,m}^{\Xi'}(M')$,
and from
$\rho_{Q,m}^{\Xi}(\hat M)$ we can determine $\hat M$ by induction. Further, it is clear that any reverse plane partition corresponds to some representation.

(3) follows again from Theorem \ref{th-ref-canon}.\end{proof}

\begin{example}\label{ex:rskthm}
Let $Q^1 = 1 \leftarrow 2 \leftarrow 3$ with $m=2$, $Q^2 = 1 \rightarrow 2 \leftarrow 3$, and $Q^3 = 1\leftarrow 2 \rightarrow 3$ as in Figure~\ref{Q_xi_figure}. We first walk through an example of the inductive argument we use in Case II of the proof of Theorem~\ref{thm_rsk}. Here the role of $\Xi^\prime$ and $\Xi$ in the proof of Theorem~\ref{thm_rsk} are played by $Q^2$ and $Q^3$ from Figure~\ref{Q_xi_figure}, respectively. We see that $Q^3$ is to the right of $Q^2$ and $S^\prime_2 \in \rep Q^3$ is in $\mathcal{C}_{Q^1,2}$.
Let $M'\in\mathcal{C}^{Q^2}_{Q^1,2}$ be $M'=010^3\oplus 011^2\oplus 110 \oplus 111^3$. Then $\hat{M}'=011^2\oplus 110 \oplus 111^3$, and $\hat M=R^+_2(\hat M')\in\mathcal{C}^{Q^3}_{Q^1,2}$ is $\hat M=001^2\oplus 100 \oplus 111^3$. By induction, $\hat M$ corresponds to a reverse plane partition on $\textsf{P}_{Q^1,2}^{Q^3}$ determined by its Jordan form data $\JFg(\hat M)=((4),(3),(5))$. Then by Theorem~\ref{th-ref-nil}, $\JFg(\hat M')=((4),(5,1),(5))$. Following the proof, we add the multiplicity of $S_2'$ in $M'$ to the largest part of $\JFg(\hat M')^2$ to obtain $\JFg(M')=((4),(8,1),(5))$.

We now compute an example of Case I of the proof. Here the role of $\Xi^\prime$ and $\Xi$ in the proof of Theorem~\ref{thm_rsk} are played by $Q^1$ and $Q^2$ from Figure~\ref{Q_xi_figure}, respectively, where we see that $Q^2$ is to the right of $Q^1$ and $S_1^\prime\in\rep Q^2$ is not in $\mathcal{C}_{Q^1,2}$. Let 
$M_2'\in\mathcal{C}_{Q^1,2}$ be $M_2'=110\oplus 111^3\oplus 010^2\oplus 011^3$. Then $M_2=R_{1}^+(M_2')=010\oplus 011^3\oplus 110^2\oplus 111^3$ and $\JFg(M_2)=((5),(7,2),(6))$. We obtain $\JFg(M'_2)=((4),(7,2),(6))$ by toggling $\JFg(M_2)$ at vertex 1. \end{example}

\begin{proof}[Proof of {Theorems \ref{Thm_can_jordan_recov}} and~\ref{th-rpp}]
The theorems follow from Theorem~\ref{thm_rsk}, with $\rep(\Xi)$ equal to $\rep(Q)$.
\end{proof}

Observe that any order filter of $\pos$ is of the form $\PX$ for some $\Xi$.  
Therefore, the following corollary is a generalization of Corollary~\ref{thm_gen_fctn} for any order filter of $\pos$.

\begin{corollary}\label{cor_ord_filt_gn_fctn} For $Q$ a Dynkin quiver and $m$ a minuscule vertex, we have
\[\sum_{\rho\in \RPP(\PX)} \prod_{i=1}^n q_i^{|\rho_{i}|} = \sum_{X\in \HX} \prod_{i=1}^n q_i^{\textbf{dim}(X)_i}=\prod_{\textsf{u}\in \PX} \frac{1}{1- \prod_{i=1}^nq_i^{\textbf{dim}(M_\textsf{u})_i}},\]
where we write $|\rho_{i}|$ for the sum of the values $\rho(\textsf{x})$ over all $\textsf{x} \in \pi^{-1}(i)$. The second sum is over isomorphism classes of representations in $\HX$. In the third sum, $M_\textsf{u}\in\HX$ is the indecomposable representation of $\Xi$ corresponding to $\textsf{u}\in \PX$. 
 \end{corollary}

\begin{proof}The first equality is from Theorem \ref{thm_rsk}, while the second comes from the fact that any representation in $\HX$ can be decomposed in a unique way as a sum of some number of copies of the representations $M_\textsf{u}$ for $\textsf{u}\in \PX$.
\end{proof}

\begin{remark}
 If $\Delta$ is a Dynkin diagram that is not simply-laced, there is a distinction between minuscule and cominuscule vertices of $\Delta$. If a vertex $m$ is minuscule, then there is an associated minuscule poset. This poset appears in the list previously discussed, but is equipped with a different heap structure. If $m$ is cominuscule, there is a simply-laced Dynkin diagram $\overline \Delta$ with
an automorphism $\phi$ of $\overline \Delta$, the orbits of whose vertices correspond to vertices of $\Delta$, with the orbit corresponding to $m$ being a single vertex 
$\overline m$, with $\overline m$ minuscule (and cominuscule) for $\overline \Delta$. We say that $(\Delta,m)$ unfolds to $(\overline \Delta,\overline m)$. 

The Dynkin diagram automorphism of $\overline \Delta$ extends to an action on reverse plane partitions of $(\overline \Delta,\overline m)$. Reverse plane partitions on $\P_{\overline \Delta,\overline m}$ which are fixed under $\phi$ can be identified with reverse plane partitions on $\P_{\Delta,m}$. In this way, reverse plane partitions associated to the cominuscule node of type $C_n$ correspond to reverse plane partitions of type $A_{2n-1}$ symmetric about the main diagonal. Gansner has studied these and obtained the analogue of our generating function identity in Corollary~\ref{cor_ord_filt_gn_fctn} in that setting \cite[Corollary 6.2]{gansner1981hillman}.

Note that reverse plane partitions associated 
to the 
cominuscule node of $C_n$ are the same thing as reverse 
plane partitions associated to the minuscule node of $B_n$,
and the minuscule poset for $B_n$ is isomorphic  to the minuscule poset associated to the "antennae" nodes of 
$D_{n+1}$. Thus, the study reverse plane partitions in the classical types boils down to the study of 
type $A_n$ reverse plane partitions and type $A_{2n-1}$ symmetric reverse plane partitions, together
with the near chains as in Figure \ref{min_posets_ex}(b). 
 \end{remark}


\subsection{Piecewise-linear description of Theorem~\ref{thm_rsk}(\ref{ptb})}\label{sec:explicitRSK}

We now give a more explicit description of $\rho_{Q,m}(-)$. We must first establish a linear order on the indecomposable representations of $Q$. Indeed, choose a linear order on the indecomposable representations of $Q$ compatible with the opposite of the AR quiver order. In other words, we number indecomposables $M_1,\dots, M_N$ from right to left starting with a simple injective. For $1\leq j \leq N$, let $i_j\in Q_0$ be the index of the 
indecomposable projective representation in the same $\tau$-orbit as $M_j$. It follows that $R_{i_N}^+$ can be applied to $\rep Q$, that $R_{i_{N-1}}^+$ can be applied to $\rep \sigma_{i_N} Q$, and so on. The composition of reflection functors $R_{i_1}^+\cdots R_{i_N}^+$ has the property that it takes every representation to 0. 

Conversely, every representation can be built up by adding simple projectives and applying reflection functors in the following way.  Let $M=\bigoplus_{j=1}^N M_j^{c_j}$.  Define $X_0$ to be the zero representation of $Q'=\sigma_{i_1}\dots \sigma_{i_N} (Q)$.  Now, assuming $X_j$ is defined, define $Y_{j+1}=R^-_{i_{j+1}}(X_j)$, and define $X_{j+1}=Y_{j+1}\oplus S_{i_{j+1}}^{c_{j+1}}$.  Then $X_N$ is isomorphic to $M$.

We can calculate $\rho_{Q,m}(M)$ by using this procedure, i.e., by understanding how $\rho_{Q,m}(M)$ changes under reflection and adding simple projectives, as follows.
Let $M=\bigoplus_{j=1}^N M_j^{c_j} \in \cat.$ The proof of Theorem~\ref{thm_rsk} shows that $\RSK$ is obtained by constructing a sequence of fillings of the minuscule poset $\pos$, starting with the zero filling $\rho_0$. These fillings are defined by
$$\begin{array}{ccccc}\rho_k(\textsf{x}) & := & \left\{\begin{array}{lcl} \displaystyle \max_{\textsf{x}\lessdot \textsf{y}}\rho_{k-1}(\textsf{y}) + c_k & : & \text{if $\textsf{x}$ is the element of 
$\pos$ corresponding to $M_k$,} \\
({t}_\textsf{x}\rho_{k-1})(\textsf{x}) & : & \text{if $\textsf{x}$ corresponds to  $\tau^\ell(M_k)$ for some $\ell < 0$, and}\\
\rho_{k-1}(\textsf{x}) & : & \text{otherwise,}
\end{array}\right.\end{array}$$ 
where $\textsf{x}$ is any element of $\pos$. We obtain the following theorem.

\begin{theorem}\label{thm_4_7_alg}
For any $M \in \cat,$ we have that $\RSK = \rho_N.$
\end{theorem}

Note that using this description of the algorithm, the intermediate fillings $\rho_k$ for $k<N$ are not reverse plane partitions of $\pos$. However, by restricting $\rho_k$ to the elements of $\pos$ corresponding to $M_1,\ldots,M_k$, we do obtain a reverse plane partition on the induced subposet of $\pos$ whose elements correspond to $M_1,\ldots,M_k$.

Observe that in the process of constructing $\rho_N$, we never toggle at a minimal element of one of these induced posets. Therefore, the result of the procedure does not depend on whether we think of the entries as being in 
$[0,N]$ for any $N$ sufficiently large, or as being in $\NN$. In either case, the entries will always consist of non-negative integers.
See Figures~\ref{fig:HGexample} and \ref{fig:Pakexample} for examples in type $A$ worked out step-by-step using this explicit description. 

\section{Periodicity}\label{sec:period}
In this section, we study reverse plane partitions on minuscule posets filled with elements of $\NN$ and show that these encode the Jordan form data of certain objects in a quotient of the derived category called the root category. We then use these results to prove that promotion on minuscule posets is periodic with period given by the Coxeter number of the associated Weyl group. 

Throughout Section \ref{sec:period}, 
we assume that $Q$ is a Dynkin quiver with
a chosen minuscule vertex $m$ and with the vertices of $Q$ numbered in
such a way that arrows go from lower-numbered to higher-numbered vertices.

\subsection{Reflection functors in the derived category}\label{derived-reflection}

Let $Q$ be a quiver, and let $k$ be a source or sink of $Q$.  There are reflection
functors that provide an equivalence between the derived categories of $\rep Q$ and $\rep \sigma_k(Q)$.  They are closely related to the reflection functors defined previously on categories of representations. We recall the definition of these reflection functors now.

Let $k$ be a sink of $Q$ and $M$ an indecomposable representation, define $\DR^+_k: D^b(Q) \to D^b(\sigma_k(Q))$ by 
$$\DR^+_k(M[i]) := \begin{cases} 
(R^+_kM)[i] & \textrm { if } M \not \simeq S_k\\ 
S_k[i-1] & \textrm{ if } M\simeq S_k.\end{cases} $$ Similarly, let $k$ be a source of $Q$ and $M$ an indecomposable representation, define $\DR^-_k: D^b(Q) \to {D}^b(\sigma_k(Q))$ by 
$$\DR^-_k(M[i]) := \begin{cases} ({R}^-_kM)[i] & \textrm{ if } M \not \simeq S_k,\\
S_k[i+1] & \textrm{ if } M\simeq S_k. \end{cases}$$ 

We will study the behaviour of the \textit{Coxeter functor}, which is defined as $\text{cox} := \widetilde{{R}}^-_{n}\cdots \widetilde{{R}}^-_{1}$. {We recall the following well-known lemma.}

\begin{lemma}\label{lemma_cox_equals_tau} For $Q$ Dynkin and 
$M\in D^b(Q)$, the inverse Auslander--Reiten translation of $M$ is
isomorphic to 
$\text{cox}(M)$.
\end{lemma}

\begin{proof}
It is well known that for Dynkin quivers the Auslander--Reiten translation $\tau$ may be written $\tau = R^+_{1}\cdots R^+_{n}$ (See, for instance, 
\cite[Proposition 5.4]{gabriel1980auslander}.  In general, as explained by Gabriel, there is an issue of signs in the maps defining the representation, but since $Q$ is Dynkin, this does not affect the representations up to isomorphism.)

We assume that $M$ is indecomposable and that $Q$ is not the
quiver with a single vertex. Then there is a category of representations of a quiver derived equivalent to $\rep Q$ in which $M$ is not projective, so the Auslander--Reiten translation of $M$ agrees with the Auslander--Reiten translation in the derived category.\end{proof}

\subsection{The root category}
Let $\mathcal{D} = D^b(Q)$ where $Q$ is an acyclic quiver, and let $F: \mathcal{D}\to \mathcal{D}$ be a triangle functor.  We assume that $F$ also satisfies the following:
\begin{enumerate}
\item[1)] For each indecomposable representation $V$ of $Q$, only a finite number of the objects $F^nV$, with $n \in \mathbb{Z}$, are indecomposable representations of $Q$, and 
\item[2)] there is some $N \in \mathbb{N}$ such that the set
\[
\{V[n] \mid V \text{ an indecomposable representation of $Q$}, n \in [-N,N]\}
\]
contains a set of representatives of the orbits of $F$ on the indecomposable objects of $\mathcal{D}$.
\end{enumerate}
We define the \textit{orbit category} $\mathcal{D}/F$ to be the category whose objects are the objects of $\mathcal{D}$ 
and whose morphisms from ${X}$ to ${Y}$ are {given by}
\[
\bigoplus_{n \in \mathbb{Z}}\text{Hom}_\mathcal{D}(X, F^nY).
\] 
Keller proved that the category $\mathcal{D}/F$ is triangulated and that the projection functor $p: \mathcal{D} \to \mathcal{D}/F$ is triangulated \cite[Theorem 1]{keller2005triangulated}. Furthermore, the shift functor in $\mathcal{D}/F$ is induced by the shift functor in $\mathcal{D}$. We therefore denote both by $[1]$.

Now, we return to the case when $Q$ is a Dynkin quiver and $m$ is a minuscule vertex of $Q$. We define the orbit category $\mathcal R_Q = D^b(Q)/[2]$. Observe that the triangle functor $[2]$ satisfies the two properties stated in the previous paragraph, and {the category $\mathcal R_Q$} is therefore triangulated by Keller's theorem. 

The definition of morphisms in $\mathcal{R}_Q$ and the fact that $\rep Q$ is hereditary imply that any indecomposable object $X \in D^b(Q)$ is isomorphic to $X[2n]$ for any $n\in \mathbb{Z}$. 

The Grothendieck group $K_0(\mathcal R_Q)$ is isomorphic to $\mathbb Z^n$; the classes of the simple objects in $\rep Q$ form a basis for it.  The map sending objects in $\mathcal R_Q$ to their classes in the Grothendieck group defines a bijection from the indecomposable objects to the roots in the root system corresponding to $Q$.  For this reason, $\mathcal R_Q$ is referred to as the \textit{root category}.  The positive roots correspond to the indecomposables in $\rep Q$, and the negative roots to the indecomposables in $\rep Q[1]$.

Next, the reflection functors defined on $D^b(Q)$ are well-defined on objects of $\mathcal{R}_Q$. This follows from the fact that $\widetilde{R}^+_i(X[2]) \simeq (\widetilde{R}^+_i(X))[2]$ in {$D^b(Q)$} 
for any $X \in {D^b(Q)}$, which is easily verified. On the level of the Grothendieck group, reflection functors act like simple reflections in the Weyl group corresponding to the root system.

We now consider the action of $\text{cox}$ on $\mathcal{R}_Q$. Let $h$ denote the \textit{Coxeter number} of $Q$: the order of the product of all of the simple reflections in the corresponding Coxeter group, {taken in any order}.

\newcommand{\cox}{\textrm{cox}}

\begin{lemma}\label{order-cox}
For any object $M\in\mathcal R_Q$, we have that $\cox^h(M)\simeq M$.  Conversely, if $M$ is indecomposable, and $0<i<h$, then $\cox^i(M)\not\simeq M$.  \end{lemma}

\begin{proof}
The functor $\cox$ acts on the Grothendieck group of $\mathcal R_Q$ by
$s_n\dots s_1$.  {Since the order of $s_n\dots s_1$ as an element of $W$ is $h$, it} follows that
$\cox^h$ sends an indecomposable object $M$ to an indecomposable object with the same class in the Grothendieck group as $M$, but such an object is necessarily isomorphic to $M$.  This 
establishes the first claim for indecomposable objects, and thus for all objects.  

The second claim follows from the fact that orbits in the set of roots under
the action of a Coxeter element are all of size $h$, see \cite[Exercise V.6.1]{Bo}. \end{proof}

\subsection{Reverse plane partitions for objects in the root category}

There is an automorphism of the Dynkin diagram induced by the action of the
longest element of the Weyl group.  We denote it $\psi$.  Concretely, it
is the obvious symmetry of the Dynkin diagram in types $A_n$, $D_n$ with $n$
odd, and $E_6$.  Otherwise it is the identity.  This symmetry plays an important r\^{o}le both in minuscule posets and in representation theory.

\begin{lemma} \label{ant-lemma}
There is a unique antiautomorphism of $\Ant:\pos\rightarrow \pos$ such that
$\pi(\Ant(\x))=\psi(\pi(\x))$. (It is easy to see how $\Ant$ is defined on the minuscule posets appearing in Figure~\ref{min_posets_ex}.)  \end{lemma}

\begin{proof} Clearly there is at most one such map, since it must send 
$\pi^{-1}(i)$ to $\pi^{-1}(\psi(i))$ while reversing the order.

Write $\pos^\op$ for the dual poset of $\pos$, and define $\pi^\op(\x)=\psi(\pi(x))$.
Now $(\pos^\op,\pi^\op)$ is again a maximal neighbourly heap that is also two neighbourly, so by the classification, it is the two-neighbourly heap corresponding to some minuscule vertex of a Dynkin diagram, and it clearly must be $(\pos,\pi)$.
The isomorphism between $(\pos,\pi)$ and $(\pos^\op,\pi^\op)$ defines 
$\Ant$.
\end{proof}

\begin{lemma} {The modules} $P_i$ and $I_{\psi(i)}$ are in the same $\tau$-orbit in
$\rep Q$.\end{lemma}\label{inj-proj}
\begin{proof} There is a sequence of reflection functors corresponding to a factorization of the longest element $w_0$ in the Weyl group and with the property that they
send $\rep Q$ to $\rep Q[1]$.  This sends $P_i$ to some  projective object $X$ of $\rep Q[1]$ in the same $\tau$-orbit as $P_i$. Since reflection functors act on the Grothendieck group by simple reflections, $[X]=w_0([P_i])$, implying that $X\simeq P_{\psi(i)}[1]$.  Since $\tau P_{\psi(i)}[1]
\simeq I_{\psi(i)}$, the desired result follows.\end{proof}

Let $\Xi$ be a reorientation of $Q$.  Then $\mathcal R_\Xi$ is
equivalent to $\mathcal R_Q$.  We fix an equivalence.  The indecomposable
objects of $\mathcal R_Q$ are thereby divided into two: those in $\mathcal R_\Xi^\even$ and those in $\mathcal R_\Xi^\odd$.
  
Let $X\in \cat$.  It is thereby divided as $X_\Xi^\even \oplus X_\Xi^\odd$.
As we discussed in the introduction, if we start from a nilpotent endomorphism of $X$, it induces a nilpotent endomorphism of $H^i(X)$ with respect to the $t$-structure induced by $\Xi$, i.e., of  
$X_\Xi^\even$ and $X_\Xi^\odd$. We then consider the Jordan data associated to $\JFg(X_\Xi^\even)$ and $\JFg(X_\Xi^\odd)$.  We will insert
$\JFg(X_\Xi^\even)$ into an order filter in $\pos$, which we denote $\PP_\Xi^\even$, exactly as we did in Section
\ref{sec:RSK}.  We will insert $\JFg(X_\Xi^\odd)$ into the complementary order ideal,
denoted $\PP_\Xi^\odd$,
after replacing each Jordan block size $i$ by $\infty - i$.  

We have to define $\PP_\Xi^\even$ and $\PP_\Xi^\odd$.  There is something slightly 
confusing which happens.  

$\cat$ is divided in two, into $\mathcal C_\Xi^\even$ and $\mathcal C_\Xi^\odd$. The AR quiver of the bounded derived category,
restricted to $\cat$, is acyclic.  As usual, we think of the arrows
as going from left to right.  
There are two possibilities: the elements of $\mathcal C_\Xi^\even$ are to the right of
the elements of $\mathcal C_\Xi^\odd$, or vice versa.  In the former case, 
the elements of $\mathcal C_\Xi^\even$ form an order ideal in $\PP$ (thought of as the
AR quiver of $\mathcal R_Q$ restricted to $\cat$), and we simply define 
 $\PP_\Xi^\even$ to be that order ideal.  In this case, we similarly define 
 $\PP_\Xi^\odd$ to be the complementary order filter, which consists of the elements of 
 $\PP$ which correspond to objects from $\mathcal C_\Xi^\odd$. 

However, in the case that the elements of $\mathcal C_\Xi^\even$ are to the left of the
elements of $\mathcal C_\Xi^\odd$, something slightly unexpected happens.  $\mathcal C_\Xi^\even$ defines an order ideal of $\PP$, while we want $\PP_\Xi^\even$ to be
an order filter.  We therefore define $\PP_\Xi^\even=\Ant(\mathcal C_\Xi^\even)$ and
$\PP_\Xi^\odd=\Ant(\mathcal C_\Xi^\odd)$.

The following proposition follows quite directly from Theorem \ref{thm_rsk}.

\begin{proposition}\label{derived-fits} 
$\JFg(X_\Xi^\even)$ fits into $\PP_\Xi^\even$ and $\JFg(X_\Xi^\odd)$ fits into
$\PP_\Xi^\odd$ after replacing each $i$ by $\infty - i$.    
\end{proposition}
  
\begin{proof}  We consider first the case that $\mathcal C_\Xi^\even$ is to the right
of $\mathcal C_\Xi^\odd$.  The fact that $\JFg(X_\Xi^\even)$ fits into $\PP_\Xi^\even$
  is a direct application of Theorem \ref{thm_rsk}.  The claim for the odd
  part follows by considering the dual algebra.

  Now consider the case that $\mathcal C_\Xi^\odd$ is to the right of $\mathcal C_\Xi^\even$.
  To analyze this, we consider the setup of Theorem \ref{thm_rsk} and imagine that $\Xi$ started to the left of $Q$ instead of to the right of $Q$.  A very similar argument applies.  The important difference is that when we reflect $\Xi$ at $i$, adding a new simple module, that module is injective rather than
  projective.  By Lemma \ref{inj-proj}, it is therefore the module $S_{\psi(i)}$. This is accounted for by the fact that $\Ant$ swaps the fibres 
  $\pi^{-1}(i)$ and $\pi^{-1}(\psi(i))$.

  Again, the odd part follows in the same way as the even part by considering the dual algebra. \end{proof}

By Proposition \ref{derived-fits}, given $\Xi$ and given $X\in \cat$, we can fit the Jordan form data for $X_\Xi^\odd$ and $X_\Xi^\even$ into $\pos$.  If we interpret $\infty - i$ as greater than $j$ for any natural numbers $i$ and $j$, the 
result is a reverse plane partition with entries in $\NN$.  
We denote this reverse plane partition by 
  $\rho_{Q,m}^{\Xi}(X)$.  

We now establish the converse mentioned in the introduction, Proposition \ref{intro-enough-X}, which says that for any reverse plane partition $\rho$ with entries in $\NN$, there exists a choice of $\Xi$ and $X$ such that $\rho=\rho_{Q,m}^\Xi(X).$

\begin{proof}[Proof of Proposition \ref{intro-enough-X}] Define $\sf P^\even$ to be the order filter in $\pos$ where the value of $\rho$ is in $\mathbb N$, and define $\sf P^\odd$ to be the complementary order ideal.
Choose $\rep\Xi\subset \mathcal R_Q$ so that $\rep\Xi\cap \cat$ corresponds to the elements of $\sf P^\even$.  Theorem \ref{thm_rsk} now says that there is a bijection between isomorphism classes of objects in
$\rep \Xi\cap \cat$ and reverse plane partitions on $\sf P^\even$.  Choose $X^\even$ to be
an object in $\rep\Xi\cap \cat$ corresponding to the reverse plane partition $\rho|_{\sf P^\even}$.  Similarly, choose $X^\odd[1] \in \rep\Xi[1]\cap \cat$ so that $\JFg(X^\odd)$, when inserted into $\sf P^\odd$ as in Proposition \ref{derived-fits}, yields $\rho|_{\sf P^\odd}$.  Let $X=X^\even\oplus X^\odd[1]$.  It follows that $\rho=\rho_{Q,m}^\Xi(X)$.\end{proof}

\subsection{Reflection functors and toggling in the derived category}\label{refl_functors_and_toggling_D}

We continue the same setup as in the previous section: we are considering $\mathcal R_Q$, and $\cat$ as a subcategory of it, with $X$ an object in $\cat$.  We have also fixed $\Xi$ a reorientation of $Q$ and an identification of $\rep \Xi$ as a subcategory of $\mathcal R_Q$.

Suppose that $i$ is a source
of $\Xi$, and let $\Xi'=\sigma_i(\Xi)$.   We identify
  $\rep(\Xi')$ as the subcategory of $\mathcal R_Q$ which coincides with
  $\rep(\Xi)$ except for the simple injective $S_i$ in $\rep(\Xi)$ and the
  simple projective $S'_i$ in $\rep(\Xi')$.  

We now prove Theorem \ref{tog-ref}, which asserts that $\rho_{Q,m}^{\Xi'}(X)=t_i \rho_{Q,m}^{\Xi}(X)$.

\begin{proof}[Proof of Theorem \ref{tog-ref}]
We divide into three cases, depending on whether $S_i$ is in $\cat$, $S_i'$ is in $\cat$, or neither.  (Note that since $S_i=S'_i[1]$, it is not possible for both
  $S_i$ and $S'_i$ to be in $\cat$.)

{\it Case I: Neither $S_i$ nor $S_i'$ is in $\cat$.}  
In this case, $\P^\even_\Xi$ and $\P^\even_{\Xi'}$ agree.  Since neither
$S_i$ nor $S_i'$ is in $\cat$, we have that $X^\even_{\Xi'}\simeq R^-_i(X^\even_{\Xi})$.  As shown in the proof of Theorem \ref{thm_rsk}, 
$\rho_{Q,m}^{\Xi'}(X^\even_{\Xi'})=t_i\rho_{Q,m}^\Xi(X^\even_{\Xi})$.
The same statements also apply to the odd parts.  

The further observation that is needed is that in this case the action of $t_i$ on $\rho_{Q,m}^\Xi(X)$ agrees
with the action of $t_i$ on each of the two parts.  This follows from Case II in the proof of Lemma \ref{comb-alt}.  

{\it Case II: $S_i'$ is in $\cat$.}  This hypothesis implies that 
$\mathcal C^\even_{\Xi}$ is to the right of $\mathcal C^\odd_{\Xi}$,
and similarly for $\Xi'$.  Let $p$ be the multiplicity of $S'_i$ in
$X$.  

{Observe that the poset} $\P^\even_{\Xi'}$ contains {exactly} one more element than $\P^\even_{\Xi}$.
Let this element be $\x$.  

{We have that $X^\even_{\Xi'}\simeq R^-_i(X^\even_{\Xi}) \oplus (S_i')^p$.}  
As shown in the proof of Theorem \ref{thm_rsk}, $\rho_{Q,m}^{\Xi'}(X^\even_{\Xi'})(\z) = (t_i\rho_{Q,m}^\Xi(X^\even_\Xi))(\z)$ for all
$\z\in \P_{\Xi'}^\even$ with $\z\ne \x$.  The same argument shows that 
$\rho_{Q,m}^{\Xi'}(X^\odd_{\Xi'})(\z) = (t_i\rho_{Q,m}^\Xi(X^\odd_\Xi))(\z)$ for all 
$\z \in \P_{\Xi'}^\odd$, and it is clear that the action of $t_i$ on
$\rho_{Q,m}^\Xi(X)$ agrees with its action on the two parts 
separately except perhaps at $\x$.  It therefore only remains to 
establish the result at $\x$.  

Consider the inductive procedure which calculates $\rho_{Q,m}^\Xi(X^\odd_\Xi)$ (as in the proof of Theorem \ref{thm_rsk}).  From that perspective, the element $\x$ corresponds to the entry of the 
reverse plane partition which has been added last.  We conclude that 
$$\rho_{Q,m}^\Xi(X)(\x)=
p+\max_{\y\lessdot \x} \rho_{Q,m}^{\Xi}(X)(\y).$$  
It follows that 
$$(t_i\rho_{Q,m}^\Xi(X))(\x)= p+\min_{\z\gtrdot \x}\rho_{Q,m}^{\Xi}(X)(\z).$$
This agrees with the results of the inductive procedure which 
calculates $$\rho_{Q,m}^{\Xi'}(X)(\x)=\rho_{Q,m}^{\Xi'}(X^\even_{\Xi'})(\x).$$  The result is established.

{\it Case III: $S_i$ is in $\cat$.} The result in this case follows by the same argument as Case II.
\end{proof}

\begin{example} Let $Q$ be the quiver
$1 \leftarrow 2 \leftarrow 3$, and let $m=2$.
Suppose we start with the representation 
$X=  110^4 \oplus 010^3 \oplus 111^2 \oplus 011$ in $\cat$.
Let $\Xi=Q$, and identify the representations of 
$\Xi$ with the representations of $Q$ contained in
$\mathcal R_Q$. We calculate $\rho_{Q,m}^\Xi(X)$ to obtain the first reverse plane partition below (where the rows of the reverse plane partition are numbered from top to bottom). 
Successively toggling at vertices 3, 2, and 1 of $Q$, we obtain the succeeding sequence of reverse plane partitions. 

$$ \begin{tikzpicture}
\node (a) at (0,0) {8};
\node (b) at (1,1) {6};
\node (c) at (1,-1) {3};
\node (d) at (2,0) {2};
\draw[->] (a) -- (b);
\draw[->] (a) -- (c);
\draw[->] (b)--(d);
\draw[->] (c)--(d);
\end{tikzpicture}
\qquad \begin{tikzpicture}
\node (a) at (0,0) {8};
\node (b) at (1,1) {6};
\node (c) at (1,-1) {7};
\node (d) at (2,0) {2};
\draw[->] (a) -- (b);
\draw[->] (a) -- (c);
\draw[->] (b)--(d);
\draw[->] (c)--(d);
\end{tikzpicture} 
\qquad \begin{tikzpicture}
\node (a) at (0,0) {$\infty-1$};
\node (b) at (1,1) {6};
\node (c) at (1,-1) {7};
\node (d) at (2,0) {4};
\draw[->] (a) -- (b);
\draw[->] (a) -- (c);
\draw[->] (b)--(d);
\draw[->] (c)--(d);
\end{tikzpicture}
\qquad 
\begin{tikzpicture}
\node (a) at (0,0) {$\infty-1$};
\node (b) at (1,1) {$\infty-3$};
\node (c) at (1,-1) {7};
\node (d) at (2,0) {4};
\draw[->] (a) -- (b);
\draw[->] (a) -- (c);
\draw[->] (b)--(d);
\draw[->] (c)--(d);
\end{tikzpicture}
$$
To interpret these representation-theoretically,
we define $\Xi'$, $\Xi''$, and $\Xi'''$ by successive reflections. Below, we draw the AR 
quiver of $\rep Q$ (thought of as a full subcategory of $\mathcal R_Q$), but we label the vertices by their dimension vectors in, successively,
$\rep \Xi$, $\rep \Xi'$, $\rep \Xi''$, and $\rep \Xi'''$.
$$\begin{tikzpicture}
\node (a) at (0,0) {100};
\node (b) at (1,-1) {110};
\node (c) at (2,-2) {111};
\node (d) at (2,0) {010};
\node (e) at (3,-1) {011};
\node (f) at (4,0) {001};
\draw[->] (a)--(b);
\draw[->] (b)--(c);
\draw[->] (b)--(d);
\draw[->] (d)--(e);
\draw[->] (c)--(e);
\draw[->] (e)--(f);
\end{tikzpicture}
\qquad
\begin{tikzpicture}
\node (a) at (0,0) {100};
\node (b) at (1,-1) {111};
\node (c) at (2,-2) {110};
\node (d) at (2,0) {011};
\node (e) at (3,-1) {010};
\node (f) at (4,0) {$001[1]$};
\draw[->] (a)--(b);
\draw[->] (b)--(c);
\draw[->] (b)--(d);
\draw[->] (d)--(e);
\draw[->] (c)--(e);
\draw[->] (e)--(f);
\end{tikzpicture}$$
$$\begin{tikzpicture}
\node (a) at (0,0) {110};
\node (b) at (1,-1) {111};
\node (c) at (2,-2) {100};
\node (d) at (2,0) {001};
\node (e) at (3,-1) {$010[1]$};
\node (f) at (4,0) {$011[1]$};
\draw[->] (a)--(b);
\draw[->] (b)--(c);
\draw[->] (b)--(d);
\draw[->] (d)--(e);
\draw[->] (c)--(e);
\draw[->] (e)--(f);
\end{tikzpicture}
\qquad
\begin{tikzpicture}
\node (a) at (0,0) {010};
\node (b) at (1,-1) {011};
\node (c) at (2,-2) {$100[1]$};
\node (d) at (2,0) {001};
\node (e) at (3,-1) {$110[1]$};
\node (f) at (4,0) {$111[1]$};
\draw[->] (a)--(b);
\draw[->] (b)--(c);
\draw[->] (b)--(d);
\draw[->] (d)--(e);
\draw[->] (c)--(e);
\draw[->] (e)--(f);
\end{tikzpicture}
$$

Since $\cat$ is contained in
$\rep \Xi'$, to obtain the second reverse plane partition directly, we simply calculate the Jordan form of a generic nilpotent endomorphism of $X$ when viewed as a representation in 
$\rep \Xi'$. Explicitly, we calculate the Jordan form of a generic nilpotent endomorphism of
$X_{\Xi'}=111^4 \oplus 011^3 \oplus 110^2 \oplus 010$.

For the third reverse plane partition, we must consider $X_{\Xi''}=111^4 \oplus 001^3 \oplus 100^2 \oplus 010[1]$. We separately calculate
the Jordan form of generic nilpotent endomorphisms of 
$X_{\Xi''}^\even=111^4 \oplus 001^3 \oplus 100^2$
and $X_{\Xi''}^\odd=010$. The result obtained for the odd part (1, in row 2) is entered into the reverse plane partition as $\infty -1$. 

Finally, for the fourth reverse plane partition, we have 
$X_{\Xi'''}^\even=011^4 \oplus 001^3$, while 
$X_{\Xi'''}^\odd=100^2 \oplus 110$. We separately calculate the Jordan form of their generic nilpotent endomorphisms and enter them into the reverse plane partition, with the entries (3,1) for $X_{\Xi'''}^\odd$ entered as 
$\infty -3 $ and $\infty-1$.
\end{example}

As explained in the introduction, we can successively apply the previous theorem at vertex 1, vertex 2, \ldots, vertex $n$.  The effect of the successive toggles is just, by definition, promotion, i.e.,  
$$t_n\dots t_1 \rho^\Xi_{Q,m}(X)=
\pro_Q\rho^\Xi_{Q,m}(X).$$
On the representation-theoretic side, the successive changes to the quiver amount to replacing $\rep \Xi$ by $\DR_n\dots\DR_1\rep\Xi$, and by Lemma \ref{lemma_cox_equals_tau}, 
$\DR_n\dots\DR_1\rep\Xi=\tau \rep \Xi$.  
We conclude (Corollary \ref{pro_-1}) that 
$\rho_{Q,m}^{\tau\Xi}(X) = \pro_Q\rho_{Q,m}^{\Xi}(X)$, where we write $\rho_{Q,m}^{\tau\Xi}$ for the reverse plane partitions associated to the decomposition of $\mathcal R_Q$ as 
$\tau \rep \Xi \oplus \tau \rep \Xi[1]$.   

We now prove Theorem \ref{thm_proh_is_Id} from the introduction, which asserts that the order of $\pro_Q$ on reverse plane partitions for $\pos$ is $h$.

\begin{proof}[Proof of Theorem \ref{thm_proh_is_Id}]
By Theorem~\ref{tog-ref}, for $X\in \cat$, we have $\pro\rho^\Xi_{Q,m}(X)=\rho^{\tau\Xi}_{Q,m}(X)$. Lemma \ref{order-cox} tells us that for $M$ in $\mathcal R_Q$, we have $\tau^{-h}(M)\simeq M$.  It follows 
that $\pro_Q^h \rho_{Q,m}(X)=\rho_{Q,m}(X)$.  In order to see that the order cannot be less than $h$, consider for example $X=S_m$.  In order for $\pro_Q^k \rho_{Q,m}(X)=\rho_{Q,m}(X)$, we see that $X$
must be the simple at vertex $m$ in $\tau^{-k}\rep Q$, or in other words that $\tau^k S_m$ must be isomorphic to $S_m$.  Again invoking Lemma \ref{order-cox}, we see that $k$ must be a multiple of $h$.  This proves the theorem.  \end{proof}

\subsection{Proof of Theorem~\ref{intro-period-N}}\label{sec:periodicity_results}
In this section, we establish Theorem \ref{intro-period-N}, which shows that promotion also has period dividing $h$ on reverse plane partitions with entries in $[0,N]$.

Let $\rho$ be a reverse plane partition of $\pos$ with entries in $\NN$. We say that $N\in\mathbb{N}$ is \textit{close enough to infinity} for $\rho$ if replacing all instances of $\infty$ with $N$ in $\rho$ yields a reverse plane partition of $\pos$ with entries in $[0,N]$. In this case, denote the resulting reverse plane partition by $\rho|_{\infty\to N}$.

\begin{lemma}\label{lem:inftytoN}
If $N$ is close enough to infinity for $\rho$, then $N$ is also close enough to infinity for $t_{\textsf{x}} \rho$, and $(t_{\textsf{x}} \rho)|_{\infty \to N} = t_{\textsf{x}}(\rho|_{\infty\to N})$.
\end{lemma}
\begin{proof}
To prove the first claim, note that \[\rho(\textsf{x}),\min_{\textsf{y} \lessdot \textsf{x}}\rho(\textsf{y}),\max_{\textsf{x} \lessdot \textsf{y}}\rho(\textsf{y})\in\{\infty-m\mid m\in\mathbb{N},m\leq N\}\cup\{m \mid m\in\mathbb{N},m\leq N\}.\] Since $\min_{\textsf{y} \lessdot \textsf{x}}\rho(\textsf{y})\geq \rho(\textsf{x}) \geq \max_{\textsf{x} \lessdot \textsf{y}}\rho(\textsf{y})$, it follows that $(t_\textsf{x}\rho)(\textsf{x})\in\{\infty-m\mid m\in\mathbb{N},m\leq N\}\cup\{m \mid m\in\mathbb{N},m\leq N\}$ and $N$ is close enough to infinity for $t_\textsf{x}\rho$.

One can easily prove the second claim by examining three cases that are determined by whether or not each term in $(t_\textsf{x}\rho)(\textsf{x})$ contains $\infty$. For example, suppose \[\rho(\textsf{x}),\min_{\textsf{y} \lessdot \textsf{x}}\rho(\textsf{y}),\max_{\textsf{x} \lessdot \textsf{y}}\rho(\textsf{y})\in\{\infty-m\mid m\leq N\}.\] Let $\min_{\textsf{y} \lessdot \textsf{x}}\rho(\textsf{y})=\infty-m_1$ $\rho(\textsf{x})=\infty-m_2$, and $\max_{\textsf{x} \lessdot \textsf{y}}\rho(\textsf{y})=\infty-m_3$, with $m_1\leq m_2\leq m_3$. Then $(t_x\rho)|_{\infty\to N}(\textsf{x})=(N-m_1)+(N-m_3)-(N-m_2)=N-(m_1+m_2-m_3),$ which clearly is the same as $(t_\textsf{x}(\rho|_{\infty\to N}))(\textsf{x})$.
\end{proof}

\begin{proof}[Proof of Theorem~\ref{intro-period-N}]
Now suppose $\rho$ is a reverse plane partition with entries in $[0,N]$. We can then view it as a reverse plane partition with entries in $\NN$, for which $N$ is close enough to infinity. By Theorem~\ref{thm_proh_is_Id}, we know that $\pro_Q^h$ is the identity for $\rho$ as a reverse plane partition with entries in $\NN$. By Lemma~\ref{lem:inftytoN}, the toggles carried out in computing $\pro_Q^h(\rho)$ as a reverse plane partition with entries in $\NN$ can all be replaced with toggles for reverse plane partitions with entries in $[0,N]$. We conclude that promotion has order dividing $h$ on reverse plane partitions with entries in $[0,N]$, as desired.
\end{proof}

\begin{remark}\label{rem_prev_periodicity_results}
{There is a different approach to proving Theorem \ref{intro-period-N}, based on existing results in the literature, which we now sketch. In \cite{grinberg2015iterative}, Grinberg
and Roby study an operation called \textit{birational rowmotion} on maps from a poset 
to rational functions.  They show that, if $\mathsf P$ is a minuscule poset of type $A_n$, then birational rowmotion has order $n+1$ (which equals the Coxeter number in this case). It is well understood that one can pass via tropicalization from a periodicity statement about 
birational rowmotion to the corresponding statement about piecewise-linear rowmotion.}

{Piecewise-linear rowmotion is an operation very similar to the (piecewise-linear) 
promotion that we have been studying.  In particular, piecewise-linear rowmotion is also defined as 
a composition of toggles in which each element of the poset is toggled once; the 
difference is in the order in which the toggles are carried out. 

It follows from work of Striker and Williams that piecewise-linear rowmotion and promotion correspond to conjugate  elements in the \textit{toggle group} \cite[Lemmas 5.1 and 5.2]{striker2012promotion}; the statements of these results are in more restrictive generality, but the proofs apply in our setting.  The Striker--Williams
result implies that the order of piecewise-linear promotion and piecewise-linear rowmotion are the same.  This implies Theorem \ref{intro-period-N} for minuscule posets of type $A_n$. 
%
%
It is possible to deduce the result for the minuscule poset of type $D_n$ and the choice of one of the two antennae as minuscule vertex via an unfolding argument to type $A_{2n-3}$.  The remaining cases can presumably be dealt with by directly analyzing the behaviour of birational promotion, though we have not carried this out.}

 We also remark that in type $A_n$ with minuscule node $m$, the periodicity of $\pro_Q$ follows from the tropical version of $A_{m-1} \times A_{n-m}$ Zamolodchikov periodicity. See \cite[Section 4.4]{Ro} for more on this. Periodicity of $\pro_Q$ acting on reverse plane partitions for other minuscule posets does not seem to be related to Zamolodchikov periodicity.									
\end{remark}

\section{Type $A$}

For this section, we fix a type $A_n$ quiver $Q$ and a vertex $m$ (which is necessarily minuscule).
\subsection{Posets}
In type $A$, for $M\in\mathcal C_{Q,m}$, we can give an explicit combinatorial description of the entries of $\RSK$ in $\pi^{-1}(i)$ in terms of certain combinatorial invariants of a poset $\tpos$. To construct $\tpos$ from $\pos$, first form the full subposet of $\pos$ whose elements are indecomposable summands of $M$ that are supported at vertex $i$. Then, replace each element in this subposet with a chain whose length is the multiplicity of the corresponding indecomposable summand of $M$ as shown in Figure~\ref{fig:posetconstruction}. We define $\tpos$ to be the resulting poset.

\begin{figure}
\begin{center}
\raisebox{.3in}{
\begin{tikzpicture}
\node (a) at (0,0) {};
\node (a') at (1,0) {$\cdots$};
\node (a'') at (2,0) {};
\node (v) at (1, -1) {$r$};
\node (b) at (0,-2) {};
\node (b') at (1, -2) {$\cdots$};
\node (b'') at (2,-2) {};
\draw[->] (v)--(a);
\draw[->] (v)--(a'');
\draw[->] (b) --(v);
\draw[->] (b'')--(v);
\end{tikzpicture}}\hspace{.5in}\raisebox{.7in}{$\longrightarrow$}\hspace{.5in}
\begin{tikzpicture}
\node (a) at (0,0) {};
\node (a') at (1,0) {$\cdots$};
\node (a'') at (2,0) {};
\node (v) at (1, -1) {$r$};
\node (v') at (1, -2) {$\vdots$};
\node (v'') at (1, -3) {$r$};
\node (b) at (0,-4) {};
\node (b') at (1, -4) {$\cdots$};
\node (b'') at (2,-4) {};
\draw[->] (v)--(a);
\draw[->] (v)--(a'');
\draw[->] (b) --(v'');
\draw[->] (b'')--(v'');
\draw[->] (v'')--(v');
\draw[->] (v')--(v);
\end{tikzpicture}
\end{center}
\caption{In constructing $\tpos$, we replace elements by a chain whose length is the multiplicity of the corresponding representation in $M$.}
\label{fig:posetconstruction}
\end{figure}
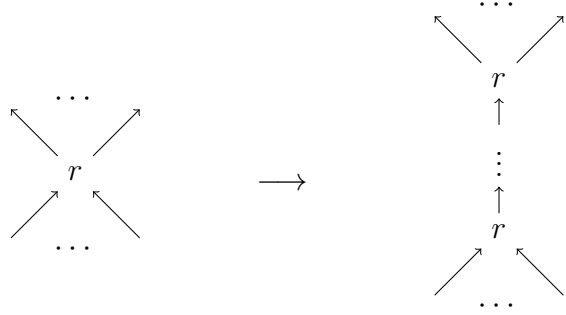

For any poset $\textsf{P}$, let $\Delta_k(\textsf{P})$ denote the largest cardinality of a subset of the elements of $\textsf{P}$ that may be partitioned into $k$ disjoint chains. The numbers $\Delta_k(P)$ are (a trivial re-encoding of) the Greene--Kleitman invariants of $P$ \cite{greene1976structure}. By convention, $\Delta_0(\textsf{P})=0$.

\begin{proposition}\label{prop:GKinv}
Let $M$ be a representation belonging to $\cat$. Let $\lambda^i=(\lambda_1^i,\ldots,\lambda_k^i)$ denote the partition on $\tau$-orbit $i$ of $\RSK$.
 Then 
\begin{equation}\lambda_j^i=\Delta_j(\tpos)-\Delta_{j-1}(\tpos).\label{delta-eq}\end{equation} 
\end{proposition}
\begin{proof} Let $I(\tpos)$ denote the incidence algebra of $\tpos$: the set of $|\tpos|\times|\tpos|$ complex matrices $E$ such that $E_{ij}\neq 0$  implies $i\leq j$ in $\tpos$. Define a strongly generic nilpotent element of $I(\tpos)$ to be a matrix in $I(\tpos)$ with the property that the entries $E_{ij}$ for $i<j$ are independent transcendentals over $\mathbb{Q}$ and the entries $E_{ii}$ are zero. A result of Gansner \cite[Theorem 2.1]{gansner1981acyclic}, 
which is also explained in \cite{britz1999finite}, shows that if $X$ is a strongly generic nilpotent matrix in $I(\tpos)$, then the sizes of the Jordan blocks of $X$ satisfy (\ref{delta-eq}).

By choosing a basis for each $\Hom(V,W)$ for $V$ and $W$ indecomposable representations in $\mathcal C_{Q,m}$, we can interpret Gansner's matrix $X$ as defining a nilpotent endomorphism of $M$.  The image of the totally generic matrices under change of basis is dense in the endomorphisms of $M$.  Thus, given Gansner's result that they all have the same Jordan form given by (\ref{delta-eq}), this must be the Jordan form of a generic nilpotent endomorphism in our sense as well.  
\end{proof}

\begin{remark}
There is no obvious analogue of Theorem~\ref{prop:GKinv} in other types. This stems from the fact that some of the indecomposable representations have vector spaces of rank greater than one at one or more vertices of the quiver. 
\end{remark}

\subsection{Type $A$ examples}
Suppose we start with the type $A$ quiver $Q$ shown in Figure~\ref{fig:HGexample}, where 3 is the chosen minuscule vertex. The corresponding AR quiver is also shown in the figure. The type $A$ minuscule poset $\pos$ associated with vertex 3 is shown in black in the figure. 

We denote a representation $M$ as a labeling of the poset $\pos$, where the label at a vertex denotes how many copies of the corresponding indecomposable are in $M$. The representation $M$ in Figure~\ref{fig:HGexample} contains 4 copies of the indecomposable with dimension vector 11100 and 3 copies of the indecomposable with dimension vector 01100. 

The order we use to compute $\RSK$ is indicated in the subscripts on the AR quiver. The procedure is shown step by step in the 12 fillings in Figure~\ref{fig:HGexample}. Note that there are 15 stages involved in the procedure---one for each indecomposable module---but the first three do not change the resulting reverse plane partition and so are omitted from the figure. In future sections, it will be useful to realize the resulting reverse plane partition as a reverse plane partition for the Young diagram of shape $(3,3,3)$, as shown.

Figure~\ref{fig:Pakexample} shows another example using a different orientation $Q$ and a representation that assigns the same multiplicities to the representations with the same dimension vectors as above. In this figure, we show how to carry out the algorithm by identifying the intermediate fillings with fillings of a Young diagram. There are again 15 stages involved in the procedure, but stages 1 through 4, 6, and 8 do not change the resulting reverse plane partition and so are omitted from the figure.

A reader who is familiar with the known bijections between multisets of rim hooks of a Young diagram and reverse plane partitions of the same Young diagram can check that the example in Figure~\ref{fig:HGexample} agrees with the well-known Hillman--Grassl correspondence and the example in Figure~\ref{fig:Pakexample} agrees with a generalization of the RSK correspondence first described by Pak \cite{pak2001hook} and Berenstein--Kirillov \cite{kirillov1995groups} and later by Sulzgruber \cite{sulzgruberfull}  and Hopkins \cite{hopkins2014rsk}. In fact, this is not a coincidence, and we make the correspondence precise in Sections~\ref{sec:HG} and ~\ref{sec:Pak}.

\begin{figure}
\[\begin{tikzpicture}[scale=1]
\node (11100) at (1,0) {$00100_{6}$};
\node (00010) at (3,0) {\textcolor{gray}{$00010_3$}};
\node (00001) at (5,0) {\textcolor{gray}{$00001_1$}}; 
\node (01100) at (0,1) {$01100_{9}$};
\node (11110) at (2,1) {$00110_5$};
\node (00011) at (4,1) {\textcolor{gray}{$00011_2$}};
\node (01110) at (1,2) {$01110_{8}$};
\node (11111) at (3,2) {$00111_4$};
\node (01111) at (2,3) {$01111_7$};
\node (00100) at (-1,2) {$11100_{13}$};
\node (00110) at (0,3) {$11110_{12}$};
\node (00111) at (1,4) {$11111_{11}$};
\node (10000) at (-3,0) {\textcolor{gray}{$10000_{15}$}};
\node (01000) at (-1,0) {\textcolor{gray}{$01000_{10}$}};
\node (11000) at (-2,1) {\textcolor{gray}{$11000_{14}$}};
\draw[->] (00100)--(01100);
\draw[->] (01100)--(11100);
\draw[->] (00100)--(00110);
\draw[->] (00110)--(00111);
\draw[->] (00110)--(01110);
\draw[->] (01100)--(01110);
\draw[->] (11100)--(11110);
\draw[->] (11110)--(11111);
\draw[->] (10000)--(11000);
\draw[->] (01110)--(11110);
\draw[->] (11110)--(00010);
\draw[->] (00111)--(01111);
\draw[->] (01111)--(11111);
\draw[->] (11111)--(00011);
\draw[->] (00011)--(00001);
\draw[->] (00010)--(00011);
\draw[->] (01110)--(01111);
\draw[->] (11000)--(01000);
\draw[->] (11000)--(00100);
\draw[->] (01000)--(01100);
\end{tikzpicture}
\ytableausetup{boxsize=.2in}
\raisebox{1in}{$M=$
\raisebox{-.65in}{\begin{tikzpicture}[scale=.7]
\node (11100) at (1,0) {0};
\node (01100) at (0,1) {3};
\node (11110) at (2,1) {1};
\node (01110) at (1,2) {1};
\node (11111) at (3,2) {1};
\node (01111) at (2,3) {0};
\node (00100) at (-1,2) {4};
\node (00110) at (0,3) {0};
\node (00111) at (1,4) {2};
\draw[->] (00100)--(01100);
\draw[->] (01100)--(11100);
\draw[->] (00100)--(00110);
\draw[->] (00110)--(00111);
\draw[->] (00110)--(01110);
\draw[->] (01100)--(01110);
\draw[->] (11100)--(11110);
\draw[->] (11110)--(11111);
\draw[->] (01110)--(11110);
\draw[->] (00111)--(01111);
\draw[->] (01111)--(11111);
\draw[->] (01110)--(01111);
\end{tikzpicture}}}\]

\[\begin{tikzpicture}[scale=.7]
\node (11100) at (1,0) {\textcolor{lightgray}{0}};
\node (01100) at (0,1) {\textcolor{lightgray}{0}};
\node (11110) at (2,1) {\textcolor{lightgray}{0}};
\node (01110) at (1,2) {\textcolor{lightgray}{0}};
\node (11111) at (3,2) {1};
\node (01111) at (2,3) {\textcolor{lightgray}{0}};
\node (00100) at (-1,2) {\textcolor{lightgray}{0}};
\node (00110) at (0,3) {\textcolor{lightgray}{0}};
\node (00111) at (1,4) {\textcolor{lightgray}{0}};
\draw[->] (00100)--(01100);
\draw[->] (01100)--(11100);
\draw[->] (00100)--(00110);
\draw[->] (00110)--(00111);
\draw[->] (00110)--(01110);
\draw[->] (01100)--(01110);
\draw[->] (11100)--(11110);
\draw[->] (11110)--(11111);
\draw[->] (01110)--(11110);
\draw[->] (00111)--(01111);
\draw[->] (01111)--(11111);
\draw[->] (01110)--(01111);
\end{tikzpicture}
\begin{tikzpicture}[scale=.7]
\node (11100) at (1,0) {\textcolor{lightgray}{0}};
\node (01100) at (0,1) {\textcolor{lightgray}{0}};
\node (11110) at (2,1) {2};
\node (01110) at (1,2) {\textcolor{lightgray}{0}};
\node (11111) at (3,2) {1};
\node (01111) at (2,3) {\textcolor{lightgray}{0}};
\node (00100) at (-1,2) {\textcolor{lightgray}{0}};
\node (00110) at (0,3) {\textcolor{lightgray}{0}};
\node (00111) at (1,4) {\textcolor{lightgray}{0}};
\draw[->] (00100)--(01100);
\draw[->] (01100)--(11100);
\draw[->] (00100)--(00110);
\draw[->] (00110)--(00111);
\draw[->] (00110)--(01110);
\draw[->] (01100)--(01110);
\draw[->] (11100)--(11110);
\draw[->] (11110)--(11111);
\draw[->] (01110)--(11110);
\draw[->] (00111)--(01111);
\draw[->] (01111)--(11111);
\draw[->] (01110)--(01111);
\end{tikzpicture}
\begin{tikzpicture}[scale=.7]
\node (11100) at (1,0) {2};
\node (01100) at (0,1) {\textcolor{lightgray}{0}};
\node (11110) at (2,1) {2};
\node (01110) at (1,2) {\textcolor{lightgray}{0}};
\node (11111) at (3,2) {1};
\node (01111) at (2,3) {\textcolor{lightgray}{0}};
\node (00100) at (-1,2) {\textcolor{lightgray}{0}};
\node (00110) at (0,3) {\textcolor{lightgray}{0}};
\node (00111) at (1,4) {\textcolor{lightgray}{0}};
\draw[->] (00100)--(01100);
\draw[->] (01100)--(11100);
\draw[->] (00100)--(00110);
\draw[->] (00110)--(00111);
\draw[->] (00110)--(01110);
\draw[->] (01100)--(01110);
\draw[->] (11100)--(11110);
\draw[->] (11110)--(11111);
\draw[->] (01110)--(11110);
\draw[->] (00111)--(01111);
\draw[->] (01111)--(11111);
\draw[->] (01110)--(01111);
\end{tikzpicture}
\begin{tikzpicture}[scale=.7]
\node (11100) at (1,0) {2};
\node (01100) at (0,1) {\textcolor{lightgray}{0}};
\node (11110) at (2,1) {2};
\node (01110) at (1,2) {\textcolor{lightgray}{0}};
\node (11111) at (3,2) {1};
\node (01111) at (2,3) {1};
\node (00100) at (-1,2) {\textcolor{lightgray}{0}};
\node (00110) at (0,3) {\textcolor{lightgray}{0}};
\node (00111) at (1,4) {\textcolor{lightgray}{0}};
\draw[->] (00100)--(01100);
\draw[->] (01100)--(11100);
\draw[->] (00100)--(00110);
\draw[->] (00110)--(00111);
\draw[->] (00110)--(01110);
\draw[->] (01100)--(01110);
\draw[->] (11100)--(11110);
\draw[->] (11110)--(11111);
\draw[->] (01110)--(11110);
\draw[->] (00111)--(01111);
\draw[->] (01111)--(11111);
\draw[->] (01110)--(01111);
\end{tikzpicture}
\begin{tikzpicture}[scale=.7]
\node (11100) at (1,0) {2};
\node (01100) at (0,1) {\textcolor{lightgray}{0}};
\node (11110) at (2,1) {2};
\node (01110) at (1,2) {3};
\node (11111) at (3,2) {0};
\node (01111) at (2,3) {1};
\node (00100) at (-1,2) {\textcolor{lightgray}{0}};
\node (00110) at (0,3) {\textcolor{lightgray}{0}};
\node (00111) at (1,4) {\textcolor{lightgray}{0}};
\draw[->] (00100)--(01100);
\draw[->] (01100)--(11100);
\draw[->] (00100)--(00110);
\draw[->] (00110)--(00111);
\draw[->] (00110)--(01110);
\draw[->] (01100)--(01110);
\draw[->] (11100)--(11110);
\draw[->] (11110)--(11111);
\draw[->] (01110)--(11110);
\draw[->] (00111)--(01111);
\draw[->] (01111)--(11111);
\draw[->] (01110)--(01111);
\end{tikzpicture}\]
\[\begin{tikzpicture}[scale=.7]
\node (11100) at (1,0) {2};
\node (01100) at (0,1) {6};
\node (11110) at (2,1) {0};
\node (01110) at (1,2) {3};
\node (11111) at (3,2) {0};
\node (01111) at (2,3) {1};
\node (00100) at (-1,2) {\textcolor{lightgray}{0}};
\node (00110) at (0,3) {\textcolor{lightgray}{0}};
\node (00111) at (1,4) {\textcolor{lightgray}{0}};
\draw[->] (00100)--(01100);
\draw[->] (01100)--(11100);
\draw[->] (00100)--(00110);
\draw[->] (00110)--(00111);
\draw[->] (00110)--(01110);
\draw[->] (01100)--(01110);
\draw[->] (11100)--(11110);
\draw[->] (11110)--(11111);
\draw[->] (01110)--(11110);
\draw[->] (00111)--(01111);
\draw[->] (01111)--(11111);
\draw[->] (01110)--(01111);
\end{tikzpicture}
\begin{tikzpicture}[scale=.7]
\node (11100) at (1,0) {4};
\node (01100) at (0,1) {6};
\node (11110) at (2,1) {0};
\node (01110) at (1,2) {3};
\node (11111) at (3,2) {0};
\node (01111) at (2,3) {1};
\node (00100) at (-1,2) {\textcolor{lightgray}{0}};
\node (00110) at (0,3) {\textcolor{lightgray}{0}};
\node (00111) at (1,4) {\textcolor{lightgray}{0}};
\draw[->] (00100)--(01100);
\draw[->] (01100)--(11100);
\draw[->] (00100)--(00110);
\draw[->] (00110)--(00111);
\draw[->] (00110)--(01110);
\draw[->] (01100)--(01110);
\draw[->] (11100)--(11110);
\draw[->] (11110)--(11111);
\draw[->] (01110)--(11110);
\draw[->] (00111)--(01111);
\draw[->] (01111)--(11111);
\draw[->] (01110)--(01111);
\end{tikzpicture}
\begin{tikzpicture}[scale=.7]
\node (11100) at (1,0) {4};
\node (01100) at (0,1) {6};
\node (11110) at (2,1) {0};
\node (01110) at (1,2) {3};
\node (11111) at (3,2) {0};
\node (01111) at (2,3) {1};
\node (00100) at (-1,2) {\textcolor{lightgray}{0}};
\node (00110) at (0,3) {\textcolor{lightgray}{0}};
\node (00111) at (1,4) {3};
\draw[->] (00100)--(01100);
\draw[->] (01100)--(11100);
\draw[->] (00100)--(00110);
\draw[->] (00110)--(00111);
\draw[->] (00110)--(01110);
\draw[->] (01100)--(01110);
\draw[->] (11100)--(11110);
\draw[->] (11110)--(11111);
\draw[->] (01110)--(11110);
\draw[->] (00111)--(01111);
\draw[->] (01111)--(11111);
\draw[->] (01110)--(01111);
\end{tikzpicture}
\begin{tikzpicture}[scale=.7]
\node (11100) at (1,0) {4};
\node (01100) at (0,1) {6};
\node (11110) at (2,1) {0};
\node (01110) at (1,2) {3};
\node (11111) at (3,2) {0};
\node (01111) at (2,3) {2};
\node (00100) at (-1,2) {\textcolor{lightgray}{0}};
\node (00110) at (0,3) {3};
\node (00111) at (1,4) {3};
\draw[->] (00100)--(01100);
\draw[->] (01100)--(11100);
\draw[->] (00100)--(00110);
\draw[->] (00110)--(00111);
\draw[->] (00110)--(01110);
\draw[->] (01100)--(01110);
\draw[->] (11100)--(11110);
\draw[->] (11110)--(11111);
\draw[->] (01110)--(11110);
\draw[->] (00111)--(01111);
\draw[->] (01111)--(11111);
\draw[->] (01110)--(01111);
\end{tikzpicture}
\begin{tikzpicture}[scale=.7]
\node (11100) at (1,0) {4};
\node (01100) at (0,1) {6};
\node (11110) at (2,1) {0};
\node (01110) at (1,2) {2};
\node (11111) at (3,2) {0};
\node (01111) at (2,3) {2};
\node (00100) at (-1,2) {10};
\node (00110) at (0,3) {3};
\node (00111) at (1,4) {3};
\draw[->] (00100)--(01100);
\draw[->] (01100)--(11100);
\draw[->] (00100)--(00110);
\draw[->] (00110)--(00111);
\draw[->] (00110)--(01110);
\draw[->] (01100)--(01110);
\draw[->] (11100)--(11110);
\draw[->] (11110)--(11111);
\draw[->] (01110)--(11110);
\draw[->] (00111)--(01111);
\draw[->] (01111)--(11111);
\draw[->] (01110)--(01111);
\end{tikzpicture}\]
\[\begin{tikzpicture}[scale=.7]
\node (11100) at (1,0) {4};
\node (01100) at (0,1) {8};
\node (11110) at (2,1) {2};
\node (01110) at (1,2) {2};
\node (11111) at (3,2) {0};
\node (01111) at (2,3) {2};
\node (00100) at (-1,2) {10};
\node (00110) at (0,3) {3};
\node (00111) at (1,4) {3};
\draw[->] (00100)--(01100);
\draw[->] (01100)--(11100);
\draw[->] (00100)--(00110);
\draw[->] (00110)--(00111);
\draw[->] (00110)--(01110);
\draw[->] (01100)--(01110);
\draw[->] (11100)--(11110);
\draw[->] (11110)--(11111);
\draw[->] (01110)--(11110);
\draw[->] (00111)--(01111);
\draw[->] (01111)--(11111);
\draw[->] (01110)--(01111);
\end{tikzpicture}
\begin{tikzpicture}[scale=.7]
\node (11100) at (1,0) {6};
\node (01100) at (0,1) {8};
\node (11110) at (2,1) {2};
\node (01110) at (1,2) {2};
\node (11111) at (3,2) {0};
\node (01111) at (2,3) {2};
\node (00100) at (-1,2) {10};
\node (00110) at (0,3) {3};
\node (00111) at (1,4) {3};
\draw[->] (00100)--(01100);
\draw[->] (01100)--(11100);
\draw[->] (00100)--(00110);
\draw[->] (00110)--(00111);
\draw[->] (00110)--(01110);
\draw[->] (01100)--(01110);
\draw[->] (11100)--(11110);
\draw[->] (11110)--(11111);
\draw[->] (01110)--(11110);
\draw[->] (00111)--(01111);
\draw[->] (01111)--(11111);
\draw[->] (01110)--(01111);
\end{tikzpicture}
\raisebox{.6in}{$=\RSK$}\hspace{.7in}
\raisebox{.8in}{
\begin{ytableau}
0 & 2 & 3 \\
2 & 2 & 3 \\
6 & 8 & 10
\end{ytableau}}
\]
\caption{The top left image is the AR quiver associated with the quiver $Q=1\leftarrow 2 \leftarrow \mathbf{3}\leftarrow 4 \leftarrow 5$ with chosen minuscule vertex 3. The dimension vectors with support on vertex 3 are black, while the others are gray. The top right image represents a representation $M$ whose indecomposable summands all have support on vertex 3. The images below show the steps in computing $\RSK$, and the bottom right shows the corresponding reverse plane partition of a Young diagram. }
\label{fig:HGexample}
\end{figure}
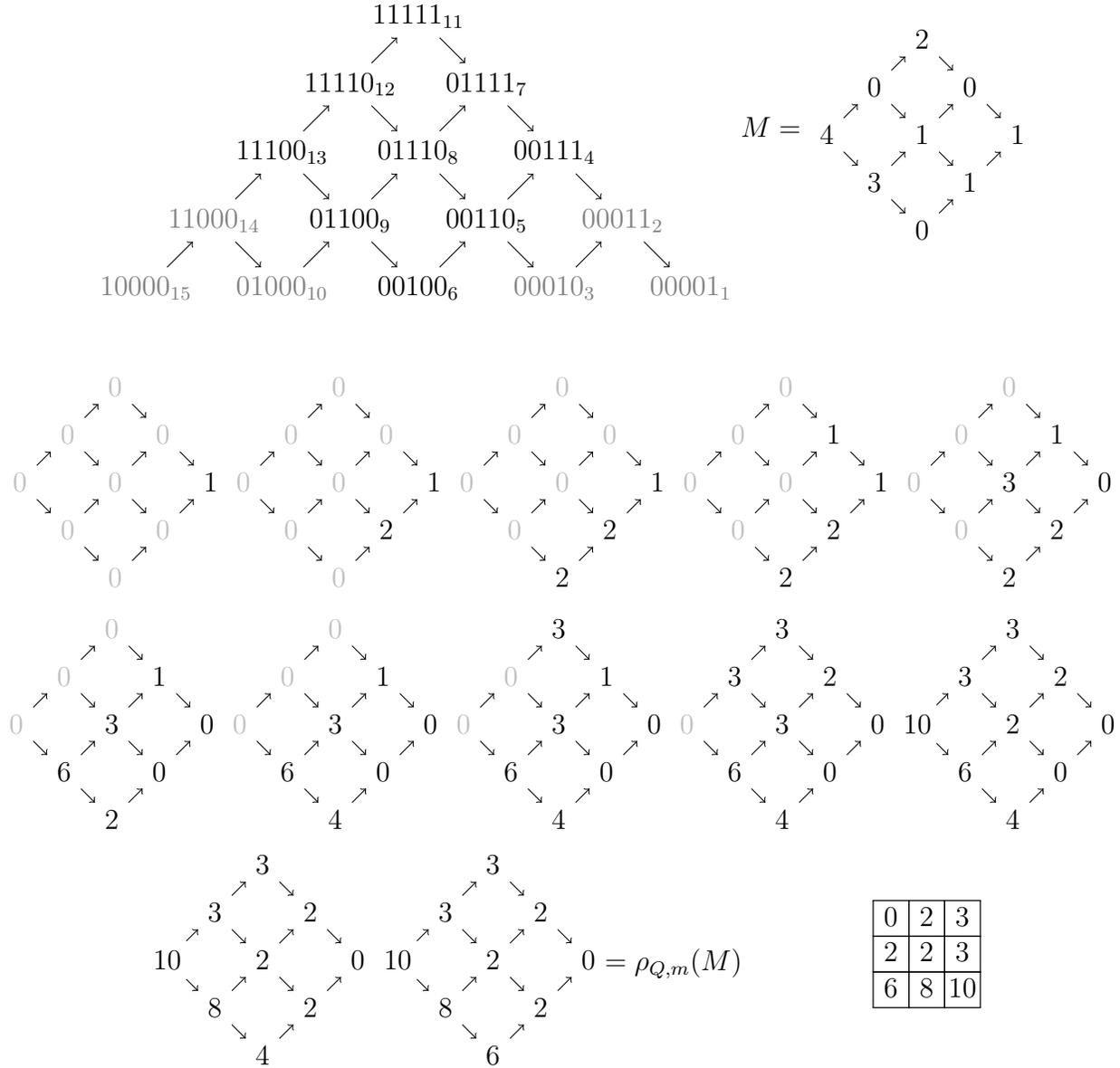

\begin{figure}
\[\begin{tikzpicture}[scale=1]
\node (11100) at (1,0) {$11100_{13}$};
\node (00010) at (3,0) {\textcolor{gray}{$00010_8$}};
\node (00001) at (5,0) {\textcolor{gray}{$00001_3$}}; 
\node (01100) at (0,1) {$01100_{14}$};
\node (11110) at (2,1) {$11110_9$};
\node (00011) at (4,1) {\textcolor{gray}{$00011_4$}};
\node (01110) at (1,2) {$01110_{10}$};
\node (11111) at (3,2) {$11111_5$};
\node (01111) at (2,3) {$01111_7$};
\node (11000) at (4,3) {\textcolor{gray}{$11000_2$}};
\node (01000) at (3,4) {\textcolor{gray}{$01000_6$}};
\node (00100) at (-1,2) {$00100_{15}$};
\node (00110) at (0,3) {$00110_{12}$};
\node (00111) at (1,4) {$00111_{11}$};
\node (10000) at (5,4) {\textcolor{gray}{$10000_1$}};
\draw[->] (00100)--(01100);
\draw[->] (01100)--(11100);
\draw[->] (00100)--(00110);
\draw[->] (00110)--(00111);
\draw[->] (00110)--(01110);
\draw[->] (01100)--(01110);
\draw[->] (11100)--(11110);
\draw[->] (11110)--(11111);
\draw[->] (11111)--(11000);
\draw[->] (11000)--(10000);
\draw[->] (01110)--(11110);
\draw[->] (11110)--(00010);
\draw[->] (00111)--(01111);
\draw[->] (01111)--(11111);
\draw[->] (11111)--(00011);
\draw[->] (00011)--(00001);
\draw[->] (00010)--(00011);
\draw[->] (01110)--(01111);
\draw[->] (01111)--(01000);
\draw[->] (01000)--(11000);
\end{tikzpicture}
\raisebox{1in}{$M=$
\raisebox{-.65in}{\begin{tikzpicture}[scale=.7]
\node (11100) at (1,0) {4};
\node (01100) at (0,1) {3};
\node (11110) at (2,1) {0};
\node (01110) at (1,2) {1};
\node (11111) at (3,2) {1};
\node (01111) at (2,3) {0};
\node (00100) at (-1,2) {0};
\node (00110) at (0,3) {1};
\node (00111) at (1,4) {2};
\draw[->] (00100)--(01100);
\draw[->] (01100)--(11100);
\draw[->] (00100)--(00110);
\draw[->] (00110)--(00111);
\draw[->] (00110)--(01110);
\draw[->] (01100)--(01110);
\draw[->] (11100)--(11110);
\draw[->] (11110)--(11111);
\draw[->] (01110)--(11110);
\draw[->] (00111)--(01111);
\draw[->] (01111)--(11111);
\draw[->] (01110)--(01111);
\end{tikzpicture}}}\]

\begin{eqnarray*}
&\begin{ytableau}
1 &\textcolor{lightgray}{0}&\textcolor{lightgray}{0} \\
\textcolor{lightgray}{0} & \textcolor{lightgray}{0}& \textcolor{lightgray}{0}\\
\textcolor{lightgray}{0} & \textcolor{lightgray}{0}& \textcolor{lightgray}{0}
\end{ytableau}\hspace{.1in}
\begin{ytableau}
1 & 1 & \textcolor{lightgray}{0}\\
\textcolor{lightgray}{0} & \textcolor{lightgray}{0}& \textcolor{lightgray}{0}\\
\textcolor{lightgray}{0} & \textcolor{lightgray}{0}& \textcolor{lightgray}{0}
\end{ytableau}\hspace{.1in}
\begin{ytableau}
1 & 1 & \textcolor{lightgray}{0}\\
1 & \textcolor{lightgray}{0}& \textcolor{lightgray}{0}\\
\textcolor{lightgray}{0} & \textcolor{lightgray}{0}& \textcolor{lightgray}{0}
\end{ytableau}\hspace{.1in}
\begin{ytableau}
0 & 1 & \textcolor{lightgray}{0}\\
1 & 2 & \textcolor{lightgray}{0}\\
\textcolor{lightgray}{0} & \textcolor{lightgray}{0}& \textcolor{lightgray}{0}
\end{ytableau}\hspace{.1in}
\begin{ytableau}
0 & 1 & 3\\
1 & 2 & \textcolor{lightgray}{0}\\
\textcolor{lightgray}{0} & \textcolor{lightgray}{0}& \textcolor{lightgray}{0}
\end{ytableau}\hspace{.1in}
\begin{ytableau}
0 & 1 & 3\\
1 & 2 & 4\\
\textcolor{lightgray}{0} & \textcolor{lightgray}{0}& \textcolor{lightgray}{0}
\end{ytableau}\hspace{.1in}
\begin{ytableau}
0 & 1 & 3\\
1 & 2 & 4\\
5 &\textcolor{lightgray}{0} &\textcolor{lightgray}{0} 
\end{ytableau}\hspace{.1in}
\begin{ytableau}
0 & 1 & 3\\
1 & 2 & 4\\
5 & 8 & \textcolor{lightgray}{0}
\end{ytableau}&\\\hspace{.1in}
&\begin{ytableau}
1 & 1 & 3\\
1 & 3 & 4\\
5 & 8 & 8
\end{ytableau}
\raisebox{-.15in}{ $=\RSK$}&
\end{eqnarray*}
\caption{The top left image is the AR quiver associated with the quiver $Q=1\rightarrow 2\rightarrow \mathbf{3}\leftarrow 4 \leftarrow 5$ with chosen minuscule vertex 3. The dimension vectors with support on vertex 3 are black, while the others are gray. The top right image represents a representation $M$ whose indecomposable summands all have support on vertex 3. The lower images show the steps in computing $\RSK$, shown on the corresponding Young diagram.}
\label{fig:Pakexample}
\end{figure}
\subsection{The Hillman--Grassl Correspondence}\label{sec:HG}
A \textit{rim hook} of the Young diagram $\lambda$ is a connected strip of border boxes in $\lambda$ such that the result of removing these boxes is again a Young diagram. The Hillman--Grassl correspondence \cite{hillman1976reverse} is a bijection between multisets of rim hooks of $\lambda$ and the set of reverse plane partitions of $\lambda$. For detailed explanation, we recommend \cite{stanley1999enumerative}. 

Let $M$ be a representation in $\cat$. Then as shown in Figure~\ref{fig:HGexample}, $\RSK$ can be viewed as a reverse plane partition of the Young diagram of shape $m^{(n+1-m)}$. We may identify the indecomposable summands of $M$ as rim hooks of $m^{(n+1-m)}$ by reading through the southeast border of $m^{(n+1-m)}$ from southwest to northeast and including a box in the rim hook exactly when the corresponding entry in the dimension vector of the indecomposable is 1. See Figure~\ref{fig:hooktodimvector}. In this way, we may identify $M$ with a multiset of rim hooks of $m^{(n+1-m)}$. Given a multiset $M$ of rim hooks of $m^{(n+1-m)}$ (i.e., an $M$ in $\cat$), let $HG(M)$ denote the reverse plane partition of shape $m^{(n+1-m)}$ obtained using the Hillman--Grassl correspondence. 

\begin{figure}
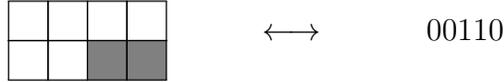

\begin{ytableau}  $ $ & & & \\
$ $ & & *(gray) & *(gray)
\end{ytableau}\hspace{.5in}$\longleftrightarrow$\hspace{.5in} 00110
\caption{A rim hook and the corresponding dimension vector.}
\label{fig:hooktodimvector}
\end{figure}

\begin{theorem}\label{thm:HG} Let $Q$ be a type $A_n$ quiver with chosen minuscule vertex $m$ and the following orientation.
\[Q=1\leftarrow 2 \leftarrow \cdots \leftarrow n-1 \leftarrow n\]
Then for any $M\in \cat$, $\RSK=HG(M)$.
\end{theorem}
\begin{proof}
We consider the vertices of $\pos$ to be a grid and index the vertices by letting its minimal element be $(n+1-m,1)$, its maximal element be $(1,m)$. See Figure~\ref{fig:HGARquiver}. The indecomposable representation in position $(i,j)$ has dimension vector 
\[\underbrace{0\cdots 0}_{j-1} 1\cdots 1 \raisebox{0.9pt}{\textcircled{\raisebox{-0.9pt} {1}}} 1 \cdots 1 \underbrace{0 \cdots 0}_{i-1},\] where the $m^{th}$ entry is circled. See Figure~\ref{fig:HGexample} for an example. 

Fix representation $M$ in $\cat$. Then if $k\leq m$, the indecomposables in $\tilde{\textsf{P}}_{{Q,m}}^k(M)$ are exactly those of $\pos$ that are in positions $(i,j)$ with $j\leq k$ and also are summands of $M$.
If $k>m$, the indecomposables in $\tilde{\textsf{P}}_{{Q,m}}^k(M)$ are exactly those of $\pos$ in positions $(i,j)$ with $i\leq n+1-k$ and also are summands of $M$.

It is straightforward to see that the cardinalities of disjoint unions of chains in these $\tilde{\textsf{P}}_{{Q,m}}^k(M)$ correspond to the sizes of the $A$-chains that determine $HG(M)$ as described in \cite{gansner1981hillman}.
\end{proof}

\begin{figure}
\[\begin{tikzpicture}
\node (1) at (0,0) {$(n+1-m,n)$};
\node (2) at (2,2) {$(1,n)$};
\node (3) at (-3,3) {$(n+1-m,1)$};
\node (4) at (-1,5) {$(1,1)$};
\draw[->] (1)--(2);
\draw[->] (4)--(2);
\draw[->] (3)--(1);
\draw[->] (3)--(4);
\end{tikzpicture}\]
\caption{Indexing of elements of $\pos$.}
\label{fig:HGARquiver}
\end{figure}

\begin{remark}\label{remark:HGrectangle}
{The Hillman--Grassl correspondence for any partition $\lambda$ can be constructed using subrectangles of $\lambda$ that contain the top left corner of $\lambda$. There will be one subrectangle for each diagonal of $\lambda$, where the box corresponding to diagonal $i$ has as its bottom right corner the southeasternmost box on diagonal $i$ and has as its top left corner the top left box of $\lambda$. To perform the Hillman--Grassl algorithm on shape $\lambda$, it suffices to perform it on each of these rectangles. This follows easily from the Greene--Kleitman invariants for Hillman--Grassl (see for example \cite[Theorem 3.3]{gansner1981hillman}). We can thus recover the Hillman--Grassl correspondence from $\rho_{Q,m}(-)$---one map for each subrectangle---for any shape $\lambda$.}
\end{remark}

\subsection{The RSK correspondence}\label{sec:Pak}
 
One can consider the classical Robinson--Schensted--Knuth correspondence \cite{knuth1970permutations,robinson1938representations,schensted} to be a bijection between the set of $n\times n$ matrices with entries in $\mathbb{N}$ and the set of reverse plane partitions of an $n\times n$ square. We recommend \cite[Section 1]{hopkins2014rsk} for a detailed explanation of this. We can consider an $n\times n$ matrix with entries in $\mathbb{N}$ to be a multiset of rim hooks of the $n\times n$ square by equating an entry $k$ in position $(i,j)$ with $k$ copies of the rim hook whose northeasternmost box is in row $i$ and whose southwesternmost box is in column $j$. We can thus view Robinson--Schensted--Knuth as a bijection between multisets of rim hooks of a square and reverse plane partitions of that square.
\newcommand{\mRSK}{\textrm{RSK}}

In \cite{pak2001hook} and \cite{kirillov1995groups}, the authors generalize this bijection to obtain a correspondence between multisets of rim hooks of any partition shape $\lambda$ and reverse plane partitions of $\lambda$ that involves toggles. We denote this correspondence by $\mRSK$ and recommend \cite[Section 2]{hopkins2014rsk} for further details about this bijection.

We argue that our bijection $\rho_{Q,m}$ recovers $\mRSK$ for rectangles. We again use the identification between rim hooks of a rectangle and dimension vectors of indecomposable representations in $\cat$ from Section~\ref{sec:HG}. Let $\mRSK(M)$ denote the reverse plane partition obtained using the multiset of rim hooks determined by $M$.  The following result is illustrated in Figure~\ref{fig:Pakexample}.

\begin{theorem}\label{thm:Pak} Let $Q$ be a type $A_n$ quiver with chosen minuscule vertex $m$ and the following orientation.
\[Q=1\rightarrow \cdots \rightarrow m \leftarrow \cdots \leftarrow n\]
Then for any $M\in \cat$, $\RSK=\mRSK(M)$.
\end{theorem}

\begin{proof}
One checks that $\pos$ is  an order ideal inside of the AR quiver of $Q$, where we think of the AR quiver as a poset. It follows that, using the piecewise-linear description of $\JFg(-)$ given in Section~\ref{sec:explicitRSK}, any nontrivial $\tau$-orbit toggle performed when constructing an intermediate filling $\rho_k$ corresponds to adding a box to the Young diagram. (This is in contrast to our toggling description of Hillman--Grassl in Figure~\ref{fig:HGexample}, where steps 10, 14, and 15 of the algorithm involve non-trivial toggles with no additional boxes being added to the Young diagram.)

Indexing $\pos$ as in the proof of Theorem~\ref{thm:HG}, we have that the dimension vector in position $(i,j)$ is shown below, where the $m^{th}$ entry is circled.
\[0\cdots 0 \underbrace{1\cdots1}_{j-1}\raisebox{.5pt}{\textcircled{\raisebox{-.9pt} {1}}} 1\cdots 1 \underbrace{0\cdots 0}_{i-1}\] See Figure~\ref{fig:Pakexample} for an example. Comparing the algorithm from Section~\ref{sec:explicitRSK} to the description of $\mRSK$ in \cite[Section 2]{hopkins2014rsk}, it is immediately clear that they coincide. This proves the result.\end{proof}

\begin{remark} We could alternatively prove Theorem~\ref{thm:Pak} in a way similar to the proof of Theorem~\ref{thm:HG}. See, for example, \cite[Section 6]{garver2017greene} for the corresponding description of $\mRSK$.
\end{remark}

\begin{remark}
Similarly to Remark~\ref{remark:HGrectangle}, we can recover the RSK correspondence for any shape $\lambda$ using our map $M\mapsto\RSK$ as the RSK correspondence can be constructed using subrectangles of $\lambda$ (see for example~~\cite[Theorem 6.1]{garver2017greene}).
\end{remark}


\begin{remark} 
Fix $n$ and $m\in\{1,\ldots,n\}$. For each type $A_n$ Dynkin quiver, we obtain a bijection $\rho_{Q,m}(-)$ that we can think of as a map between multisets of rim hooks of $m^{(n+1-m)}$ and reverse plane partitions of $m^{(n+1-m)}$. We have shown that for two particular orientations, we recover the Hillman-Grassl and RSK correspondences. For other orientations, we obtain different bijections. 
We demonstrate this with an example.

We begin with the quiver 
\[Q=1\leftarrow 2\rightarrow \mathbf{3}\leftarrow 4\rightarrow 5\] with minuscule vertex 3 and use the following multiset of rim hooks of the 3-by-3 rectangle.
\[01100,01110,11111,11111,00110,00110\]
This multiset of rim hooks determines the representation $M$ of the quiver $Q$ shown below. We then obtain $\RSK$ as shown. We leave it to the reader to check that both Hillman--Grassl and the RSK correspondences with this same multiset of rim hooks yield a reverse plane partition where the maximal element is labeled with 0.

\[\begin{tikzpicture}[scale=1]
\node (00110) at (1,0) {$00110_{9}$};
\node (11000) at (3,0) {\textcolor{gray}{$11000_4$}};
\node (00111) at (0,1) {$00111_{11}$};
\node (11110) at (2,1) {$11110_6$};
\node (01000) at (4,1) {\textcolor{gray}{$01000_1$}};
\node (11111) at (1,2) {$11111_{8}$};
\node (01110) at (3,2) {$01110_3$};
\node (01111) at (2,3) {$01111_7$};
\node (00010) at (4,3) {\textcolor{gray}{$00010_2$}};
\node (00011) at (3,4) {\textcolor{gray}{$00011_5$}};
\node (00100) at (-1,2) {$00100_{13}$};
\node (11100) at (0,3) {$11100_{12}$};
\node (01100) at (1,4) {$01100_{10}$};
\node (00001) at (-1,0) {\textcolor{gray}{$00001_{14}$}};
\node (10000) at (-1,4) {\textcolor{gray}{$10000_{15}$}};
\draw[->] (00001)--(00111);
\draw[->] (00100)--(00111);
\draw[->] (00100)--(11100);
\draw[->] (10000)--(11100);
\draw[->] (00111)--(00110);
\draw[->] (00111)--(11111);
\draw[->] (11100)--(11111);
\draw[->] (11100)--(01100);
\draw[->] (00110)--(11110);
\draw[->] (11111)--(11110);
\draw[->] (11111)--(01111);
\draw[->] (01100)--(01111);
\draw[->] (11110)--(11000);
\draw[->] (11110)--(01110);
\draw[->] (01111)--(01110);
\draw[->] (01111)--(00011);
\draw[->] (00011)--(00010);
\draw[->] (01110)--(00010);
\draw[->] (01110)--(01000);
\draw[->] (11000)--(01000);
\end{tikzpicture}
\raisebox{.85in}{$M=$}
\raisebox{.23in}{
\begin{tikzpicture}[scale=.7]
\node (00110) at (1,0) {$2$};
\node (00111) at (0,1) {$0$};
\node (11110) at (2,1) {$0$};
\node (11111) at (1,2) {$2$};
\node (01110) at (3,2) {$1$};
\node (01111) at (2,3) {$0$};
\node (00100) at (-1,2) {$0$};
\node (11100) at (0,3) {$0$};
\node (01100) at (1,4) {$1$};
\draw[->] (00100)--(00111);
\draw[->] (00100)--(11100);
\draw[->] (00111)--(00110);
\draw[->] (00111)--(11111);
\draw[->] (11100)--(11111);
\draw[->] (11100)--(01100);
\draw[->] (00110)--(11110);
\draw[->] (11111)--(11110);
\draw[->] (11111)--(01111);
\draw[->] (01100)--(01111);
\draw[->] (11110)--(01110);
\draw[->] (01111)--(01110);
\end{tikzpicture}}
\raisebox{.85in}{$\RSK=$}
\raisebox{.23in}{
\begin{tikzpicture}[scale=.7]
\node (00110) at (1,0) {$2$};
\node (00111) at (0,1) {$3$};
\node (11110) at (2,1) {$2$};
\node (11111) at (1,2) {$2$};
\node (01110) at (3,2) {$1$};
\node (01111) at (2,3) {$1$};
\node (00100) at (-1,2) {$3$};
\node (11100) at (0,3) {$3$};
\node (01100) at (1,4) {$2$};
\draw[->] (00100)--(00111);
\draw[->] (00100)--(11100);
\draw[->] (00111)--(00110);
\draw[->] (00111)--(11111);
\draw[->] (11100)--(11111);
\draw[->] (11100)--(01100);
\draw[->] (00110)--(11110);
\draw[->] (11111)--(11110);
\draw[->] (11111)--(01111);
\draw[->] (01100)--(01111);
\draw[->] (11110)--(01110);
\draw[->] (01111)--(01110);
\end{tikzpicture}}\]
\end{remark}

\section*{Acknowledgements}
AG was suppported by NSERC grant RGPIN/05999-2014 and the Canada Research Chairs program. AG thanks Gabe Frieden for useful discussions and helpful comments on an earlier version of the paper. BP was supported by NSERC, CRM-ISM, and the Canada Research Chairs program. 
HT was supported by an NSERC Discovery Grant and the Canada Research Chairs program.  He thanks Arkady Berenstein and Steffen Oppermann for helpful discussions, and Guillaume Chapuy for an inspiring explanation of Robinson--Schensted--Knuth.  He also thanks the Universit\'e Paris VII for the excellent working conditions under which part of the research was carried out. The authors thank Bernhard Keller for a comment on an earlier version of the paper and Robin Sulzgruber for sharing an early version of his paper \cite{sulzgruber2016building}, which was inspirational to the authors.

\bibliographystyle{plain}
\bibliography{main_v2.bib}

\end{document}